\documentclass[a4paper, reqno,11pt]{amsart}

\usepackage{latexsym}
\usepackage{float}
\usepackage{amssymb}
\usepackage{amsmath}
\usepackage{amsthm}
\usepackage{mathrsfs}
\usepackage{euscript}
\usepackage[cmtip,all]{xy}

\usepackage{bbold}
\usepackage{tabu}
\usepackage{enumerate}

\usepackage{framed}

\usepackage{graphicx}
\usepackage{stackrel}

\usepackage{mathtools}

\usepackage{tikz}
\usetikzlibrary{intersections}
\usetikzlibrary{backgrounds}

\usepackage{adjustbox}

\usepackage{pgfplots}

\pgfplotsset{compat=1.10}
\usepgfplotslibrary{fillbetween}
\usetikzlibrary{patterns}

\usetikzlibrary{plotmarks}
\usetikzlibrary{cd}

\makeatletter
\g@addto@macro \normalsize {%
 \setlength\abovedisplayskip{7pt}%
 \setlength\belowdisplayskip{6pt}%
}
\makeatother

\newtheorem{thm}[equation]{Theorem}
\newtheorem{lem}[equation]{Lemma}
\newtheorem{cor}[equation]{Corollary}
\newtheorem{prop}[equation]{Proposition}

\theoremstyle{remark}
\newtheorem{rem}[equation]{Remark}

\newtheorem{claim}[equation]{Claim}

\newtheorem{example}[equation]{Example}
\newtheorem{defn}[equation]{Definition}

\numberwithin{equation}{section}

\newcommand{\gb}{\beta}
\newcommand{\ga}{\alpha}

\newcommand{\gl}{\nu}

\newcommand{\gD}{\Delta}

\newcommand{\eps}{\varepsilon}

\newcommand{\fa}{{\mathfrak a}}             
\newcommand{\fb}{{\mathfrak b}}
\newcommand{\fg}{{\mathfrak g}}
\newcommand{\fh}{{\mathfrak h}}
\newcommand{\fk}{{\mathfrak k}}
\newcommand{\fl}{{\mathfrak l}}
\newcommand{\fm}{{\mathfrak m}}
\newcommand{\fn}{{\mathfrak n}}

\newcommand{\fp}{{\mathfrak p}}

\newcommand{\fy}{{\mathfrak y}}

\newcommand{\f}{\mathfrak}

\newcommand{\R}{\mathbb{R}}          
\newcommand{\C}{\mathbb{C}}          
           
\newcommand{\Z}{\mathbb{Z}}

\newcommand{\ad}{\mathrm{ad}}
\newcommand{\Ad}{\mathrm{Ad}}
\newcommand{\Cal}{\mathcal}

\newcommand{\Hom}{\operatorname{Hom}}

\newcommand{\Ind}{\mathrm{Ind}}
\newcommand{\IP}[2]{\langle#1 , #2\rangle}     

\newcommand{\spn}{\mathrm{span}}

\newcommand{\Tr}{\text{Tr}}

\newcommand{\Sol}{\mathrm{Sol}}

\newcommand{\Pol}{\mathrm{Pol}}
\newcommand{\To}{\longrightarrow}

\newcommand{\Diff}{\mathrm{Diff}}
\newcommand{\Irr}{\mathrm{Irr}}

\newcommand{\sym}{\mathrm{sym}}

\newcommand{\D}{\Cal{D}}

\newcommand{\Symb}{\mathrm{Symb}}

\newcommand{\dpi}{d\pi}

\newcommand{\id}{\mathrm{id}}

\newcommand{\pp}{\textnormal{\mbox{\smaller($+$,$+$)}}}
\newcommand{\pmi}{\textnormal{\mbox{\smaller($+$,$-$)}}}
\newcommand{\mip}{\textnormal{\mbox{\smaller($-$,$+$)}}}
\newcommand{\mm}{\textnormal{\mbox{\smaller($-$,$-$)}}}

\newcommand{\ssv}{{\scriptscriptstyle \vee}}

\newcommand{\Ker}{\mathrm{Ker}}

\newcommand{\const}{\mathrm{const}}

\newcommand{\Trun}{\mathrm{Trun}}

\newcommand{\abs}[1]{\left\vert#1\right\vert}

\newcommand{\Mp}{M^{\fg}_{\fp}}
\newcommand{\Mpp}{M^{\fg'}_{\fp'}}

\newcommand{\EuD}{\EuScript{D}}

\newcommand{\Rest}{\mathrm{Rest}}

\newcommand{\w}{\widetilde}
\newcommand{\bD}{\mathbb{D}}

\newcommand{\acts}{ \, {\raisebox{1pt} {$\scriptstyle \bullet$} } \,}

\newcommand{\hdL}{\widehat{dL}}

\newcommand{\dpil}{d\pi_{\lambda^*}}
\newcommand{\hdpil}{\widehat{d\pi_{\lambda^*}}}

\newcommand{\wSol}{\widetilde{\Sol}}

\newcommand{\Cay}{\mathrm{Cay}}

\newcommand{\Supp}{\mathrm{Supp}}

\providecommand*{\donothing}[1]{}

\makeatletter
\@namedef{subjclassname@2020}{%
  \textup{2020} Mathematics Subject Classification}
\makeatother

\begin{document}

\baselineskip=16pt
\tabulinesep=1.2mm


\title[]{The truncated symbol of a differential symmetry breaking operator}

\author{Toshihisa Kubo}
\author{V{\' i}ctor P{\'e}rez-Vald{\'e}s}

\address{Faculty of Economics, 
Ryukoku University,
67 Tsukamoto-cho, Fukakusa, Fushimi-ku, Kyoto 612-8577, Japan}
\email{toskubo@econ.ryukoku.ac.jp}

\address{Graduate School of Mathematical Sciences, The University of Tokyo, 3-8-1 Komaba, Meguro-ku, Tokyo 153-8914, Japan}
\email{perez-valdes@g.ecc.u-tokyo.ac.jp}
\subjclass[2020]{
22E46, 
17B10} 

\keywords{differential symmetry breaking operator,
differential intertwining operator,
generalized Verma module,
truncated symbol map,
F-method,
factorization identity, 
Cayley continuant,
binary Krawtchouk polynomial,
Jacobi polynomial}

\date{\today}


\begin{abstract} 
In this paper, we introduce the \emph{truncated symbol} of 
a differential symmetry breaking operator 
between parabolically induced representations.
This generalizes the symbol map,
which is defined for the case of abelian nilpotent radicals, to the non-abelian setting.
The inverse of the truncated symbol map
enables one to develop a recipe for the F-method
for any nilpotent radical.

As an application, we classify and construct
differential intertwining operators on the full flag variety 
$SL(3,\R)/B$
and homomorphisms between Verma modules.
It turned out that, surprisingly, Cayley continuants
appeared in the coefficients of one of the five families of operators 
that we constructed. At the end of the paper, 
factorization identities for the differential operators
and homomorphisms are also classified.
Binary Krawtchouk polynomials play a key role
in the proof.
\end{abstract}

\maketitle

\setcounter{tocdepth}{1}
\tableofcontents

\section{Introduction}\label{sec:intro}

Symmetry breaking operators (SBOs) are
$G'$-intertwining operators from a representation of $G$ to that of $G'$
for a pair of Lie groups $G \supseteq G'$.
The SBOs that can be realized as differential operators are called 
differential symmetry breaking operators (DSBOs). 

The classification and construction of DSBOs $\mathbb{D}$ are in general 
very difficult problems. In case that 
$G \supseteq G'$ are a pair of reductive Lie groups and 
that the representations of $G$ and $G'$ are 
both parabolically induced representations, 
one can apply the powerful machinery of the \emph{F-method} to solve these problems.
Loosely speaking, 
via a certain Fourier transform $F_c$,
the F-method converts
the problem of the classification and construction of DSBOs $\mathbb{D}$ 
into that of solving a system of partial differential equations over
polynomials $\psi(\zeta) \in \Pol(\fn_+)$ 
on the nilpotent radical $\fn_+$ of a parabolic subalgebra 
$\fp=\fl \oplus \fn_+$ of the complexified Lie algebra $\fg$ of $G$.
Since its introduction by T.\ Kobayashi around 2010,
the F-method has been intensively used to study DSBOs $\mathbb{D}$
(cf.\ \cite{DiKuPV26+,FJS20,  Kobayashi13, Kobayashi14, KKP16, KOSS15,
KP1, KP2, KrSom17, Kubo26, PV23, PV26a, PV26b, Somberg25}).
 
One of the key ideas
of the F-method in the case that the nilpotent radical 
$\fn_+$ is abelian is the symbol map $\Symb$,
whose inverse $\Symb^{-1}$
effectively transforms the polynomial solutions $\psi(\zeta)$ into DSBOs $\mathbb{D}$
 (\cite[Sect.\ 4.2]{KP1}).
In this paper, we introduce 
a variant of the symbol map, which we call 
the \emph{truncated symbol map} $\Symb_0$, to generalize the symbol map
to the non-abelian case.
As an application to the truncated symbol map $\Symb_0$, 
we take $G=G'=SL(3,\R)$ and  consider 
DSBOs $\mathbb{D}$ between parabolically induced
representations for $(SL(3,\R), B)$, 
where $B$ is a minimal parabolic subgroup of $SL(3,\R)$.
(For the maximal parabolic case, see \cite{Kubo26, KuOr25a, KuOr26}.)

In the following, we shall discuss  the F-method and truncated 
symbol map $\Symb_0$, 
the DSBOs $\mathbb{D}$ for $(SL(3,\R), B)$, and 
a relationship between our results on DSBOs $\mathbb{D}$ and 
the ones on 
the homomorphisms $\varphi$ between Verma modules in the literature in order.

We remark that, in order to distinguish the symmetry breaking setting 
$G \supseteq G'$ from the classical one $G=G'$, 
a DSBO $\mathbb{D}$ in the latter case will be referred to as 
a differential intertwining operator (DIO) $\D$ in the rest of the paper.

\subsection{F-method and truncated symbol map $\Symb_0$}
\label{sec:IntroFmethod}
 
In principle, the F-method can be applied  no matter what 
the nilpotent radical $\fn_+$ is abelian or not; 
nevertheless, there is a certain difficulty for the non-abelian case.
To compute DSBOs $\mathbb{D}$, one needs to evaluate
the inverse image of $F_c^{-1}(\psi(\zeta))$ of the Fourier transform $F_c$.
The problem is that the inverse image $F_c^{-1}(\psi(\zeta))$ is not apparent
unless the nilpotent radical $\fn_+$ is abelian. 
Then, in this paper, 
we shall first resolve this issue.
This is done in Section \ref{sec:Fc}.
A key idea is to observe a well-known 
linear isomorphism between the space $\Diff_G(G)$
of left-invariant differential operators $D$ on a Lie group $G$ and 
the symmetric algebra $S(\fg)$ of the complexified Lie algebra $\fg$ of $G$.
The inverse $F_c^{-1}(\psi(\zeta))$ is then determined by
relating the infinitesimal left translation $dL$
with the right translation $dR$ via the Dirac delta function $\delta_0$.

The elucidation of the inverse image $F_c^{-1}(\psi(\zeta))$ 
leads us to the notion of the aforementioned truncated symbol
$\Symb_0(\mathbb{D})$ of $\mathbb{D}$.
This is the symbol 
of the constant coefficient part $\mathbb{D}_0$ of $\mathbb{D}$ in the usual sense.
For instance, we shall show in Example \ref{example:symb3} that
the differential operator
\begin{align*}
\mathbb{D}
=\frac{\partial^3}{\partial x_1^2 \partial x_2}
+\frac{\partial^2}{\partial x_ 1\partial x_3}
&+\frac{1}{2}x_2\frac{\partial^3}{\partial x_3^3}
+x_2\frac{\partial^3}{\partial x_1\partial x_2 \partial x_3}\\[3pt]
&+\frac{1}{4}x_2^2\frac{\partial^3}{\partial x_2 \partial x_3^2}
-\frac{1}{2}x_1\frac{\partial^3}{\partial x_1^2 \partial x_3}
-\frac{1}{2}x_1x_2\frac{\partial^3}{\partial x_1 \partial x_3^2}
-\frac{1}{8}x_1x_2^2\frac{\partial^3}{\partial x_3^3}
\end{align*}
is a DSBO (or DIO) between parabolically induced representations
for $(SL(3,\R), B)$.
The truncated symbol $\Symb_0(\mathbb{D})$ is then given as
\begin{equation*}
\Symb_0({\mathbb{D}})=\zeta_1^2\zeta_2+\zeta_1\zeta_3.
\end{equation*}
By the intertwining property,
the DSBO $\mathbb{D}$ is determined at the identity $e$,
which in principle allows one to recover $\mathbb{D}$ from $\Symb_0(\mathbb{D})$.
By definition, if the nilpotent radical $\fn_+$ is abelian,
then $\Symb_0(\mathbb{D}) = \Symb(\mathbb{D})$
as $\mathbb{D} = \mathbb{D}_0$ in that case (cf.\ \cite[Sect.\ 2.8]{KuOr25a}).
The truncated symbol map $\Symb_0$ 
in particular allows one to carry out a recipe for the F-method
for any nilpotent radical
as in the abelian setting \cite[Sect.\ 4.4]{KP1}.
We shall discuss the details in Section \ref{sec:symbF}.

Remark that K{\v r}i{\v z}ka--Somberg
also computed $F_c^{-1}(\psi(\zeta))$ in their formulation
in \cite{KrSom17}. Their proof is, however, different from ours.
For instance, their proof requires some brute-force computations,
whereas ours does not. We would also like to note that the truncated symbol 
$\Symb_0{(\mathbb{D}})$ is not discussed in their paper.

\subsection{Differential intertwining operators $\D$ for $(SL(3,\R), B)$}
\label{sec:IntroDIO}

As an application to the truncated symbol map $\Symb_0$, 
we provide the classification of DIOs $\D$ 
for the pair $(SL(3,\R), B)$ with their explicit formulas.
This is accomplished 
in Theorem \ref{thm:DIO1} for
the classification of the parameters 
$(\eps, \delta; \lambda, \nu) \in (\Z/2\Z)^4 \times \C^4$,
for which the space $\Diff_G(I(\lambda)^{(\eps)}, I(\nu)^{(\delta)})$
of DIOs $\D$ is non-zero. Here,
$I(\lambda)^{(\eps)}$ and  $I(\nu)^{(\delta)}$ are 
principal series representations of $SL(3,\R)$, which will be defined 
in Section \ref{subsec:SL3}.
The classification of DIOs $\D$ with explicit formulas 
is obtained in Theorem \ref{thm:DIO2}.

Our results show that  the DIOs $\D$ other than the identity map
are classified as the following 
five families of operators:
\begin{equation*}
\D_1^k, \quad
\D_2^k, \quad
\D_{+}^{(k,\ell)}, \quad
\D_{-}^{(k,\ell)}, \quad
\D_{\Cay}^{(s;k)},
\end{equation*}
where $k,\ell \in 1+\Z_{\geq 0}$ and $s \in \C$.
By a celebrated theorem of BGG--Verma, it is well-known that 
every DIO $\D$ is associated with a Weyl group element. 
If $\ga, \beta, \gamma:=\ga+\beta$ denote the positive roots of $\f{sl}(3,\C)$,
then the corresponding Weyl group elements are given as follows:
\begin{equation}\label{eqn:corresp}
\D_j^k\;  (j=1,2) \, \leftrightarrow s_\ga, s_\beta,\qquad 
\D_\pm^{(k,\ell)} \leftrightarrow s_\ga s_\beta, s_\beta s_\ga, \qquad 
\D_{\Cay}^{(s;k)} \leftrightarrow s_\gamma,
\end{equation}
where $s_\ga, s_\gb, s_\gamma$ are 
the root reflections with respect to $\ga, \gb, \gamma$, respectively.

In relation to their formulas,
the first two operators
$\D_j^k$ for $j=1,2$
are given by powers of root vectors for negative simple roots. 
The next two operators $\D_{\pm}^{(k,\ell)}$ are characterized as the ones
containing the products of two binomial coefficients in their coefficients.
The last operator $\D_{\Cay}^{(s;k)}$, which may be regarded as the most intriguing one, has coefficients involving the so-called Cayley continuants 
 $\{\Cay_m(x;y)\}_{m=0}^\infty$, 
 a family of tridiagonal determinants first studied by Cayley in 1858 
 (cf.\ \cite{Cayley58} and \cite[Intro.]{MT05}).
Remark that the ``Cayley operator'' $\D_{\Cay}^{(s;k)}$ is defined 
for  a continuous parameter $s \in \C$.
Although there are a number of studies for DIOs $\D$ 
in the form of homomorphisms between Verma modules,
none of them seems to identify 
the classical tridiagonal determinants (see Section \ref{sec:IntroHom}).
We shall discuss in detail 
the explicit formulas of $\D$ in Section  \ref{sec:construction}
and
the correspondence \eqref{eqn:corresp} in Section  \ref{sec:InfChar}.
Moreover, it is also shown in Section \ref{sec:Uniform}
that the coefficients of $\D_{\pm}^{(k,\ell)}$ and $\D_{\Cay}^{(s;k)}$ are
given in terms of the values of Jacobi polynomials $P^{(\ga, \beta)}_m(z)$ 
at $z=0$. 
(Here, $\ga, \beta$ stand for complex parameters as opposed to 
\eqref{eqn:corresp}.)
The uniform expressions play a role to study factorization identities
for $\D$ in degenerated cases (see 
Remark \ref{rem:pJ} and Cases (5) and (6) of 
Theorems \ref{thm:factor1} and \ref{thm:factor2}).

Remark that
our families of operators include some differential operators
constructed by Barchini--Kable--Zierau \cite{BKZ08}, Kable \cite{Kable12c},
and the first author \cite{Kubo11}. For instance, Kable
constructed a second order differential operator $\square_s^{(n)}$ 
with complex parameter $s\in \C$ for $(\f{sl}(n,\C), \fp_{\text{Heis}})$ 
in \cite{Kable12c}, where $\fp_{\text{Heis}}=\fl \oplus \fn_+$
is the parabolic subalgebra of $\f{sl}(n,\C)$ with Heisenberg nilpotent radical $\fn_+$.
The operator $\square_s^{(3)}$ with $n=3$ 
is exactly the Cayley operator $\D_{\Cay}^{(s;1)}$ with $k=1$.
The representations realized on the kernel 
$\Ker(\square_s^{(3)})$ of the differential operator 
$\square_s^{(3)}=\D_{\Cay}^{(s;1)}$ 
are recently studied from various points of views.
On this matter, see, for instance, 
\cite{Dahl19, Frahm24, Kable12c, KuOr19, KuOr25b, Tamori21}.
The basic solutions of $\square_s^{(3)}$ are investigated in \cite{Kable12b}.

The DIOs $\D$ associated with $(\f{sl}(n,\C), \fp_{\text{Heis}})$ are also studied by 
K{\v r}i{\v z}ka--Somberg  in connection with a branching law
of generalized Verma modules.
In \cite[Examples 3.6]{KrSom17}, 
they gave some examples of 
DIOs $\D$ for $(\f{sl}(n,\C), \fp_{\text{Heis}})$
with $n > 4$ 
and DSBOs $\mathbb{D}$ for 
$(\f{sl}(n,\C), \f{sl}(n-r,\C);\fp_{\text{Heis}}, \fp'_{\text{Heis}})$
with $n-r>4$,
where $\fp'_{\text{Heis}}$ is the parabolic subalgebra
of $\f{sl}(n-r,\C)$  with Heisenberg nilpotent radical.

Recently, DSBOs $\mathbb{D}$
for the pair $(G, G') = (GL(3,\R), GL(2,\R))$ with 
 their minimal paraboilc subgroups $(B, B')$
are studied by Ditlevsen--Labriet \cite{DL26} from a distribution point of view.
They considered several
embeddings $\iota\colon G'/B' \hookrightarrow G/B$
and constructed DSBOs $\mathbb{D}$ with respect to each $\iota$.
We study those DSBOs $\mathbb{D}$ in light  of the F-method in \cite{DiKuPV26+}.

\subsection{Homomorphisms $\varphi$ between Verma modules and 
their factorization identities}
\label{sec:IntroHom}

It is well-known that the category of DIOs $\D$ between parabolically induced
representations for a pair $(G,P)$
is equivalent to that of $(\fg, P)$-homomorphisms $\varphi$ 
between the corresponding generalized Verma modules for 
$(\fg, \fp)$
(see, for instance, \cite{CS90, KP1, KR00, Kostant75}).
Here $\fg$ and $\fp$ denote the complexified Lie algebras of $G$ and $P$,
respectively.
Thus, we classify and construct homomorphisms $\varphi$ 
between Verma modules corresponding to DIOs $\D$  as well. 
These are achieved in Theorems \ref{thm:Hom1} and \ref{thm:Hom2}.

Some homomorphisms $\varphi \colon M(\mu_3) \to M(\mu_1)$ 
may admit an identity 
$\varphi= \varphi_1 \circ \varphi_2$
with $\varphi_1\colon M(\mu_2) \to M(\mu_1)$ and 
$\varphi_2 \colon M(\mu_3) \to M(\mu_2)$.
Such an identity is called a \emph{factorization identity}.
We shall show the factorization identities of our homomorphisms $\varphi$
in Section \ref{sec:factorization}.
It is noteworthy that 
a binomial Krawtchouk polynomial $K_m(x;y)$,
which is an incarnation of a Cayley continuant $\Cay_m(x;y)$,
appears in the proof.

As Verma modules are fundamental objects in the representation theory of 
finite-dimensional complex simple Lie algebras,
there are already a number of works for the construction
of the homomorphisms $\varphi$ between Verma modules,
or more precisely, 
the explicit formulas of  singular vectors.
Those include the works of
de Graaf \cite{deGraaf05},
Malikov--Fe{\v i}gin--Fuks \cite{MFF86},
Xiao \cite{Xiao15b, Xiao15a, Xiao18},
and
Xu \cite{Xu05, Xu12},  among others.
For instance, 
Malikov--Fe{\v i}gin--Fuks \cite{MFF86} studied 
explicit formulas of singular vectors of Verma modules over Kac--Moody algebras. 
They, in particular, gave singular vectors for $\f{sl}(n,\C)$.
Remark that the Cayley continuants $\Cay_m(x;y)$ do not appear 
in any of the aforementioned works. 
A key difference between our approach and theirs 
is symmetrization; they fix a certain order of root vectors to obtain
singular vectors, while we take symmetrization of them.

\subsection{Organization of the paper}

The paper consists of eight sections including the introduction.
In the first two sections, namely, Sections \ref{sec:Fmethod} and \ref{sec:symbF},
we discuss the F-method in a full generality, namely, 
in a symmetry breaking setting $G \supseteq G'$.
More precisely, in Section \ref{sec:Fmethod}, 
we overview some necessary facts about the F-method.
Then, in Section \ref{sec:symbF}, 
we determine the inverse 
$F_c^{-1}\colon \Pol(\fn_+)\otimes V^\vee 
\stackrel{\sim}{\to}
M_\fp^\fg(V^\vee)$
of the algebraic Fourier transform $F_c$ of the 
generalized Verma module $M_\fp^\fg(V^\vee)$. 
We achieve it in Theorem \ref{thm:hdL} and Corollary \ref{cor:Fc}.
As part of our study of $F_c^{-1}$,
we develop several variants of symbol maps on the space $\Diff_N(N)$
of left-invariant differential operators on a simply-connected,
connected, nilpotent Lie group $N$.
The arguments culminate with the truncated symbol map 
$\Symb_0 \colon
\operatorname{Diff}_{G'}(\mathcal V, \mathcal W)
\stackrel{\sim}{\to}
\mathrm{Sol}(\mathfrak n_+;V,W)$
of the space 
$\operatorname{Diff}_{G'}(\mathcal V, \mathcal W)$ of DSBOs
 in Theorem \ref{thm:symb}.
At the end of the section, we give a recipe for the F-method by using
the inverse map $\Symb_0^{-1}$, 
which can be applied even for non-abelian nilpotent radicals.

The aim of Section \ref{sec:SL3} is to state our main results on 
the classification and construction of DIOs $\D$
and $(\fg, B)$-homomorphisms $\varphi$ between Verma modules
for $(SL(3,\R),B)$.
The main results are stated in Theorems \ref{thm:DIO1} and \ref{thm:Hom1}
for classification of the parameters and Theorems \ref{thm:DIO2} and \ref{thm:Hom2}
for explicit formulas.
We also summarize 
the homomorphisms $\varphi$ in terms of infinitesimal characters
in Theorem \ref{thm:Hom3}.

Sections \ref{sec:proof1} and \ref{sec:proof2} are devoted to the proof 
of the theorems in Section \ref{sec:SL3}. Our proof is based on 
the recipe for the F-method.
We proceed with the first half of the steps in Section \ref{sec:proof1} and 
the remaining steps are carried out in Section \ref{sec:proof2}. 
We need to solve a system of PDEs at the outset.
Via a certain linear isomorphism, it is reduced to a system of 
ODEs, which successfully uncovers a striking connection between
DIOs $\D$ and Cayley continuants $\Cay_m(x;y)$.

In Section \ref{sec:Uniform}, we discuss some uniform expressions
of certain polynomials $p^{(k,\ell)}_{\pm}(t)$ and $p^{(s;k)}_{\Cay}(t)$, which 
in principle provide DIOs $\D$ and homomorphisms $\varphi$ in concern.
We shall show in Theorem \ref{thm:unif}  
that these polynomials can be expressed in terms of 
Jacobi polynomials $P^{(\ga,\beta)}_{m}(z)$ with certain parameters $(\ga,\beta)$
at $z=0$. At the end of the section, 
binary Krawtchouk polynomials $K_m(x;y)$ are also discussed.

The factorization identities of DIOs $\D$ and homomorphisms $\varphi$ are 
the main topics of Section \ref{sec:factorization}.
In this section we first recall well-known results of Verma and BGG on the 
classification of homomorphisms between Verma modules. 
The factorization identities are then stated in 
Theorems \ref{thm:factor1} and \ref{thm:factor2} 
for $(\fg, B)$-homomorphisms $\varphi$ and DIOs $\D$, respectively.
The binary Krawtchouk polynomials $K_m(x;y)$
play a pivotal role in the proof.

\section{Preliminaries: the F-method}
\label{sec:Fmethod}

The aim of this section is to give a quick review on the F-method,
which plays a central role in this paper. In the next section 
we shall discuss the truncated symbol $\Symb_0(\mathbb{D})$ 
of a DSBO $\mathbb{D}$. 
In this section we mainly take the expositions from \cite{Kubo26, KuOr25a}.

\subsection{Notation}\label{sec:prelim}

Let $G$ be a real reductive Lie group and $P=MAN_+ $ a Langlands decomposition of a parabolic subgroup $P$ of $G$. We denote by $\fg(\R)$ and 
$\fp(\R) = \fm(\R) \oplus \fa(\R) \oplus \fn_+(\R)$ the Lie algebras of $G$ and 
$P=MAN_+$, respectively.

 For a real Lie algebra $\f{y}(\R)$, we write $\f{y}$
and $\Cal{U}(\fy)$ for its complexification and the universal enveloping algebra of 
$\fy$, respectively. 
For instance, $\fg, \fp, \fm, \fa$, and $\fn_+$ are the complexifications of $\fg(\R), \fp(\R), \fm(\R), \fa(\R)$, and $\fn_+(\R)$, 
respectively.

For $\lambda \in \fa^* \simeq \Hom_\R(\fa(\R),\C)$,
we denote by $\C_\lambda$ 
the one-dimensional representation of $A$ defined by 
$a\mapsto a^\lambda:=e^{\lambda(\log a)}$. 
For a finite-dimensional irreducible 
representation $(\xi, V)$ of $M$ and $\lambda \in \fa^*$,
we denote by $\xi_\lambda$ the outer tensor product 
representation $\xi \boxtimes \C_\lambda$
on $V$,
namely, 
$\xi_\lambda \colon
ma \mapsto a^\lambda\xi(m)$. By letting $N_+$ act trivially, we regard 
$\xi_\lambda$ as a representation of $P$. Let $\Cal{V}:=G \times_P V \to G/P$
be the $G$-equivariant vector bundle over the real flag variety $G/P$
 associated with the 
representation $(\xi_\lambda, V)$ of $P$. We identify the Fr{\' e}chet space 
$C^\infty(G/P, \Cal{V})$ of smooth sections with 
\begin{equation*}
C^\infty(G, V)^P:=\{f \in C^\infty(G,V) : 
f(gp) = \xi_\lambda^{-1}(p)f(g)
\;\;
\text{for any $p \in P$ and $g \in G$}\},
\end{equation*}
the space of $P$-invariant smooth functions on $G$.
Then, via the left regular representation $L$ of $G$ on $C^\infty(G)$,
we realize the parabolically induced representation 
$\pi_{(\xi, \lambda)} := \Ind_{P}^G(\xi_\lambda)$ on $C^\infty(G/P, \Cal{V})$.
We denote by $R$ the right regular representation of $G$ on $C^\infty(G)$.

Let $G'$ be a reductive subgroup of $G$ and $P'=M'A'N_+'$ 
a parabolic subgroup of $G'$
with $P' \subset P$
so that there is a natural morphism $p\colon G'/P' \to G/P$
from the real flag variety $G'/P'$ to $G/P$.
We further assume that $M'A' \subset MA$ and $N_+' \subset N_+$.

As for $G/P$, for a finite-dimensional irreducible
representation $(\varpi_\nu, W)$ of $M'A'$, we define
the induced representation $\pi'_{(\varpi, \nu)}:=\Ind_{P'}^{G'}(\varpi_\nu)$ on the 
space $C^\infty(G'/P', \Cal{W})$ of smooth sections for the
$G'$-equivariant vector bundle $\Cal{W}:=G'\times_{P'}W \to G'/P'$.

Via the morphism $p\colon G'/P' \to G/P$, one can define 
differential operators $\bD\colon C^\infty(G/P, \Cal{V}) \to C^\infty(G'/P', \Cal{W})$
with respect to $p$
although $G/P$ and $G'/P'$ are different manifolds 
(see, for instance, \cite[Def.\ 2.1]{KP1}).
As $C^\infty(G/P, \Cal{V})$ is a $G$-representation and $G' \subseteq G$,
the space $C^\infty(G/P, \Cal{V})$ is also a $G'$-representation.
We then write $\Diff_{G'}(\Cal{V},\Cal{W})$ for the space of 
\emph{differential symmetry breaking operators} 
(differential $G'$-intertwining operators)
$\bD \colon C^\infty(G/P, \Cal{V}) \to C^\infty(G'/P', \Cal{W})$.
When $G=G'$ and $P=P'$, we also simply refer to 
them as \emph{differential intertwining operators}
(or intertwining differential operators).

Let 
$\mathfrak{g}(\R)=\mathfrak{n}_-(\R) \oplus \mathfrak{m}(\R) 
\oplus \mathfrak{a}(\R) \oplus \mathfrak{n}_+(\R)$ 
be the 
Gelfand--Naimark decomposition of $\mathfrak{g}(\R)$
with $\fp(\R) = \mathfrak{m}(\R) \oplus \mathfrak{a}(\R) \oplus \mathfrak{n}_+(\R)$,
and write $N_- := \exp(\fn_-(\R))$. We identify $N_-$ with the 
open Bruhat cell $N_-P$ of $G/P$ via the embedding 
$\iota\colon N_- \hookrightarrow G/P$, $\bar{n} \mapsto \bar{n}P$.
Via the restriction of the vector bundle $\Cal{V} \to G/P$ to the open Bruhat cell
$N_-\stackrel{\iota}{\hookrightarrow} G/P$,
we regard $C^\infty(G/P,\mathcal{V})$ as a subspace of 
$C^\infty(N_-) \otimes V$.

Likewise,
let 
$\mathfrak{g}'(\R)=\mathfrak{n}_-'(\R) \oplus \mathfrak{m}'(\R) 
\oplus \mathfrak{a}'(\R) \oplus \mathfrak{n}_+'(\R)$ 
be the
Gelfand--Naimark decomposition of $\mathfrak{g}'(\R)$ with
$\fp(\R)'= \mathfrak{m}'(\R) 
\oplus \mathfrak{a}'(\R) \oplus \mathfrak{n}_+'(\R)$,
and write $N_-' := \exp(\fn_-'(\R))$.
As for $C^\infty(G/P,\Cal{V})$, we regard $C^\infty(G'/P',\Cal{W})$
as a subspace of $C^\infty(N_-')\otimes W$.

We often view differential symmetry breaking operators
$\bD \colon
C^\infty(G/P,\mathcal{V})
\to C^\infty(G'/P',\mathcal{W})$
as differential operators 
$\w{\bD} \colon C^\infty(N_-) \otimes V
\to C^\infty(N_-') \otimes W$ 
such that
the restriction $\w{\bD}\vert_{C^\infty(G/P,\mathcal{V})}$
to $C^\infty(G/P,\mathcal{V})$ is a map
$\w{\bD}\vert_{C^\infty(G/P,\mathcal{V})}\colon
C^\infty(G/P,\mathcal{V})
\to C^\infty(G'/P',\mathcal{W})$ (see \eqref{eqn:21} below).
\begin{equation}\label{eqn:21}
\begin{aligned}
\xymatrix{
C^\infty(N_-) \otimes V 
\ar[r]^{\w{\bD} } 
& C^\infty(N_-') \otimes W\\ 
C^\infty(G/P,\mathcal{V}) 
\;\; \ar[r]_{\stackrel{\phantom{a}}{\hspace{20pt}\bD=\w{\bD} 
\vert_{C^\infty(G/P,\mathcal{V})}} }
 \ar@{^{(}->}[u]^{\iota^*}
& \;\; C^\infty(G'/P',\mathcal{W}) \ar@{^{(}->}[u]_{\iota^*}
}
\end{aligned}
\end{equation}

\noindent
In particular, we often regard $\Diff_{G'}(\Cal{V},\Cal{W})$ as 
\begin{equation}\label{eqn:DN}
\Diff_{G'}(\Cal{V},\Cal{W}) \subset 
\Diff_\C(C^\infty(N_-)\otimes V, C^\infty(N_-')\otimes W).
\end{equation}

\subsection{Duality theorem}\label{sec:duality}

For a finite-dimensional irreducible representation $(\xi_\lambda,V)$ of $MA$,
we write $V^\vee = \Hom_\C(V,\C)$ and $((\xi_\lambda)^\vee, V^\vee)$ for
the contragredient representation of $(\xi_\lambda,V)$. By letting $\fn_+$ act 
on $V^\vee$ trivially, we regard the infinitesimal representation 
$d\xi^\vee \boxtimes \C_{-\lambda}$ of $(\xi_\lambda)^\vee$ as a $\fp$-module. For a finite-dimensional irreducible representation 
$(\varpi_\nu, W)$ of $M'A'$, the $\fp'$-module 
$d\varpi^\vee \boxtimes \C_{-\nu}$ is defined similarly.
We write
\begin{equation*}
\Mp(V^\vee) := \Cal{U}(\fg)\otimes_{\Cal{U}(\fp)}V^\vee
\quad
\text{and}
\quad
\Mpp(W^\vee) := \Cal{U}(\fg')\otimes_{\Cal{U}(\fp')}W^\vee
\end{equation*}
for the generalized Verma modules for $(\fg, \fp)$ and $(\fg',\fp')$ induced from 
$d\xi^\vee \boxtimes \C_{-\lambda}$ and $d\varpi^\vee \boxtimes \C_{-\nu}$,
respectively.
Via the diagonal action of $P$ on $\Mp(V^\vee)$, we regard
$M_\fp^\fg(V^\vee)$ as a $(\fg, P)$-module. Likewise, we regard 
$\Mpp(W^\vee)$ as a $(\fg', P')$-module.

The following theorem is often referred to as the \emph{duality theorem}. 

\begin{thm}[Duality theorem, {\cite[Thm.\ 2.9]{KP1}}]
\label{thm:duality}
There is a natural linear isomorphism
\begin{equation}\label{eqn:duality1}
\EuD
\colon
\operatorname{Hom}_{P'}(W^\vee,\Mp(V^\vee))
\stackrel{\sim}{\To} 
\operatorname{Diff}_{G'}(\mathcal V, \mathcal W).
\end{equation}
Equivalently,
\begin{equation}\label{eqn:duality2}
\EuD
\colon
\operatorname{Hom}_{\fg', P'}(\Mpp(W^\vee),\Mp(V^\vee))
\stackrel{\sim}{\To} 
\operatorname{Diff}_{G'}(\mathcal V, \mathcal W).
\end{equation}
\end{thm}

Observe that 
a $P'$-equivariant map 
\begin{equation*}
\Phi \in \Hom_{P'}(W^\vee, \Mp(V^\vee)) 
\simeq (\Mp(V^\vee) \otimes W)^{P'} \subset 
\Cal{U}(\fn_-)\otimes V^\vee \otimes W
\end{equation*}
is of the form
\begin{equation*}
\Phi = \sum_{i,j}u_{i,j}\otimes v^\vee_i \otimes w_j
\end{equation*}
for some $u_{i,j} \in \Cal{U}(\fn_-)$, $v_i^\vee \in V^\vee$, and $w_j \in W$.
Then, under the identification \eqref{eqn:DN},
the differential operator 
\begin{equation*}
\EuD(\Phi)\colon C^\infty(N_-)\otimes V\to C^\infty(N_-')\otimes W
\end{equation*}
is given as
\begin{equation}\label{eqn:HD}
\EuD(\Phi) = \sum_{i,j}dR(u_{i,j})\vert_{N_-'} \otimes v^\vee_i \otimes w_j,
\end{equation}
where $dR$ is the infinitesimal right translation.

\subsection{Algebraic Fourier transform $\widehat{\;\;\cdot\;\;}$ of Weyl algebras}
\label{sec:Weyl}

Let $U$ be a complex finite-dimensional vector space with $\dim _\C U=n$.
Fix a basis $u_1,\ldots, u_n$ of $U$ and let 
$(z_1, \ldots, z_n)$ denote the coordinates of $U$ with respect to the basis.
Then the algebra
\begin{equation*}
\C[U;z, \tfrac{\partial}{\partial z}]:=
\C[z_1, \ldots, z_n, 
\tfrac{\partial}{\partial z_1}, \ldots, \tfrac{\partial}{\partial z_n}]
\end{equation*}
with relations $z_iz_j = z_jz_i$, 
$\frac{\partial}{\partial z_i}\frac{\partial}{\partial z_j} 
=\frac{\partial}{\partial z_j}\frac{\partial}{\partial z_i}$,
and $\frac{\partial}{\partial z_j} z_i=\delta_{i,j} + z_i \frac{\partial}{\partial z_j}$
is called the Weyl algebra of $U$, where $\delta_{i,j}$ is the Kronecker delta.
Similarly,
let $(\zeta_1,\ldots, \zeta_n)$ denote the coordinates of 
the dual space $U^\vee$ of $U$ with respect to the dual basis of 
$u_1,\ldots, u_n$. We write 
$\C[U^\vee;\zeta, \tfrac{\partial}{\partial \zeta}]$
for the Weyl algebra of $U^\vee$.
Then the map determined by
\begin{equation}\label{eqn:Weyl}
\widehat{\frac{\partial}{\partial z_i}}:= -\zeta_i, \quad
\widehat{z_i}:=\frac{\partial}{\partial \zeta_i}
\end{equation}
gives a Weyl algebra isomorphism 
\begin{equation}\label{eqn:Weyl2}
\widehat{\;\;\cdot\;\;}\;\colon\C[U;z, \tfrac{\partial}{\partial z}] 
\stackrel{\sim}{\To} 
\C[U^\vee;\zeta, \tfrac{\partial}{\partial \zeta}], 
\quad T \mapsto \widehat{T}.
\end{equation}
\vskip 0.1in
\noindent
The map \eqref{eqn:Weyl2} is called
the \emph{algebraic Fourier transform of Weyl algebras}
(\cite[Def.\ 3.1]{KP1}).

\subsection{Fourier transformed representation $\widehat{d\pi_{(\xi,\lambda)^*}}$}
\label{sec:dpi2}

For a representation $\eta$ of $G$, we denote by $d\eta$ 
the infinitesimal representation of $\fg(\R)$. 
For instance,  $dL$ denotes the infinitesimal representation of
the left regular representation of $G$ on $C^\infty(G)$.
As usual, we naturally extend representations of $\fg(\R)$ 
to ones for the universal enveloping algebra $\Cal{U}(\fg)$ of its complexification $\fg$.
The same convention is applied for closed subgroups of $G$.

For $g \in N_-MAN_+$, we write
\begin{equation*}
g = p_-(g)p_0(g)p_+(g),
\end{equation*}
where $p_\pm(g) \in N_{\pm}$ and $p_0(g) \in MA$. 
Similarly, for $Z \in \fg = \fn_- \oplus \fl\oplus \fn_+$ with $\fl= \fm \oplus \fa$,
we write
\begin{equation*}
Z=Z_{\fn_-} + Z_{\fl} + Z_{\fn_+},
\end{equation*}
where $Z_{\fn_{\pm}} \in \fn_{\pm}$ and $Z_\fl \in \fl$.

For $2\rho\equiv 2\rho(\fn_+)= \mathrm{Trace}(\ad\vert_{\fn_+})\in \mathfrak{a}^*$,
we denote by  $\C_{2\rho}$ the one-dimensional representation of $P$
defined by
$p \mapsto \chi_{2\rho}(p)=
\abs{\mathrm{det}(\mathrm{Ad}(p)\colon 
\mathfrak{n}_+ \to \mathfrak{n}_+)}$.
For the contragredient representation
$((\xi_\lambda)^\vee, V^\vee)$ of $(\xi_\lambda,V)$,
we put
$\xi^*_\lambda := \xi^\vee \boxtimes \C_{2\rho-\lambda}$.
As for $\xi_\lambda$, we regard $\xi^*_\lambda$ as a representation of $P$.
Define the induced representation
$\pi_{(\xi, \lambda)^*} := \mathrm{Ind}_P^G(\xi^*_\lambda)$
on the space $C^\infty(G/P,\mathcal{V}^*)$ of smooth sections 
for the vector bundle $\mathcal{V}^*:=G\times_P (V^\vee\otimes \C_{2\rho})$
associated with $\xi^*_\lambda$, which is isomorphic to the tensor bundle
of the dual vector bundle $\mathcal{V}^\vee = G\times_PV^\vee$ and the bundle
of densities over $G/P$.
Then the integration on $G/P$ gives a 
$G$-invariant non-degenerate bilinear form
$
\mathrm{Ind}^G_P(\xi_\lambda) 
\times \mathrm{Ind}^G_P(\xi^*_\lambda) \to \C
$
for $\mathrm{Ind}^G_P(\xi_\lambda)$ and 
$\mathrm{Ind}^G_P(\xi^*_\lambda)$.

As for $C^\infty(G/P,\mathcal{V})$,
the space $C^\infty(G/P,\mathcal{V}^*)$ can be regarded as 
a subspace of $C^\infty(N_-) \otimes V^\vee$.
Then the infinitesimal representation $d\pi_{(\xi,\lambda)^*}(Z)$
on $C^\infty(N_-)\otimes V^\vee$
for $Z \in \fg$ is given as
\begin{equation}\label{eqn:dpi3}
d\pi_{(\xi,\lambda)^*}(Z)f(\bar{n})
=d\xi_\lambda^*((\Ad(\bar{n}^{-1})Z)_\fl)f(\bar{n})
-\left(dR((\Ad(\cdot^{-1})Z)_{\fn_-})f\right)(\bar{n}).
\end{equation}
(For the details, see, for instance, \cite[Sect.\ 2]{KuOr25a}.)

Via the exponential map $\exp\colon \fn_-(\R) \stackrel{\sim}{\to} N_-$, one can regard
$\dpi_{(\xi,\lambda)^*}(Z)$ as a representation 
on $C^\infty(\mathfrak{n}_-(\R)) \otimes V^\vee$.
It then follows from \eqref{eqn:dpi3} that 
$\dpi_{(\xi,\lambda)^*}$ gives a Lie algebra homomorphism
\begin{equation*}
d\pi_{(\xi,\lambda)^*}\colon \mathfrak{g} \To 
\C[\fn_-;z, \tfrac{\partial}{\partial z}]  \otimes \mathrm{End}(V^\vee),
\end{equation*}
where $(z_1, \ldots, z_n)$ are coordinates of $\fn_-$ with $n=\dim_{\C} \fn_-$.

Now we fix a non-degenerate $\Ad$-invariant 
symmetric bilinear form $\kappa$ on $\fg$.
Via $\kappa$, we identify $\fn_+$  with the dual space $\fn_-^\vee$ of 
$\fn_-$. Then
the algebraic Fourier transform $\widehat{\;\;\cdot\;\;}$
of Weyl algebras \eqref{eqn:Weyl2} gives a Weyl algebra isomorphism
\begin{equation*}
\widehat{\;\;\cdot\;\;}\;\colon
\C[\fn_-;z, \tfrac{\partial}{\partial z}] 
\stackrel{\sim}{\To} 
\C[\fn_+;\zeta, \tfrac{\partial}{\partial \zeta}].
\end{equation*}
In particular, it gives a Lie algebra homomorphism
\begin{equation}\label{eqn:hdpi}
\widehat{d\pi_{(\xi,\lambda)^*}}\colon \mathfrak{g} \To 
\C[\fn_+;\zeta, \tfrac{\partial}{\partial \zeta}]\otimes \mathrm{End}(V^\vee).
\end{equation}

Let 
$\Cal{D}'_{[o]}(G/P,\Cal{V}^*)$ 
denote the space of 
$\mathcal{V}^*$-valued
distributions on $G/P$ supported at the identity coset $[o]=eP$ of $G/P$.
Likewise, we write $\Cal{D}'_{0}(\fn_-(\R),V^\vee)$ for the space 
of $V^\vee$-valued distributions on $\fn_-(\R)$ supported at $0$.
Then the following chain of $(\fg,P)$-isomorphisms holds:
\begin{alignat}{4}\label{eqn:VP}
\Mp(V^\vee)
&\stackrel{\sim}{\To} \Cal{D}'_{[o]}(G/P,\Cal{V}^*) 
&&\stackrel{\sim}{\To} \Cal{D}'_{0}(\fn_-(\R),V^\vee)
&&\stackrel{\sim}{\relbar\joinrel\relbar\joinrel\To} \Pol(\fn_+)\otimes V^\vee\\
\rotatebox{90}{$\in$}\hspace{20pt}
&\hspace{50pt}\rotatebox{90}{$\in$}
&&\hspace{50pt}\rotatebox{90}{$\in$}
&&\hspace{55pt}\rotatebox{90}{$\in$}& \nonumber\\[-2pt]
u\otimes v^\vee \hspace{8pt}
&\mapsto dL(u)(\delta_{[o]} \otimes v^\vee) 
&&\mapsto d\pi_{(\xi,\lambda)^*}(u)(\delta_0\otimes &&\!v^\vee)
\mapsto \widehat{d\pi_{(\xi,\lambda)^*}}(u)(1\otimes v^\vee), \nonumber
\end{alignat}
where $\delta_{[o]}$ and $\delta_0$ denote the Dirac delta functions supported at
$[o]$ and $0$, respectively  (cf.\ \cite{KP1, KuOr25a}).

Let $F_c$ denote the composition of the three $(\mathfrak{g}, P)$-module isomorphisms. We call the resulting isomorphism
\begin{equation}\label{eqn:Fc}
F_c\colon \Mp(V^\vee) 
\stackrel{\sim}{\To}
\Pol(\fn_+) \otimes V^\vee,
\quad u\otimes v^\vee \longmapsto 
\widehat{d\pi_{(\xi,\lambda)^*}}(u)(1\otimes v^\vee)\\
\end{equation}
the \emph{algebraic Fourier transform of the generalized Verma module  $M_\fp(V^\vee)$} (\cite[Sect.\ 3.4]{KP1}).

\subsection{The F-method}
\label{sec:Fmethod2}
Observe that
the algebraic Fourier transform
$F_c$  in \eqref{eqn:Fc} 
gives an $M'A'$-isomorphism
\begin{equation}\label{eqn:VP2}
\Mp(V^\vee)^{\fn_+'}
\stackrel{\sim}{\To}
(\Pol(\fn_+) 
\otimes V^\vee)^{\widehat{d\pi_{(\xi,\lambda)^*}}(\fn_+')},
\end{equation}
which induces a linear isomorphism
\begin{equation}\label{eqn:VP31}
\operatorname{Hom}_{M'A'}(W^\vee,
\Mp(V^\vee)^{\fn_+'})
\stackrel{\sim}{\To}
\operatorname{Hom}_{M'A'}
\left(W^\vee,
(\mathrm{Pol}(\mathfrak{n}_+) 
\otimes V^\vee)^{\widehat{d\pi_{(\xi,\lambda)^*}}(\fn_+')}\right).
\end{equation}
Here $M'A'$ acts on $\Pol(\fn_+)$ via the action
\begin{equation}\label{eqn:sharp}
\Ad_{\#}(l) \colon p(\cdot) \mapsto p(\Ad(l^{-1})\cdot)
\;\; \text{for $l \in M'A'$}.
\end{equation}

Now we set
\begin{equation}\label{eqn:Sol2a}
\mathrm{Sol}(\mathfrak n_+;V,W):=
\operatorname{Hom}_{M'A'}(W^\vee,
(\mathrm{Pol}(\mathfrak{n}_+) 
\otimes V^\vee)^{\widehat{d\pi_{(\xi,\lambda)^*}}(\fn_+')}).
\end{equation}
Via the identification
$\textrm{Hom}_{M'A'}(W^\vee, \Pol(\fn_+)\otimes V^\vee)
\simeq
\big((\Pol(\fn_+)\otimes V^\vee) \otimes W\big)^{M'A'}$,
we have
\begin{align}\label{eqn:Sol}
&\Sol(\fn_+;V,W)\nonumber \\[3pt]
&=\{ \psi \in \Hom_{M'A'}(W^\vee, \Pol(\fn_+)\otimes V^\vee): 
\text{
$\psi$ satisfies the system \eqref{eqn:Fsys} of PDEs below.}\}
\end{align}
\begin{equation}\label{eqn:Fsys}
(\widehat{d\pi_{(\xi,\lambda)^*}}(C) \otimes \id_W)\psi=0
\,\,
\text{for all $C \in \fn_+'$}.
\end{equation}
We refer to the system \eqref{eqn:Fsys} of PDEs 
as the \emph{F-system} (\cite[Fact 3.3 (3)]{KKP16}).
Since
\begin{equation*}
\operatorname{Hom}_{P'}(W^\vee,\Mp(V^\vee))
=
\operatorname{Hom}_{M'A'}(W^\vee,\Mp(V^\vee)^{\fn_+'}),
\end{equation*}
the isomorphism \eqref{eqn:VP31} together with \eqref{eqn:Sol2a} shows the following.

\begin{thm}
[F-method, {\cite[Thm.\ 4.1]{KP1}}]
\label{thm:Fmethod}
There exists a linear isomorphism
\begin{equation*}
F_c \otimes \mathrm{id}_{W}\colon 
\operatorname{Hom}_{P'}(W^\vee,
\Mp(V^\vee))
\stackrel{\sim}{\To}
\mathrm{Sol}(\mathfrak n_+;V,W).
\end{equation*}
Equivalently, we have
\begin{equation*}
F_c \otimes \mathrm{id}_{W}\colon 
\mathrm{Hom}_{\mathfrak g',P'}(\Mpp(W^\vee),\Mp(V^\vee))\stackrel{\sim}{\To}
\mathrm{Sol}(\mathfrak n_+;V,W).
\end{equation*}
\end{thm}

The following diagram summarizes the duality theorem and the F-method.
\begin{equation}\label{eqn:isom}
\begin{tikzcd}[row sep=1.3cm, column sep=1cm]
& 
\mathrm{Sol}(\mathfrak n_+;V,W)
\arrow[dr, dashed, "\sim"' sloped] \\
\operatorname{Hom}_{\mathfrak{g}',P'}(\Mpp(W^\vee),\Mp(V^\vee)) 
  \arrow[rr, "\sim", "\EuScript{D}"']
  \ar[ur, "\sim"' sloped, "F_c\otimes\mathrm{id}_W"] 
 &
 & 
 \operatorname{Diff}_{G'}(\mathcal V, \mathcal W).
\end{tikzcd}
\end{equation}

\vskip 0.1in
The next lemma will be useful for later arguments.

\begin{lem}\label{lem:dpiL}
We have
\begin{equation}\label{eqn:dpiL}
d\pi_{(\xi,\lambda)^*}\vert_{\Cal{U}(\fn_-)} =
dL \vert_{\Cal{U}(\fn_-)}\otimes \id_{V^\vee}.
\end{equation}
\end{lem}

\begin{proof}
The identity readily follows from the derivation of the infinitesimal action $d\pi_{(\xi,\lambda)^*}$. 
Alternatively, one can also show \eqref{eqn:dpiL} from the formula 
\eqref{eqn:dpi3} directly. Indeed, for $Z \in \fg$, the action 
$d\pi_{(\xi,\lambda)^*}(Z)$ on $C^\infty(N_-)\otimes V^\vee$ is given by
\begin{equation}\label{eqn:dpi2}
d\pi_{(\xi,\lambda)^*}(Z)f(\bar{n})
=d\xi_\lambda^*((\Ad(\bar{n}^{-1})Z)_\fl)f(\bar{n})
-\left(dR((\Ad(\cdot^{-1})Z)_{\fn_-})f\right)(\bar{n}).
\end{equation}
For $X\in \fn_-$, we have
$(\Ad(\bar{n})X)_\fl = 0$
and $(\Ad(\bar{n})X)_{\fn_-} = \Ad(\bar{n})X$.
Therefore, for $Z=X \in \fn_-$, the formula \eqref{eqn:dpi2} reads
\begin{equation}\label{eqn:dpiL2}
d\pi_{(\xi,\lambda)^*}(X)f(\bar{n})
=-\left(dR(\Ad(\cdot^{-1})X)f\right)(\bar{n}).
\end{equation}
On the other hand, the infinitesimal left translation $dL(X)$ is given as
\begin{align}\label{eqn:LR}
dL(X)f(\bar{n}) 
&=\frac{d}{dt}\bigg\vert_{t=0} f(\exp(-tX)\bar{n}) \nonumber \\[3pt]
&=\frac{d}{dt}\bigg\vert_{t=0} f(\bar{n}\exp(-\Ad(\bar{n}^{-1})X))  \nonumber \\[3pt]
&=-dR(\Ad(\cdot^{-1})X)f(\bar{n}).
\end{align}
Now \eqref{eqn:dpiL2} and \eqref{eqn:LR} conclude the proposed identity.
\end{proof}

\begin{rem}\label{rem:Fc}
It follows from Lemma \ref{lem:dpiL} that,
for $u\otimes v^\vee \in \Cal{U}(\fn_-)\otimes V^\vee 
\simeq \Mp(V^\vee)$, the diagram \eqref{eqn:VP} is reduced to the following.
\begin{alignat}{4}\label{eqn:VP1}
\Cal{U}(\fn_-)\otimes V^\vee
&\stackrel{\sim}{\To}  \Cal{D}'_{[o]}(G/P,\Cal{V}^*) 
&&\stackrel{\sim}{\To}  \Cal{D}'_{0}(\fn_-(\R),V^\vee)
&&\stackrel{\sim}{\relbar\joinrel\To}  \Pol(\fn_+)\otimes V^\vee\\
\rotatebox{90}{$\in$}\hspace{20pt}
&\hspace{50pt}\rotatebox{90}{$\in$}
&&\hspace{50pt}\rotatebox{90}{$\in$}
&&\hspace{55pt}\rotatebox{90}{$\in$}& \nonumber\\[-2pt]
u\otimes v^\vee \hspace{8pt}
&\mapsto dL(u)(\delta_{[o]} \otimes v^\vee) 
&&\longmapsto\hspace{5pt} dL(u)(\delta_0 \otimes v^\vee&&) 
\longmapsto \hspace{5pt} \widehat{dL}(u)(1\otimes v^\vee) \nonumber
\end{alignat}

\end{rem}

\subsection{Notation for the action of the Weyl algebra}
\label{sec:Weyl2}

In the rest of the paper,
we denote by $\acts$ the action of the elements of
the Weyl algebra
$
\C[\zeta, \tfrac{\partial}{\partial \zeta}]
=\C[\zeta_1, \ldots, \zeta_n, 
\tfrac{\partial}{\partial \zeta_1}, \ldots, \tfrac{\partial}{\partial \zeta_n}]$
on $\C[\zeta]$, unless it is clear from the context.
For instance, the notation $\tfrac{\partial}{\partial  \zeta_j}\zeta_i$ stands for
the product of the differential operator $\tfrac{\partial}{\partial \zeta_j}$ and 
the multiplication operator by $\zeta_i$
as a Weyl algebra element, while $\tfrac{\partial}{\partial \zeta_j}\acts \zeta_i$ means
the action of $\tfrac{\partial}{\partial \zeta_j}$ on the monomial
$\zeta_i$. For instance, we have
\begin{alignat*}{4}
&\frac{\partial}{\partial \zeta_j} 1 
&&= \frac{\partial}{\partial \zeta_j},
\qquad
&&\frac{\partial}{\partial \zeta_j}\zeta_i 
&&= \zeta_i \frac{\partial}{\partial \zeta_j} + \delta_{i,j},\\[5pt]
&\frac{\partial}{\partial \zeta_j}\acts 1 
&&= 0,
\qquad
&&\frac{\partial}{\partial \zeta_j}\acts \zeta_i  &&= \delta_{i,j}.
\end{alignat*}

Remark that
the polynomial $(D_1D_2) \acts \psi(\zeta)$ 
for $D_1, D_2 \in \C[\zeta, \tfrac{\partial}{\partial \zeta}]$ and 
$\psi(\zeta) \in \C[\zeta]$
is computed as
\begin{equation}\label{eqn:acts}
(D_1D_2) \acts \psi(\zeta) = D_1 \acts (D_2\acts \psi(\zeta)).
\end{equation}
For instance, we have 
\begin{equation*}
(\frac{\partial}{\partial \zeta_j}\zeta_i) \acts \psi(\zeta)
= \frac{\partial}{\partial \zeta_j}\acts (\zeta_i\acts \psi(\zeta))
=\frac{\partial}{\partial \zeta_j} \acts \zeta_i\psi(\zeta)
=\delta_{i,j}\psi(\zeta) + \zeta_i \frac{\partial}{\partial \zeta_j} \acts \psi(\zeta).
\end{equation*}
In particular,
\begin{equation*}
(\frac{\partial}{\partial \zeta_j}\zeta_i) \acts 1
= \frac{\partial}{\partial \zeta_j}\acts (\zeta_i\acts 1)
=\frac{\partial}{\partial \zeta_j} \acts \zeta_i
=\delta_{i,j}.
\end{equation*}

Recall from Remark \ref{rem:Fc} that,
for a finite-dimensional irreducible representation $V$ of $MA$, the algebraic Fourier
transform $F_c$ gives an $MA$-isomorphism
\begin{equation}\label{eqn:FcL0}
F_c\colon \Cal{U}(\fn_-) \otimes V^\vee \stackrel{\sim}{\To} \Pol(\fn_+)\otimes V^\vee,
\; 
u\otimes v^\vee
\longmapsto
\widehat{dL}(u)(1\otimes v^\vee),
\end{equation}
where $\widehat{dL}(u)(1\otimes v^\vee)$ is given by
\begin{align*}
\widehat{dL}(u)(1\otimes v^\vee)=(\widehat{dL}(u) \acts 1)\otimes v^\vee.
\end{align*}

For the sake of simplicity, we write
\begin{equation*}
\hdL(u)_1:= \hdL(u)\acts 1
\end{equation*}
so that
\begin{equation}\label{eqn:FcL}
F_c(\cdot) = \hdL(\cdot)_1 \otimes \id_{V^\vee}
\end{equation}
on $\Cal{U}(\fn_-) \otimes V^\vee$.

Equation \eqref{eqn:FcL} will play a role to determine the inverse
$F_c^{-1}$ in the next section.

\section{The inverse of the algebraic Fourier transform $F_c$
and the truncated symbol map $\Symb_0$}\label{sec:symbF}

There are two objectives in this section. 
The first is to determine the inverse 
\begin{equation*}
F_c^{-1}\colon 
\Pol(\fn_+) \otimes V^\vee
\stackrel{\sim}{\To}
M_\fp^{\fg}(V^\vee)
\end{equation*}
of the algebraic Fourier transform $F_c$. This is done in 
Theorem \ref{thm:hdL} and Corollary \ref{cor:Fc}.
The other is to define a linear map
\begin{equation}\label{eqn:symb01}
\Symb_0\colon
\operatorname{Diff}_{G'}(\mathcal V, \mathcal W)
\stackrel{\sim}{\To}
\mathrm{Sol}(\mathfrak n_+;V,W)
\end{equation}
in such a way that the diagram \eqref{eqn:isom} commutes.
The linear map $\Symb_0$ is obtained in Theorem \ref{thm:symb} and Corollary \ref{cor:symb}.

\subsection{Preliminaries}\label{sec:prelim2}
We start with some preliminaries.
Let $G$ be a  Lie group with Lie algebra $\fg(\R)$.
We denote by $\fg(\R)^\vee$ the dual space of $\fg(\R)$. 
Fix an ordered basis $\Cal{B}:=\{X_1,\ldots, X_n\}$ of $\fg(\R)$
with $n=\dim(G)$ and write 
$\Cal{B}^\vee:=\{X_1^\vee,\ldots, X_n^\vee\}$ for the dual basis of $\fg(\R)^\vee$.
Let $(x_1, \ldots, x_n)$ denote the coordinates on $\fg(\R)$ with respect to
the ordered basis $\Cal{B}$. By slight abuse of notation, we also denote by 
$(x_1, \ldots, x_n)$ the coordinates on the 
complexification $\fg:=\fg(\R)\otimes_\R\C$ of $\fg(\R)$.
Likewise,
let $(\zeta_1,\ldots, \zeta_n)$
denote the coordinates on the dual space $\fg^\vee$ of $\fg$
with respect to $\Cal{B}^\vee$.

Let $S(\fg)$ denote the symmetric algebra of $\fg$ and 
write $\Pol(\fg)$ for  the ring of polynomial functions on $\fg$, namely,
$\Pol(\fg)=S(\fg^\vee)$.
We put 
\begin{align*}
\Diff_G(G)&:=
\text{the space of the left-invariant differential operators on $G$},\\[3pt]
\Diff^{\const}(\fg)&:= 
\text{the space of holomorphic differential operators on $\fg$
with constant coefficients}.
\end{align*}
We write  $\partial_j:=\frac{\partial}{\partial x_j} \in \Diff^{\const}(\fg)$ 
for $j=1,\ldots, n$.

\vskip 0.1in

The following well-known fact plays a key role.

\begin{thm}[{cf.\ \cite[Ch.\ II, Thm.\ 4.3]{Hel00}}]\label{thm:Hel00}
There exists a linear isomorphism
\begin{equation}\label{eqn:symbG}
 \Symb_G\colon  \Diff_G(G)  \stackrel{\sim}{\To} S(\fg)
\end{equation}
such that, for $P(X_1,\ldots, X_n) \in S(\fg)$, we have
\begin{equation*}
(\Symb_G^{-1}(P)f)(g) = P(\partial_1,\ldots, \partial_n)f(g\exp(x_1X_1+\cdots + x_nX_n))\big\vert_{x=0},
\end{equation*}
where $f(g) \in C^\infty(G)$ and $x:=(x_1,\ldots, x_n)$.
\end{thm}

For $Y_1, Y_2, \ldots, Y_k \in \fg$, 
we define
the symmetrization
 $\sym_k(Y_1Y_2\cdots Y_k) \in \Cal{U}(\fg)$
 of the product $Y_1Y_2\cdots Y_k \in \Cal{U}(\fg)$ of $k$ elements,
by
\begin{equation}\label{eqn:symm}
\sym_k(Y_1Y_2\cdots Y_k):=
\frac{1}{k!}\sum_{\tau \in \mathfrak{S}_k}
Y_{\tau(1)} \cdots Y_{\tau(k)},
\end{equation}
where $\mathfrak{S}_k$ denotes the symmetric group on $k$ letters.
As the next example shows, one needs extra care when taking the 
symmetrization of $u \in \Cal{U}(\fg)$.

\begin{example}
Let $\fg=\spn_\C\{X_1, X_2, X_3\}$ be
the three-dimensional Heisenberg Lie algebra, where $[X_1, X_2]=X_3$. Then $\sym_1(X_3) \neq \sym_2(X_3)$.
Indeed, we have
\begin{equation*}
\sym_1(X_3) = X_3,
\end{equation*}
whereas
\begin{equation*}
\sym_2(X_3) = \sym_2(X_1X_2 - X_2X_1) = \sym_2(X_1X_2)-\sym_2(X_2X_1)=0.
\end{equation*}
\end{example}

It follows from Theorem \ref{thm:Hel00} that if $Y_1, \ldots, Y_k \in \fg$,
then 
\begin{equation}\label{eqn:symbG0}
\Symb_G^{-1}(Y_1\cdots Y_k) 
= dR\big(\sym_k(Y_1\cdots Y_k)\big)
\end{equation}
(cf.\  \cite[pp.\ 280--282]{Hel00}).
Here, the product $Y_1\cdots Y_k$ on the left-hand side  
of \eqref{eqn:symbG0} is in $S(\fg)$ and on the right-hand side is in 
$\Cal{U}(\fg)$.

As $S(\fg) = \Pol(\fg^\vee)$,
one can understand \eqref{eqn:symbG} as
\begin{equation*}
\Symb_G\colon \Diff_G(G)\stackrel{\sim}{\To}\Pol(\fg^\vee).
\end{equation*}
Under the identification of the basis elements $X_j$ in $\Cal{B}$ for $\fg$
with the coordinate functions $\zeta_j$ on $\fg^\vee$,
the identity \eqref{eqn:symbG0} implies
\begin{equation}\label{eqn:symG}
\Symb_G^{-1}(\zeta_1^{r_1}\cdots \zeta_n^{r_n}) 
= dR\big(\sym_{|\mathbf{r}|}(X_1^{r_1}\cdots X_n^{r_n})\big),
\end{equation}
where $\mathbf{r} :=(r_1,\ldots, r_n) \in \Z_{\geq 0}^n$
and  $|\mathbf{r}|:=r_1+\cdots +r_n$.
\vskip 0.1in

Now recall that the symbol map 
\begin{equation*}
\Symb\colon \Diff^{\const}(\fg) \stackrel{\sim}{\To} \Pol(\fg^\vee)
\end{equation*}
is an algebra isomorphism such that 
\begin{equation}\label{eqn:symb1013}
De^{\IP{x}{\zeta}} = \Symb(D)e^{\IP{x}{\zeta}},
\end{equation}
where $\IP{\cdot}{\cdot}$ is the natural pairing of $\fg$ and $\fg^\vee$.
The symbol $\Symb(P(\partial_1,\ldots, \partial_n))$
of $P(\partial_1,\ldots, \partial_n)$ is given as
\begin{equation*}
\Symb(P(\partial_1,\ldots, \partial_n)) = P(\zeta_1,\ldots, \zeta_n).
\end{equation*}
Then we define a linear isomorphism
\begin{equation*}
\phi\colon \Diff_G(G)\stackrel{\sim}{\To}\Diff^{\const}(\fg)
\end{equation*}
by
\begin{equation*}
 \phi:=\Symb^{-1}\circ \Symb_G.
\end{equation*}
By definition, the following diagram commutes.
\begin{equation}\label{eqn:diagram1}
\begin{tikzcd}[row sep=1cm, column sep=1cm]
\mathrm{Pol}(\fg^\vee)
& 
\arrow[dl, pos=0.85, phantom, "\circlearrowleft"]
\\
\Diff_G(G) 
\arrow[u, "\Symb_G",  "\sim"' sloped]
\arrow[r, "\phi"', "\sim" sloped]
& \Diff^{\const}(\fg)
\arrow[lu, "\Symb"',  "\sim"' sloped]
\end{tikzcd}
\end{equation}
The differential operator 
$\phi(D) \in \Diff^{\const}(\fg)$ 
for $D \in \Diff_G(G)$ with $\Symb_G(D)=P(\zeta_1,\ldots, \zeta_n)$
is given by
\begin{equation*}
\phi(D) = P(\partial_1,\ldots, \partial_n).
\end{equation*}

\subsection{The truncated symbol $\Symb_0$ for the Weyl algebra $\D(E)$}\label{sec:symb0}

As in Section \ref{sec:prelim2},
given a finite-dimensional vector space $E$ over $\C$, we denote by
$E^\vee$ and $\Diff^\const(E)$  the dual space of $E$ 
and the space of holomorphic differential operators on $E$ with constant coefficients,
respectively. 

\begin{defn}
We define a linear map
\begin{equation*}
\Trun_0\colon S(E^\vee) \otimes  \Diff^\const(E) \To \Diff^\const(E)
\end{equation*}
by
\begin{equation*}
\Trun_0\colon \sum_{\alpha}a_{\alpha}D_{\alpha} \longmapsto 
\sum_{\alpha}a_{\alpha}(0)D_{\alpha},
\end{equation*}
namely,
\begin{equation*}
\Trun_0=\delta_0 \otimes \id,
\end{equation*}
where $\delta_0$ is the Dirac delta function on $E$.
\end{defn}

Let $\D(E)$ denote the Weyl algebra of $E$. Then, as vector spaces, we have
\begin{equation*}
\D(E) \simeq \Pol(E)\otimes \Diff^\const(E) = S(E^\vee) \otimes  \Diff^\const(E).
\end{equation*}
Via this isomorphism, we regard the map $\Trun_0$ as
\begin{equation*}
\Trun_0\colon \D(E) \To \Diff^\const(E).
\end{equation*}
We then call $\Trun_0$ the \emph{truncation map of $\D(E)$}.

\begin{example}\label{example:xy}
If 
$D=\partial_1\partial_2-\frac{1}{2}x_1\partial_1\partial_3
+\frac{1}{2}x_2\partial_2\partial_3-\frac{1}{4}x_1x_2\partial_3^2$, then
$\Trun_0(D)=\partial_1\partial_2$.
\end{example}

\begin{rem}
Since any change of basis for $E$ does not change
the degrees of monomials and order of partial derivatives, 
the truncation map $\Trun_0$ is well-defined for any basis for $E$.
For explicit computations of change of basis, see, for instance, 
\cite{CS96}.
\end{rem}

Now observe that we have 
\begin{equation*}
\D(E)\stackrel{\Trun_0}{\To}\Diff^\const(E)\stackrel[\sim]{\Symb}{\To}\Pol(E^\vee).
\end{equation*}
Write
\begin{equation*}
\Symb_0:=\Symb\circ \Trun_0,
\end{equation*}
so that the following diagram commutes.
\begin{equation}\label{eqn:diagram2}
\begin{tikzcd}[row sep=1cm, column sep=1cm]
\mathrm{Pol}(E^\vee)
& 
\arrow[dl, pos=0.85, phantom, "\circlearrowleft"]
\\
\D(E) 
\arrow[u, "\Symb_0"]
\arrow[r, "\Trun_0"']
& \Diff^{\const}(E)
\arrow[lu, "\Symb"',  "\sim"' sloped]
\end{tikzcd}
\end{equation}
We call $\Symb_0$ the \emph{truncated symbol map}.

\subsection{The case of a simply-connected, connected, nilpotent Lie group}\label{sec:nilp}

Now, let $N$ be 
a simply-connected, connected, nilpotent Lie group with Lie algebra $\fn(\R)$.
In the following, we resume the notation in Sections \ref{sec:prelim2} and  
\ref{sec:symb0} for the nilpotent group $N$, unless otherwise specified.

In this case 
the exponential map $\exp\colon \fn(\R) \stackrel{\sim}{\to} N$ 
is a diffeomorphism (cf.\ \cite[Thm.\ 1.127]{Knapp02}). 
Thus, via the local coordinates,
\begin{equation}\label{eqn:diffeo}
\R^n \stackrel{\sim}{\To} N,
\quad
(x_1, \ldots, x_n) \longmapsto \exp(x_1X_1+ \cdots + x_n X_n),
\end{equation}
the space $\Diff_N(N)$ can be understood as $\Diff_N(N) \subset  \D(\fn)$.
Then we write
\begin{equation*}
\Trun_0\big\vert_N:=\Trun_0\big\vert_{\Diff_N(N)}
\quad \text{and} \quad
\Symb_0\big\vert_N:=\Symb\circ \Trun_0\big\vert_N,
\end{equation*}
so that one can add the information about $\Diff_N(N)$ to
the diagram \eqref{eqn:diagram2} for $E=\fn$ as follows.
\begin{equation}\label{eqn:diagram3}
\begin{tikzcd}[row sep=1cm, column sep=1cm]
\Symb_0(\Diff_N(N)) 
\ar[draw=none]{d}{\bigcap}&\\[-15pt]
\mathrm{Pol}(\fn^\vee)
& 
\arrow[dl, pos=0.85, phantom, "\circlearrowleft"]
\\
\D(\fn) 
\arrow[u, "\Symb_0"]
\arrow[r, "\Trun_0"']
\ar[draw=none]{d}{\bigcup}
& 
\Diff^{\const}(\fn)
\arrow[lu, "\Symb"',  "\sim"' sloped]
\ar[draw=none]{d}{\bigcup}\\[-15pt]
\Diff_N(N)
\arrow[uuu, bend left=60, "\Symb_0\vert_{N}"]
\arrow[r, "\Trun_0\vert_{N}"']
&
\Trun_0(\Diff_N(N))
\end{tikzcd}
\end{equation}
Meanwhile, as in \eqref{eqn:diagram1}, the following diagram also commutes.
\begin{equation}\label{eqn:diagram4}
\begin{tikzcd}[row sep=1cm, column sep=1cm]
\mathrm{Pol}(\fn^\vee)
& 
\arrow[dl, pos=0.85, phantom, "\circlearrowleft"]
\\
\Diff_N(N) 
\arrow[u, "\Symb_N",  "\sim"' sloped]
\arrow[r, "\phi"', "\sim" sloped]
& \Diff^{\const}(\fn)
\arrow[lu, "\Symb"',  "\sim"' sloped]
\end{tikzcd}
\end{equation}
By comparing \eqref{eqn:diagram3} with \eqref{eqn:diagram4}, we have 
\begin{equation}\label{eqn:symbN}
\begin{tikzcd}[row sep=1.5cm, column sep=1.5cm]
\mathrm{Pol}(\fn^\vee)
& 
\arrow[dl, pos=0.85, phantom, "\circlearrowleft"]
\\
\Diff_N(N) 
\arrow[u, shift left=0.7ex, "\Symb_N"]
\arrow[u, shift right=0.7ex,"\Symb_0\vert_N"']
\arrow[r, shift right=0.5ex,"\phi"']
\arrow[r, shift left=0.5ex, "\Trun_0\vert_N"]
& \Diff^{\const}(\fn)
\arrow[lu, "\sim"' sloped,"\Symb"']
\end{tikzcd}
\end{equation}

We claim that $\Symb_N=\Symb_0\vert_N$. To prove it,
we first show the following lemma.
For $\mathbf{r}=(r_1,\ldots,r_n) \in \Z_{\geq 0}^n$, write
$|\mathbf{r}|:=r_1+\cdots+r_n$
and 
$\frac{\partial^{|\mathbf{r}|}}{\partial x^{\mathbf{r}}}:=
\frac{\partial^{|\mathbf{r}|}}{\partial x_1^{r_1}\cdots \partial x_n^{r_n}}$.

\begin{lem}\label{lem:trun}
If $D_1$ and $D_2$ are differential operators with constant coefficients
on $\R^n$ such that
\begin{equation}\label{eqn:trun}
\IP{\delta_0}{D_1 f} = \IP{\delta_0}{D_2 f} 
\quad \text{for all $f \in C^\infty(\R^n)$},
\end{equation}
then $D_1 = D_2$, where $\IP{\cdot}{\cdot}$ denotes the natural paring of 
$\D'_0(\R^n) \times C^\infty(\R^n)$ and $\delta_0$ is the Dirac delta function.
\end{lem}

\begin{proof}
Write $D_j:=\sum_{\mathbf{r}\in \Z_{\geq 0}^n}a^j_{\mathbf{r}}
\frac{\partial^{|\mathbf{r}|}}{\partial x^{\mathbf{r}}}$ 
with $a^j_{\mathbf{r}} \in \C$ for $j=1,2$.
Then, given $\mathbf{r} = (r_1, \ldots, r_n) \in \Z_{\geq 0}^n$, 
for $f_{\mathbf{r}}(x):=(r_1!\cdots r_n!)^{-1}x_1^{r_1}\cdots x_n^{r_n}$,
we have 
\begin{align*}
a^1_{\mathbf{r}}
=\IP{\delta_0}{D_1  f_{\mathbf{r}}(x)}
=\IP{\delta_0}{D_2  f_{\mathbf{r}}(x)}
=a^2_{\mathbf{r}}.
\end{align*}
Now the lemma follows.
\end{proof}

\begin{prop}\label{prop:phi}
We have
\begin{equation*}
\Symb_N=\Symb_0\vert_N.
\end{equation*}
\end{prop}

\begin{proof}
It follows from \eqref{eqn:symbN} that 
it suffices to show $\phi=\Trun_0\vert_N$.
Take $D \in \Diff_N(N)$. Then, by Theorem \ref{thm:Hel00}, 
there exists a unique $P(\zeta_1,\ldots,\zeta_n)\in \Pol(\fn^\vee) =S(\fn)$ such that $\Symb_N(D)=P(\zeta_1,\ldots,\zeta_n)$.
Then the differential operator $\phi(D) \in \Diff^\const(\fn)$ is given by
\begin{equation*}
\phi(D)=P(\partial_1, \ldots, \partial_n),
\end{equation*}
which satisfies
\begin{align*}
(Df)(e)
&=P(\partial_1,\ldots, \partial_n)f(\exp(x_1X_1+\cdots + x_nX_n))\big\vert_{x=0}\\[3pt]
&\equiv P(\partial_1,\ldots, \partial_n)f(0,\ldots,0)\\[3pt]
&=\phi(D)f(0,\ldots,0)
\end{align*}
for $f(n) \in C^\infty(N)\simeq C^\infty(\R^n)$.

On the other hand, via the local coordinates \eqref{eqn:diffeo}, the 
differential operator $D\in \Diff_N(N)\subset \D(\fn)$ can be realized as
\begin{equation*}
D=
\sum_{\mathbf{r} \in \Z^n_{\geq 0}}
a_\mathbf{r}(x)\frac{\partial^{|\mathbf{r}|}}{\partial x^{\mathbf{r}}}.
\end{equation*}
Then $\Trun_0(D) \in  \Diff^\const(\fn)$ is 
\begin{equation*}
\Trun_0(D)=\sum_{\mathbf{r} \in \Z^n_{\geq 0}}
a_\mathbf{r}(0)\frac{\partial^{|\mathbf{r}|}}{\partial x^{\mathbf{r}}}.
\end{equation*}
Observe that, in the local coordinates \eqref{eqn:diffeo}, we have 
\begin{align*}
(Df)(e)
&=\big(\sum_{\mathbf{r} \in \Z^n_{\geq 0}}
a_\mathbf{r}(x)\frac{\partial^{|\mathbf{r}|}}{\partial x^{\mathbf{r}}}\big)
f(x_1,\ldots, x_n)\big\vert_{x=0}\\[3pt]
&=\big(\sum_{\mathbf{r} \in \Z^n_{\geq 0}}
a_\mathbf{r}(0)\frac{\partial^{|\mathbf{r}|}}{\partial x^{\mathbf{r}}}\big)
f(0,\ldots,0)\\[3pt]
&=\Trun_0(D)f(0,\ldots,0).
\end{align*}
Thus,
\begin{equation*}
\Trun_0(D)f(0,\ldots,0) = Df(e) = \phi(D)f(0,\ldots,0)
\end{equation*}
for all $f \in C^\infty(\R^n) \simeq C^\infty(N)$.
Now Lemma \ref{lem:trun} concludes the proposition.
\end{proof}

Proposition \ref{prop:phi} shows that the truncated symbol map
$\Symb_0$ for the Weyl algebra $\D(E)$ gives rise to a linear isomorphism
\begin{equation*}
\Symb_0\big\vert_{N}\colon \Diff_N(N) \stackrel{\sim}{\To}\Pol(\fn^\vee),
\end{equation*}
which satisfies the following commutative diagram.
\begin{equation}\label{eqn:diagram5}
\begin{tikzcd}[row sep=1cm, column sep=1cm]
\mathrm{Pol}(\fn^\vee)
& 
\arrow[dl, pos=0.85, phantom, "\circlearrowleft"]
\\
\Diff_N(N) 
\arrow[u, "\Symb_0\vert_N",  "\sim"' sloped]
\arrow[r, "\Trun_0\vert_N"', "\sim" sloped]
& \Diff^{\const}(\fn)
\arrow[lu, "\Symb"',  "\sim"' sloped]
\end{tikzcd}
\end{equation}

Recall from \eqref{eqn:symm} that 
$\sym_{|\mathbf{r}|}(X_1^{r_1}\cdots X_n^{r_n})$ denotes the symmetrization of 
the product $X_1^{r_1}\cdots X_n^{r_n} \in \Cal{U}(\fn)$.
The next lemma will simplify later arguments.

\begin{lem}\label{lem:pr}
We have 
\begin{equation*}
\Trun_0(dR(\sym_{|\mathbf{r}|}(X_1^{r_1}\cdots X_n^{r_n})))
=\partial_1^{r_1}\cdots \partial_n^{r_n}.
\end{equation*}
\end{lem}

\begin{proof}
It follows from \eqref{eqn:symG} with $G=N$ and Proposition \ref{prop:phi} that 
\begin{equation}\label{eqn:symb0}
dR(\sym_{|\mathbf{r}|}(X_1^{r_1}\cdots X_n^{r_n}))
=
(\Symb_0\vert_{N})^{-1}(\zeta_1^{r_1}\cdots \zeta_n^{r_n}).
\end{equation}
Since 
\begin{equation*}
\Trun_0\big\vert_N = \Symb^{-1}\circ \Symb_0\big\vert_N,
\end{equation*}
we have
\begin{align*}
\Trun_0(dR(\sym_{|\mathbf{r}|}(X_1^{r_1}\cdots X_n^{r_n})))
&=\Symb^{-1}\circ \Symb_0\big\vert_N \circ (\Symb_0\big\vert_N)^{-1}
(\zeta_1^{r_1}\cdots \zeta_n^{r_n})\\[3pt]
&=\Symb^{-1}(\zeta_1^{r_1}\cdots \zeta_n^{r_n})\\[3pt]
&=\partial_1^{r_1}\cdots \partial_n^{r_n},
\end{align*}
which shows the desired identity.
\end{proof}

\begin{rem}\label{rem:symb}
If $\fn$ is abelian, then $\Symb_0\vert_{N} = \Symb$.
Indeed, in this case, we have 
$\Diff_N(N)=\Diff^\const(\fn)$ as
the differential operators $dR(X_j)$ are 
$dR(X_j) = \frac{\partial}{\partial x_j}$ 
via the diffeomorphism \eqref{eqn:diffeo}
and also as $dR\colon \Cal{U}(\fn) \stackrel{\sim}{\to} \Diff_N(N)$
is an algebra isomorphism (cf.\ \cite[Ch.\ II, Prop.\ 1.9]{Hel01}).
\end{rem}

\subsection{The inverse of the algebraic Fourier transform $F_c$}\label{sec:Fc}

In the rest of this section, we resume the notation from Section \ref{sec:Fmethod}.
For instance, the Lie group $G$ is real reductive and 
$P=MAN_+$ is a parabolic subgroup of $G$.
Via the bilinear form $\kappa$, we identify $\fn_+ \simeq \fn_-^\vee$,
yielding $\Pol(\fn_-^\vee)\simeq \Pol(\fn_+)$.

As $N_-$ is a simply-connected, connected, nilpotent Lie group, the arguments
in Section \ref{sec:nilp} can be applied to $N_-$; in particular,
the following diagram commutes.
\begin{equation*}
\begin{tikzcd}[row sep=1cm, column sep=1cm]
\mathrm{Pol}(\fn_+)
& 
\arrow[dl, pos=0.85, phantom, "\circlearrowleft"]
\\
\Diff_{N_-}(N_-) 
\arrow[u, "\Symb_0\vert_{N_-}",  "\sim"' sloped]
\arrow[r, "\Trun_0\vert_{N_-}"', "\sim" sloped]
& \Diff^{\const}(\fn_-)
\arrow[lu, "\Symb"',  "\sim"' sloped]
\end{tikzcd}
\end{equation*}

Recall from \eqref{eqn:FcL0} and \eqref{eqn:FcL} that the algebraic 
Fourier transform $F_c$ satisfies
\begin{equation}\label{eqn:FcL1}
F_c\colon \Cal{U}(\fn_-) \otimes V^\vee \stackrel{\sim}{\To} \Pol(\fn_+)\otimes V^\vee,
\; 
u\otimes v^\vee
\longmapsto
\widehat{dL}(u)_1\otimes v^\vee
\end{equation}
with $\widehat{dL}(u)_1:=\widehat{dL}(u)\acts 1$.
Our goal is to determine $F_c^{-1}$. 
It follows from \eqref{eqn:FcL1} that, 
for each $\zeta_1^{r_1}\cdots \zeta_n^{r_n} \in \Pol(\fn_+)$,
it suffices to determine $u_{\mathbf{r}} \in \Cal{U}(\fn_-)$ such that 
\begin{equation*}
\hdL(u_{\mathbf{r}})\acts 1 = \zeta_1^{r_1}\cdots \zeta_n^{r_n}.
\end{equation*}

To this end, we first prove one technical lemma, 
which relates $dL$ with $dR$.
For $u = Z_1\cdots Z_\ell$ with $Z_j \in \fn_-$, write
\begin{equation*}
u^\circ:=(-1)^\ell Z_\ell \cdots Z_1.
\end{equation*}
Under the identification $N_- \simeq \R^n$ via the local coordinates
\eqref{eqn:diffeo}, the distribution $dL(u)\delta_0$ is defined by
\begin{equation*}
\IP{dL(u)\delta_0}{f} := \IP{\delta_0}{dL(u^\circ) f}
\quad 
\text{for $f \in C^\infty(N_-)$}.
\end{equation*}

\begin{lem}\label{lem:dLR}
For $u\in \Cal{U}(\fn_-)$, we have 
\begin{equation*}
\IP{dL(u)\delta_0}{f} = \IP{\delta_0}{dR(u) f}
\quad 
\text{for $f \in C^\infty(N_-)$}.
\end{equation*}
\end{lem}

\begin{proof}
Observe that $dL(Z)f(e) = -dR(Z)f(e)$ for 
$Z \in \fn_-$ and $f \in C^\infty(N_-)$. Thus,
for 
$f \in C^\infty(N_-)$ and 
$u = Z_1\cdots Z_\ell$ with $Z_j \in \fn_-$, we have
\begin{align*}
\IP{dL(u)\delta_0}{f}
&=\IP{\delta_0}{dL(u^\circ)f}\\[3pt]
&=(-1)^\ell\IP{\delta_0}{dL(Z_\ell)\cdots dL(Z_1) f}\\[3pt]
&=(-1)^\ell dL(Z_\ell)\cdots dL(Z_1)f(e)\\[3pt]
&=dR(Z_1)\cdots dR(Z_\ell)f(e)\\[3pt]
&=\IP{\delta_0}{dR(u)f},
\end{align*}
which shows the lemma.
\end{proof}

\begin{thm}\label{thm:hdL}
We have 
\begin{equation}\label{eqn:hdL1}
\hdL(\sym_{|\mathbf{r}|}(X_1^{r_1}\cdots X_n^{r_n}))\acts 1
=\zeta_1^{r_1}\cdots \zeta_n^{r_n}.
\end{equation}
Consequently, the following identity holds on $\Cal{U}(\fn_-)$:
\begin{equation}\label{eqn:hdL2}
\hdL(\cdot)_1 = \Symb_0\big\vert_{N_-} \circ dR.
\end{equation}
\end{thm}

\begin{proof}
We start with a proof of \eqref{eqn:hdL1}.
It follows from Lemmas \ref{lem:pr} and \ref{lem:dLR} that,
for  $f\in C^\infty(N_-)\simeq C^\infty(\R^n)$, we have 
\begin{align}\label{eqn:dL}
\IP{dL(\sym_{|\mathbf{r}|}(X_1^{r_1}\cdots X_n^{r_n}))\delta_0}{f}
&=\IP{\delta_0}{dR(\sym_{|\mathbf{r}|}(X_1^{r_1}\cdots X_n^{r_n}))f} \nonumber\\[3pt]
&=\IP{\delta_0}{\Trun_0(dR(\sym_{|\mathbf{r}|}(X_1^{r_1}\cdots X_n^{r_n})))f}\nonumber\\[3pt]
&=\IP{\delta_0}{\partial_1^{r_1}\cdots \partial_n^{r_n}f}\nonumber\\[3pt]
&=\IP{(-1)^{|r|}\partial_1^{r_1}\cdots \partial_n^{r_n}\delta_0}{f}.
\end{align}
Therefore,
\begin{equation*}
dL(\sym_{|\mathbf{r}|}(X_1^{r_1}\cdots X_n^{r_n}))\delta_0
=(-1)^{|r|}\partial_1^{r_1}\cdots \partial_n^{r_n}\delta_0,
\end{equation*}
which yields
\begin{equation*}
\hdL(\sym_{|\mathbf{r}|}(X_1^{r_1}\cdots X_n^{r_n}))\acts 1
=(-1)^{|r|}\widehat{\partial_1^{r_1}\cdots \partial_n^{r_n}}\acts 1
=\zeta_1^{r_1}\cdots \zeta_n^{r_n}.
\end{equation*}

To show \eqref{eqn:hdL2},
let $S^j(\fn_-)$ be the subspace of the symmetric algebra
$S(\fn_-)$ spanned by the products of $j$ elements of $\fn_-$.
It is known that, for any $u \in \Cal{U}(\fn_-)$, there exists $k \in \Z_{\geq 0}$
such that $u \in \bigoplus_{j=0}^k \textbf{sym}_j(S^j(\fn_-))$,
where 
$\textbf{sym}_j\colon S^j(\fn_-) \hookrightarrow \Cal{U}(\fn_-)$,
$Y_1\cdots Y_j \mapsto \sym_j(Y_1\cdots Y_j)$,
is the symmetrization map from $S^j(\fn_-)$ to $\Cal{U}(\fn_-)$
(cf.\ \cite[p.\ 225]{Knapp02}).
Now \eqref{eqn:symb0} and \eqref{eqn:hdL1} conclude the desired identity.
\end{proof}

By Theorem \ref{thm:hdL}, the following diagram commutes.
\begin{equation*}
\begin{tikzcd}[row sep=1cm, column sep=1cm]
&
\mathrm{Pol}(\fn_+)
& 
\arrow[dl, pos=0.7, phantom, "\circlearrowleft"]
\\
\Cal{U}(\fn_-)
\arrow[ur, pos=0.2, phantom, "\circlearrowleft"]
\arrow[ur, bend left=30,"\hdL(\cdot)_1", "\sim"' sloped]
\arrow[r, "dR"', "\sim" sloped]
&
\Diff_{N_-}(N_-) 
\arrow[u, "\Symb_0\vert_{N_-}",  "\sim"' sloped]
\arrow[r, "\Trun_0\vert_{N_-}"', "\sim" sloped]
& \Diff^{\const}(\fn_-)
\arrow[lu,bend right=30, "\Symb"',  "\sim"' sloped]
\end{tikzcd}
\end{equation*}

Now we determine the inverse of the algebraic Fourier transform $F_c^{-1}$.

\begin{cor}\label{cor:Fc}
The inverse 
\begin{equation*}
F_c^{-1}\colon \Pol(\fn_+)\otimes V^\vee
\stackrel{\sim}{\To} 
\Cal{U}(\fn_-)\otimes V^\vee
\end{equation*}
of the algebraic Fourier transform $F_c$ satisfies
\begin{equation*}
F_c^{-1}(\zeta_1^{r_1}\cdots \zeta_n^{r_n}\otimes v^\vee) =
\sym_{|\mathbf{r}|}(X_1^{r_1}\cdots X_n^{r_n})\otimes v^\vee.
\end{equation*}
\end{cor}

\begin{proof}
This is an immediate consequence of \eqref{eqn:FcL1} and Theorem \ref{thm:hdL}.
\end{proof}

\begin{rem}
Write $X^{\mathbf{r}}:=X_1^{r_1}\cdots X_n^{r_n}$ and 
$\zeta^{\mathbf{r}}:=\zeta_1^{r_1}\cdots \zeta_n^{r_n}$.
It follows from Theorem \ref{thm:hdL} that 
\begin{equation*}
\widehat{dL}(\sym_{|\mathbf{r}|}(X^{\mathbf{r}}))(1\otimes v^\vee)
=\zeta^{\mathbf{r}}\otimes v^\vee.
\end{equation*}
Also, we have 
\begin{equation*}
\sym_{|\mathbf{r}|}(X^{\mathbf{r}})\otimes v^\vee=
\sym_{|\mathbf{r}|}(X^{\mathbf{r}})(1\otimes v^\vee).
\end{equation*}
Then $F_c^{-1}$ maps
$\zeta^{\mathbf{r}}\otimes v^\vee$ to 
$\sym_{|\mathbf{r}|}(X^{\mathbf{r}})\otimes v^\vee$ as follows.
\begin{equation*}
\begin{adjustbox}{max width=\textwidth}
$\displaystyle
\begin{alignedat}{4}
\Pol(\fn_+)\otimes V^\vee
&
\stackrel{\sim}{\relbar\joinrel\relbar\joinrel\relbar\joinrel\relbar\joinrel\To} 
\hspace{5pt}
\Cal{D}'_{0}(\fn_-(\R),V^\vee)
&&
 \hspace{5pt}
\stackrel{\sim}{\relbar\joinrel\relbar\joinrel\relbar\joinrel\relbar\joinrel\To} 
 \hspace{5pt}
\Cal{D}'_{[o]}(G/P,\Cal{V}^*)
&&\hspace{5pt}
\stackrel{\sim}{\relbar\joinrel\relbar\joinrel\relbar\joinrel\relbar\joinrel\relbar\joinrel\relbar\joinrel\To} 
 \hspace{5pt}
 \Cal{U}(\fn_-)\otimes V^\vee
\\
\rotatebox{90}{$\in$}\hspace{20pt}
&\hspace{80pt}\rotatebox{90}{$\in$}
&&\hspace{80pt}\rotatebox{90}{$\in$}
&&\hspace{90pt}\rotatebox{90}{$\in$}
\\[-2pt]
\zeta^{\mathbf{r}}\otimes v^\vee\hspace{8pt}
&\xmapsto{\hspace{20pt}}  \hspace{5pt}dL(\sym_{|\mathbf{r}|}(X^{\mathbf{r}}))
   (\delta_0&&\otimes v^\vee)
\longmapsto
dL(\sym_{|\mathbf{r}|}(X^{\mathbf{r}}))
   &&(\delta_{[o]}\otimes v^\vee)
\longmapsto\hspace{5pt}
\sym_{|\mathbf{r}|}(X^{\mathbf{r}})\otimes v^\vee
\end{alignedat}
$
\end{adjustbox}
\end{equation*}
\end{rem}

\subsection{Truncated symbol map $\Symb_0$ 
for  $\Diff_{G'}(\Cal{V}, \Cal{W})$}
\label{sec:symb2b}

For finite-dimensional irreducible 
representations $V$ and $W$ of $P$ and $P'$,
respectively, 
the truncated symbol map 
$\Symb_0\big\vert_{N_-}\ \colon\Diff_{N_-}(N_-) \stackrel{\sim}{\to} \Pol(\fn_+)$ 
gives rise to a linear isomorphism
\begin{equation*}
\Symb_0\big\vert_{N_-}\otimes \id\colon
\Diff_{N_-}(N_-)
 \otimes \operatorname{Hom}_{\C}(V,W)
\stackrel{\sim}{\To}
\Pol(\fn_+)
\otimes \operatorname{Hom}_{\C}(V,W),
\end{equation*}
where $\id$ is the identity map on $\Hom_{\C}(V,W)$.
Since $\Diff_{G'}(\Cal{V},\Cal{W})$ can be understood as
\begin{equation*}
\Diff_{G'}(\Cal{V},\Cal{W}) \subset \Diff_{N_-}(N_-)
 \otimes \operatorname{Hom}_{\C}(V,W),
\end{equation*}
by abuse of notation, we write
\begin{equation*}
\Symb_0
\equiv(\Symb_0\big\vert_{N_-}\otimes \id)
\big\vert_{\Diff_{G'}(\Cal{V},\Cal{W})}.
\end{equation*}

Observe that we have
\begin{alignat}{3}\label{eqn:const2}
\Symb_0\big\vert_{N_-}\otimes \id\colon&
\Diff_{N_-}(N_-)
 \otimes \operatorname{Hom}_{\C}(V,W)
\stackrel{\sim}{\To} &&\;
\Pol(\fn_+)
\otimes \operatorname{Hom}_{\C}(V,W) \nonumber \\[-3pt]
&\text{\hspace{2.07cm}{\footnotesize{$\bigcup$}}}&& 
\hspace{1.5cm} 
\text{{\footnotesize{$\bigcup$}}} \\[-1pt]
\Symb_0\colon&
\hspace{1.3cm}
\textnormal{Diff}_{G'}(\mathcal{V}, \mathcal{W}) 
&&
\hspace{-0.7cm}
\stackrel{\sim}{\To}
\hspace{0.5cm} 
\Symb_0(\textnormal{Diff}_{G'}(\mathcal{V}, \mathcal{W})).
 \hspace{0.5cm}\nonumber
\end{alignat}
On the other hand, by \eqref{eqn:Sol2a}, we also have
\begin{equation*}
\mathrm{Sol}(\mathfrak n_+;V,W)
\subset 
\Pol(\fn_+)
\otimes \operatorname{Hom}_{\C}(V,W).
\end{equation*}
Hence, the following holds:
\begin{equation*}
\xymatrix@R=-4pt@C=-35pt{
& 
\hspace{-15pt}
\mathrm{Pol}(\mathfrak{n}_+) \otimes \operatorname{Hom}_{\C}(V,W)
&\\
\rotatebox[origin=c]{-45}{$\bigcup$}
&
&
\rotatebox[origin=c]{45}{$\bigcup$}
\\
\Symb_0(\textnormal{Diff}_{G'}(\mathcal{V}, \mathcal{W}))
 \hspace{10pt}
&
&
\hspace{10pt}
\mathrm{Sol}(\mathfrak n_+;V,W)
}
\end{equation*}
We claim that
\begin{equation*}
\Symb_0(\textnormal{Diff}_{G'}(\mathcal{V}, \mathcal{W}))
= 
\mathrm{Sol}(\mathfrak n_+;V,W).
\end{equation*}

\begin{thm}\label{thm:symb}
We have a linear isomorphism
\begin{equation*}
\Symb_0 \colon \Diff_{G'}(\Cal{V}, \Cal{W}) \stackrel{\sim}{\To}
\Sol(\fn_+;V,W).
\end{equation*}
\end{thm}

\begin{proof}
It follows from \eqref{eqn:HD} and \eqref{eqn:hdL2} that we have 
\begin{equation}\label{eqn:symbFD}
\Symb_0 = (F_c\otimes \id_W) \circ \EuD^{-1},
\end{equation}
where $\EuD$ is the linear isomorphism 
from 
$\operatorname{Hom}_{P'}(W^\vee,\Mp(V^\vee))$
to 
$\operatorname{Diff}_{G'}(\mathcal V, \mathcal W)$
in \eqref{eqn:duality2}.
Theorem \ref{thm:Fmethod} then concludes the assertion.
\end{proof}

Let $\Rest$ denote the restriction map 
from $C^\infty(G/P',\Cal{W})$ to $C^\infty(G'/P',\Cal{W})$.

\begin{cor}\label{cor:symb}
The map $\Rest \circ \Symb_0^{-1}$ provides a linear isomorphism
\begin{equation*}
\Rest \circ \Symb_0^{-1} \colon 
\Sol(\fn_+;V,W) \stackrel{\sim}{\To} 
\operatorname{Diff}_{G'}(\mathcal V, \mathcal W)
\end{equation*}
such that the following diagram commutes.
\begin{equation}\label{eqn:isom3}
\begin{tikzcd}[row sep=1.2cm, column sep=0.7cm]
& 
\mathrm{Sol}(\mathfrak n_+;V,W)
\arrow[dr,  "\sim"' sloped, "\Rest\, \circ \, \Symb_0^{-1}"] \\
\operatorname{Hom}_{\mathfrak{g}',P'}(\Mpp(W^\vee),\Mp(V^\vee)) 
  \arrow[rr, "\sim", "\EuScript{D}"']
  \ar[ur, "\sim"' sloped, "F_c\otimes\mathrm{id}_W"] 
 &
 \arrow[u,  pos=0.5, phantom, "\circlearrowleft"]
 & 
 \operatorname{Diff}_{G'}(\mathcal V, \mathcal W).
\end{tikzcd}
\end{equation}
\end{cor}

\begin{proof}
This readily follows from the proof of Theorem \ref{thm:symb}.
\end{proof}

It follows from Corollary \ref{cor:symb} that
the following diagram commutes.
\begin{equation}\label{eqn:diagramKP}
\begin{tikzcd}[row sep=1cm, column sep=1.5cm]
\Hom_\C(W^\vee, \Mp(V^\vee))
\ar[r, "\sim"' sloped, "F_c\otimes\mathrm{id}_W"] 
\ar[draw=none]{d}{\bigcup}
&
\mathrm{Pol}(\fn_+)\otimes\Hom_{\C}(V,W)
& 
\Diff_{N_-}(N_-)\otimes \Hom_{\C}(V,W)
\ar[l, "\sim"' sloped, "\Symb_0\vert_{N_-}\otimes \id"'] 
\ar[draw=none]{d}{\bigcup}
\\[-5pt]
\operatorname{Hom}_{P'}(W^\vee,\Mp(V^\vee)) 
  \arrow[rr, "\sim", "\EuScript{D}"']
 &
 \arrow[u,  pos=0.5, phantom, "\circlearrowleft"]
 & 
 \operatorname{Diff}_{G'}(\mathcal V, \mathcal W)
\end{tikzcd}
\end{equation}

\begin{rem}
It follows from Remark \ref{rem:symb} that if $\fn_{\pm}$ are abelian,
then
$\Symb_0 =(\Symb\otimes \id)
\big\vert_{\Diff_{G'}(\Cal{V},\Cal{W})}$
and $\Diff_{N_-}(N_-) = \Diff^\const(\fn_-)$.
Thus, Corollary \ref{cor:symb}
and the diagram \eqref{eqn:diagramKP} 
generalize \cite[Cor.\ 4.3]{KP1}
and the diagram in \cite[Thm.\ 4.1 (2)]{KP1}, respectively.
\end{rem}
  
It follows from \eqref{eqn:symbFD} that we have
$\Symb_0^{-1}=\EuD \circ (F_c^{-1}\otimes \id_W)$.
Thus, the differential operator $\Symb_0^{-1}(\psi(\zeta))
\in \operatorname{Diff}_{G'}(\mathcal V, \mathcal W)$
for $\psi(\zeta) \in  \Pol(\fn_+)\otimes V^\vee\otimes W$ 
is obtained in the following way.
As $\psi(\zeta) \in  \Pol(\fn_+)\otimes V^\vee\otimes W$,
the polynomial $\psi(\zeta)$ is of the form 
\begin{equation*}
\psi(\zeta)=\sum_{i,j}\psi_{i,j}(\zeta) \otimes v_i^\vee \otimes w_j
\end{equation*}
for some $\psi_{i,j}(\zeta) \in \Pol(\fn_+)$, $v_i^\vee \in V^\vee$, and $w_j \in W$.
Write
\begin{equation*}
\psi_{i,j}(\zeta)=\sum_{\mathbf{r} \in \Z_{\geq 0}^n}a_\mathbf{r} \, 
\zeta_1^{r_1}\cdots \zeta_n^{r_n}
\end{equation*}
for $\mathbf{r} = (r_1, \ldots, r_n)$. 
We put
\begin{equation*}
\psi(\sym(X)):=\sum_{i,j}\psi_{i,j}(\sym(X)) \otimes v_i^\vee \otimes w_j
\in \Cal{U}(\fn_-)\otimes V^\vee \otimes W,
\end{equation*}
where
\begin{equation*}
\psi_{i,j}(\sym(X))
:=\sum_{\mathbf{r} \in \Z_{\geq 0}^n}a_\mathbf{r}\, 
\sym_{|\mathbf{r}|}(X_1^{r_1}\cdots X_n^{r_n}).
\end{equation*}
Then $(F_c^{-1} \otimes \id_W)(\psi(\zeta))$ and
$\Symb_0^{-1}(\psi(\zeta))$ are given by
\begin{align}
(F_c^{-1} \otimes \id_W)(\psi(\zeta))&=\psi(\sym(X)),
\label{eqn:Finv2}\\[3pt]
\Symb_0^{-1}(\psi(\zeta))&=\psi(dR(\sym(X))).\label{eqn:symb} 
\end{align}
Remark that the polynomial $\psi(\zeta)$ satisfies the following identities:
\begin{equation}\label{eqn:Finv3}
\begin{aligned}
\psi(\zeta) 
&= \Symb_0(\psi(dR(\sym(X)))\\[3pt]
&=(F_c \otimes \id_W)(\psi(\sym(X))) \\[3pt]
&=\sum_{i,j}\widehat{dL}(\psi_{i,j}(\sym(X)))_1\otimes v_i^\vee \otimes w_j.
\end{aligned}
\end{equation}

\subsection{A recipe for the F-method}
\label{sec:recipe}

By \eqref{eqn:Sol} and Corollary \ref{cor:symb},
one can classify and construct DSBOs
$\bD \in  \operatorname{Diff}_{G'}(\mathcal V, \mathcal W)$ and
$(\fg', P')$-homomorphisms $\Phi \in 
\Hom_{\mathfrak{g}',P'}(\Mpp(W^\vee),\Mp(V^\vee))$
by determining $\psi \in \mathrm{Sol}(\mathfrak n_+;V,W)$ as follows.
\vskip 0.1in

\begin{enumerate}

\item[Step 1]
Compute $d\pi_{(\xi,\lambda)^*}(C)$ and 
$\widehat{d\pi_{(\xi,\lambda)^*}}(C)$
for $C \in \fn_+'$.
\vskip 0.1in

\item[Step 2]
Classify and construct
$\psi \in \Hom_{M'A'}(W^\vee, \Pol(\fn_+)\otimes V^\vee)$.
\vskip 0.1in

\item[Step 3]
Solve the F-system \eqref{eqn:Fsys}
for $\psi \in \Hom_{M'A'}(W^\vee, \Pol(\fn_+)\otimes V^\vee)$.
\vskip 0.1in

\item[Step 4]
For $\psi \in \mathrm{Sol}(\mathfrak n_+;V,W)$ obtained in Step 3,
do the following.
\vskip 0.1in

\begin{enumerate}
\item[Step 4a]
Apply $\Rest \circ \Symb_0^{-1}$ 
to $\psi \in \mathrm{Sol}(\mathfrak n_+;V,W)$
to obtain $\bD \in  \operatorname{Diff}_{G'}(\mathcal V, \mathcal W)$.
\vskip 0.1in
\item[Step 4b]
Apply $F_c^{-1} \otimes \id_W$ 
to $\psi \in \mathrm{Sol}(\mathfrak n_+;V,W)$
to obtain $\Phi \in \operatorname{Hom}_{\mathfrak{g}',P'}
(\Mpp(W^\vee), \Mp(V^\vee))$.
\end{enumerate}

\vskip 0.2in

\end{enumerate}

In Sections \ref{sec:proof1} and \ref{sec:proof2},
we shall apply the recipe to classify and construct  
differential intertwining operators $\D$ and homomorphisms $\varphi$
between Verma modules
for $G=G'=SL(3,\R)$ with $P=P'=B$, where $B$ is
a minimal parabolic subgroup of $G$.

\section{DIOs $\D$ and $(\fg, B)$-homomorphisms $\varphi$ 
between Verma modules for $SL(3,\R)$}\label{sec:SL3}

The aim of this section is to 
classify and construct  DIOs $\D$ and $(\fg, B)$-homomorphisms 
$\varphi$ between Verma modules for 
 $G=SL(3,\R)$ and $P=B$, a minimal parabolic subgroup of $G$.
The classification is done in Theorems \ref{thm:DIO1} and 
\ref{thm:Hom1}, while 
the explicit formulas are given in Theorems \ref{thm:DIO2} and \ref{thm:Hom2}.
At the end of this section we also give the classification of 
$(\fg, B)$-homomorphisms $\varphi$ in terms of infinitesimal characters.
The proof of these results will be discussed in detail in 
Sections \ref{sec:proof1} and \ref{sec:proof2}.

\subsection{Notation and normalizations}
\label{subsec:SL3}

We start with some notation.
Let $G := SL(3,\R)$ with Lie algebra $\fg(\R):=\f{sl}(3,\R)$.
We put
\begin{alignat*}{3}
N_1^+:&=E_{1,2},
\quad
N_2^+&:=E_{2,3},
\quad
N_3^+&:=E_{1,3},\\[3pt]
N_1^-&:=E_{2,1},
\quad
N_2^-&:=E_{3,2},
\quad
N_3^-&:=E_{3,1},
\end{alignat*}
where $E_{i,j}$ denote matrix units.
We also set $\fa(\R):=\text{span}_\R\{E_{i,i}-E_{i+1, i+1}: i=1, 2\}$ 
and $\fn_\pm(\R):=\spn_\R\{N_1^\pm, N_2^\pm, N_3^\pm\}$.
Then $\fg(\R) = \fn_-(\R) \oplus \fa(\R) \oplus \fn_+(\R)$ is a Gelfand--Naimark 
decomposition of $\fg(\R)$ and
$\fb(\R):=\fa(\R) \oplus \fn_+(\R)$ is a minimal parabolic subalgebra of $\fg(\R)$.

Let $K:=SO(3)$ with Lie algebra $\fk(\R):=\f{so}(3)$ . 
Also, let $A$ and $N_+$ 
be the analytic subgroups of $G$ with Lie algebras 
$\fa(\R)$ and $\fn_+(\R)$, respectively.
Then $G=K AN_+$ is an Iwasawa decomposition of $G$.
We write $M := Z_{K}(\fa(\R))$, so that
$B:=MAN_+$ is a minimal parabolic subgroup of $G$ 
with Lie algebra $\fb(\R)$.

We denote by $\fg$ the complexification of the Lie algebra $\fg(\R)$ of $G$.
A similar convention is employed also for subgroups of $G$;
for instance, $\mathfrak{b}= \mathfrak{a} \oplus \mathfrak{n}_+$ 
is a Borel subalgebra of $\fg= \f{sl}(3,\C)$.

Let $\Delta\equiv \Delta(\mathfrak{g},\mathfrak{a})$ denote
the set of roots of $\mathfrak{g}$ with respect to $\mathfrak{a}$. 
We denote by $\Delta^+$ and $\Pi$ 
the positive system corresponding to $\fb$ and 
the set of simple roots of $\gD^+$, respectively.
Then we have 
\begin{equation*}
\gD^+=\{\gamma_1-\gamma_2, \gamma_2-\gamma_3, \gamma_1-\gamma_3\}
\quad \text{and} \quad
\Pi=\{\gamma_1-\gamma_2, \gamma_2-\gamma_3\},
\end{equation*}
where $\gamma_j$ is the dual basis element of $E_{j,j}$ for $j=1, 2, 3$.
The root spaces $\fg_{\pm(\gamma_i-\gamma_j)}$ for 
$1\leq i < j \leq 3$ are then given as 
\begin{equation*}
\fg_{\pm(\gamma_1-\gamma_2)} = \C N_1^{\pm},
\quad
\fg_{\pm(\gamma_2-\gamma_3)} = \C N_2^{\pm}, 
\quad
\fg_{\pm(\gamma_1-\gamma_3)} = \C N_3^{\pm}.
\end{equation*}
We write $\rho$ for half the sum of the positive roots, namely,
$\rho=\gamma_1 - \gamma_3$.

For a closed subgroup $J$ of $G$, we denote by $\Irr(J)$
the set of equivalence classes of irreducible representations of $J$.

\vskip 0.1in

Write 
\begin{equation}\label{eqn:H}
H_1 := 
\begin{pmatrix}1 & 0 & 0 \\ 0 & -1 & 0 \\ 0 & 0 & 0\end{pmatrix}
\quad
\text{and}
\quad
H_2 := 
\begin{pmatrix}0 & 0 & 0 \\ 0 & 1 & 0 \\ 0 & 0 & -1\end{pmatrix}.
\end{equation}
Then $\fa(\R) = \spn_{\R}\{ H_1, H_2\}$. For $\lambda_1, \lambda_2 \in \C$, 
we define a one-dimensional representation 
$\C_{(\lambda_1, \lambda_2)}:=(\chi^{(\lambda_1,\lambda_2)}_A, \C)$ 
of $A = \exp(\fa(\R))$ by
\begin{equation}\label{eqn:chi}
\chi_A^{(\lambda_1,\lambda_2)} \colon \exp(t_1 H_1 + t_2 H_2) 
\longmapsto \exp(\lambda_1t_1 + \lambda_2 t_2).
\end{equation}
Then $\Irr(A)$ is given as
\begin{equation*}
\Irr(A)=\{\C_{(\lambda_1, \lambda_2)} : \lambda_1, \lambda_2 \in \C\} \simeq \C^2.
\end{equation*}

Let
\begin{equation*}
m_0:=
\begin{pmatrix}
1& 0 & 0\\
0 & 1 & 0\\
0 & 0 & 1
\end{pmatrix}
,\;
m_1:=
\begin{pmatrix}
-1& 0 & 0\\
0 & 1 & 0\\
0 & 0 & -1
\end{pmatrix}
,\;
m_2:=
\begin{pmatrix}
1& 0 & 0\\
0 & -1 & 0\\
0 & 0 & -1
\end{pmatrix}
,\;
m_3:=
\begin{pmatrix}
-1& 0 & 0\\
0 & -1 & 0\\
0 & 0 & 1
\end{pmatrix}.
\end{equation*}
Then 
\begin{equation*}
M=\{m_0, m_1, m_2, m_3\}.
\end{equation*}
For $\eps, \eps' \in \{\pm\}$,
a one-dimensional $M$-representation 
$\C_{(\eps,\eps')}:=(\chi_M^{(\eps, \eps')}, \C)$ is defined by
a character $\chi^{(\eps,\eps')}_M\colon M \to \{\pm 1\}$ as
\begin{equation*}
\chi_M^{(\eps, \eps')}
(\mathrm{diag}(b_1, b_2, b_3))
:=|b_1|_{\eps}\;|b_3|_{\eps'},
\end{equation*}
where
$|b|_+:=|b|$ and $|b|_-:=b$.
Table \ref{table:char} summarizes the values of the character 
$(\eps,\eps') :=\chi_M^{(\eps,\eps')}$ at each $m_j \in M$.
The set $\Irr(M)$ of 
equivalence classes of irreducible representations of $M$ is given as
\begin{equation}\label{eqn:char}
\Irr(M)= \{ \pp, \pmi, \mip, \mm\} \simeq (\Z/2\Z)^2.
\end{equation}

\begin{table}
\caption{Character table for $(\eps,\eps')$}
\begin{center}
\begin{normalsize}
\renewcommand{\arraystretch}{1.2} 
{
\begin{tabular}{|c|c|c|c|c|}
\hline
& $m_0$ 
& $m_1$ 
& $m_2$ 
& $m_3$\\
\hline
$\pp$ & $1$ & $1$ & $1$ & $1$ \\
\hline
$\pmi$ & $1$  & $-1$  & $-1$ & $1$\\
\hline
$\mip$ & $1$ & $-1$ & $1$ & $-1$\\
\hline
$\mm$ & $1$ & $1$ & $-1$ & $-1$\\
\hline
\end{tabular}
\label{table:char}
}
\end{normalsize}
\end{center}
\end{table}%

Since $\Irr(B)\simeq \Irr(M) \times \Irr(A)$, the set $\Irr(B)$ can be 
parametrized as
\begin{equation*}
\Irr(B)\simeq (\Z/2\Z)^2 \times \C^2.
\end{equation*}
For $(\eps_1, \eps_2;\lambda_1,\lambda_2) \in \Irr(B)$,
we write
\begin{equation}\label{eqn:Ind}
I(\lambda_1,\lambda_2)^{(\eps_1,\eps_2)} 
:= 
\Ind_{B}^G\left(\C_{(\eps_1,\eps_2)} \boxtimes \C_{(\lambda_1,\lambda_2)}\right)
\end{equation}
for an unnormalized parabolically induced representation 
(a principal series representation) of $G$.
Also, as a $(\fg, B)$-module, we denote by 
\begin{equation}\label{eqn:Verma}
M(\lambda_1,\lambda_2)^{(\eps_1,\eps_2)} 
:=\Cal{U}(\fg) \otimes_{\Cal{U}(\fb)}
(\C_{(\eps_1, \eps_2)}\boxtimes \C_{(\lambda_1,\lambda_2)})
\end{equation}
the Verma module induced from 
$\C_{(\eps_1, \eps_2)}\boxtimes \C_{(\lambda_1,\lambda_2)}$.
When we regard it merely as a $\fg$-module,
we simply write $M(\lambda_1,\lambda_2)$. 

We shall classify and construct  differential intertwining  operators 
\begin{equation*}
\D \in 
\Diff_G(I(\lambda_1, \lambda_2)^{(\eps_1, \eps_2)}, I(\nu_1,\nu_2)^{(\delta_1,\delta_2)})
\end{equation*}
and $(\fg, B)$-homomorphisms
\begin{equation*}
\varphi \in 
\Hom_{\fg, B}(
M(-\nu_1,-\nu_2)^{(\delta_1, \delta_2)},
M(-\lambda_1, -\lambda_2)^{(\eps_1, \eps_2)})
\end{equation*}
in the next subsections.

\subsection{Classification of DIOs $\D$ and $(\fg, B)$-homomorphisms $\varphi$}
\label{sec:DIO}

We put
\begin{align*}
\Lambda_1&:=(-\Z_{\geq 0})\times \C, \\[3pt]
 \Lambda_2&:=\C \times (-\Z_{\geq 0}),
\end{align*}
\begin{align*}
 \Lambda_+&:=\{(1-k, 1-\ell+k) : k, \ell \in 1+\Z_{\geq 0}\}, \\[3pt]
 \Lambda_-&:=\{(1-k+\ell, 1-\ell) : k, \ell \in 1+\Z_{\geq 0}\}, 
\end{align*}
\begin{align*}
\Lambda_{\Cay}&:=\{(\tfrac{1}{2}(2-k-s), \tfrac{1}{2}(2-k+s)) : k \in 1+\Z_{\geq 0} \; 
\text{and} \;  s \in \C\}.
\end{align*}
Remark that, for instance, the set $ \Lambda_+$ can also be defined as 
$ \Lambda_+=\{(-k, 1-\ell+k) : k, \ell \in \Z_{\geq 0}\}$. 
We defined $ \Lambda_+$, $ \Lambda_-$,
and $\Lambda_{\Cay}$ as above for convenience to later arguments
(see, for instance, Sections \ref{sec:Fsys1} and \ref{subsec:Verma2}).

Write
$\eps:= (\eps_1, \eps_2),\,
\delta:= (\delta_1, \delta_2)
\in (\Z/2\Z)^2$
with 
$\Z/2\Z = \{ \pm 1\} \equiv \{\pm\}$
and 
$\lambda:=(\lambda_1, \lambda_2),\,
\nu:=(\nu_1,\nu_2)
\in \C^2$.
For $a \in \{\id, 1, 2, +,-,\Cay\}$,  we define
$\Supp_{a} \subset  (\Z/2\Z)^4\times \C^4$ as follows:
\begin{equation*}
\Supp_{a}:=\{(\eps, \delta; \lambda, \nu) \in (\Z/2\Z)^4\times \C^4: 
\text{($L_a$) holds}\}.
\end{equation*}
\begin{alignat}{3}
&(\lambda_1, \lambda_2) \in \C^2, \quad
&&(\nu_1, \nu_2) = (\lambda_1, \lambda_2), \quad
&&(\delta_1, \delta_2) = (\eps_1, \eps_2) 
\tag{$L_\id$}\label{eqn:L0}\\[3pt]
&(\lambda_1, \lambda_2) \in \Lambda_1, \quad
&&(\nu_1, \nu_2) = (2-\lambda_1, \lambda_1 + \lambda_2 -1 ), \quad
&&(\delta_1, \delta_2) = (\eps_1, \eps_2\cdot (-)^{1-\lambda_1}) 
\tag{$L_1$}\label{eqn:L1}\\[3pt]
&(\lambda_1, \lambda_2) \in  \Lambda_2, \quad
&&(\nu_1, \nu_2) = (\lambda_1 + \lambda_2 -1, 2-\lambda_2), \quad
&&(\delta_1, \delta_2) = (\eps_1\cdot (-)^{1-\lambda_2}, \eps_2) 
\tag{$L_2$}\label{eqn:L2} \\[3pt]
&(\lambda_1, \lambda_2) \in  \Lambda_+, \quad
&&(\nu_1, \nu_2) = (\lambda_2, 3-\lambda_1-\lambda_2), \quad
&&(\delta_1, \delta_2) 
= (\eps_1\cdot (-)^{2-\lambda_1-\lambda_2}, \eps_2\cdot (-)^{1-\lambda_1})
\tag{$L_+$}\label{eqn:L3} \\[3pt]
&(\lambda_1, \lambda_2) \in  \Lambda_-, \quad
&&(\nu_1, \nu_2) = (3-\lambda_1-\lambda_2, \lambda_1), \quad
&&(\delta_1, \delta_2) 
= (\eps_1\cdot (-)^{1-\lambda_2}, \eps_2\cdot (-)^{2-\lambda_1-\lambda_2})
\tag{$L_-$}\label{eqn:L4} \\[3pt]
&(\lambda_1, \lambda_2) \in \Lambda_{\Cay}, \quad
&&(\nu_1, \nu_2) = (2-\lambda_2, 2-\lambda_1), \quad
&&(\delta_1, \delta_2) 
= (\eps_1\cdot (-)^{2-\lambda_1-\lambda_2}, \eps_2\cdot (-)^{2-\lambda_1-\lambda_2})
\tag{$L_\Cay$}\label{eqn:L5}
\end{alignat}

We put
\begin{equation*}
\Supp: =\bigcup_{a \in \{\id, 1, 2, + ,-, \Cay\}} \Supp_a.
 \end{equation*}

\begin{thm}\label{thm:DIO1}
The following conditions on $(\eps, \delta; \lambda, \nu) \in (\Z/2\Z)^4\times \C^4$ 
are equivalent:

\begin{enumerate}

\item[\emph{(i)}] 
$\Diff_G(I(\lambda_1, \lambda_2)^{(\eps_1, \eps_2)}, I(\nu_1,\nu_2)^{(\delta_1, \delta_2)})
\neq \{0\}$.

\vskip 0.1in

\item[\emph{(ii)}] 
$\dim_\C \Diff_G(I(\lambda_1, \lambda_2)^{(\eps_1, \eps_2)}, 
I(\nu_1,\nu_2)^{(\delta_1, \delta_2)})=1$.

\vskip 0.1in

\item[\emph{(iii)}] 
$(\eps, \delta; \lambda, \nu) \in \Supp$.

\end{enumerate}

\end{thm}

\begin{thm}\label{thm:Hom1}
The following conditions on $(\eps, \delta; \lambda, \nu) \in (\Z/2\Z)^4\times \C^4$ 
are equivalent:

\begin{enumerate}

\item[\emph{(i)}] 
$
\Hom_{\fg, B}(
M(-\nu_1,-\nu_2)^{(\delta_1, \delta_2)},
M(-\lambda_1, -\lambda_2)^{(\eps_1, \eps_2)})
\neq \{0\}
$.

\vskip 0.1in

\item[\emph{(ii)}] 
$\dim_\C 
\Hom_{\fg, B}(
M(-\nu_1,-\nu_2)^{(\delta_1, \delta_2)},
M(-\lambda_1, -\lambda_2)^{(\eps_1, \eps_2)})
=1
$.

\vskip 0.1in

\item[\emph{(iii)}] 
$(\eps, \delta; \lambda, \nu) \in \Supp$.

\end{enumerate}

\end{thm}

\subsection{Construction of DIOs $\D$ and $(\fg, B)$-homomorphisms $\varphi$}
\label{sec:construction}

Next we shall give the explicit formulas for DIOs $\D$ 
and $(\fg,B)$-homomorphisms $\varphi$. Let 
$\{\Cay_m(x;y)\}_{m=0}^\infty$ denote the family of 
the determinants of $m\times m$ tridiagonal matrices given by
\begin{itemize}

\item[$\diamond$] $m =0:$ $\Cay_0(x;y)=1$,
\vskip 0.1in

\item[$\diamond$] $m=1:$ $\Cay_1(x;y) = x$,
\vskip 0.1in

\item[$\diamond$] $m \geq 2:$ $\Cay_m(x;y) = 
\begin{vmatrix}
x& 1    &              &            &             &       \\
y& x    &     2       &            &             &       \\
  & y-1 &     x       &  3        &             &       \\
  &       &  \dots  &  \dots & \dots  &       \\
  &       &              &y-m+3   &x            & m-1 \\
  &       &              &            &y-m+2     & x    
\end{vmatrix}
$.
\end{itemize}
The tridiagonal determinants $\Cay_m(x;y)$ are called
\emph{Cayley continuants} (cf.\ \cite{MT05}).
We have
\begin{equation*}
\Cay_m(x;y)=\sum_{j=0}^m\binom{m}{j}
\left(\frac{x+y}{2}\right)^{\underline{j}}
\left(\frac{x-y}{2}\right)^{\overline{m-j}},
\end{equation*}
where $r^{\underline{j}}$ and $r^{\overline{j}}$ denote the 
falling factorial and rising factorial, respectively, namely,
\begin{align*}
r^{\underline{j}}&:=r(r-1)(r-2)\cdots (r-j+1)
=\frac{\Gamma(r+1)}{\Gamma(r-j+1)},\\[3pt]
r^{\overline{j}}&:=r(r+1)(r+2)\cdots (r+j-1)
=\frac{\Gamma(r+j)}{\Gamma(r)}.
\end{align*}
(See, for instance, \cite[(5)]{MT05}.)

Let $\Cal{A}$ be an algebra over $\C$.
Then, for $k, \ell \in \Z_{\geq 0}$ and $s \in \C$, we define polynomials
$\eta_+^{(k,\ell)}(a_1,a_2,a_3)$, 
$\eta_-^{(k,\ell)}(a_1,a_2,a_3)$,  
and $\eta_{\Cay}^{(s;k)}(a_1,a_2,a_3)$ on $\Cal{A}$ as follows.
\begin{align}
\eta_+^{(k,\ell)}(a_1,a_2,a_3) &:= \sum_{m=0}^{\min(k,\ell)} \frac{1}{2^m}m!\binom{k}{m}\binom{\ell}{m}
\sym_{k+\ell-m}(a_1^{k-m} a_2^{\ell-m} a_3^m) \label{eqn:eta1} \\[3pt]
\eta_-^{(k,\ell)}(a_1,a_2,a_3) &:= \sum_{m=0}^{\min(k,\ell)} \frac{(-1)^m}{2^m}m!\binom{k}{m}\binom{\ell}{m}
\sym_{k+\ell-m}(a_1^{k-m} a_2^{\ell-m} a_3^m) \label{eqn:eta2}\\[3pt]
\eta_{\Cay}^{(s;k)}(a_1,a_2,a_3)&:= \sum_{m=0}^{k} \frac{1}{2^m}\binom{k}{m}\Cay_m(s;k)
\sym_{2k-m}(a_1^{k-m} a_2^{k-m} a_3^m) \label{eqn:eta3}
\end{align}
Here, $\sym_n(a_{j_1}\cdots a_{j_n})$ denotes the symmetrization of 
the product $a_{j_1}\cdots a_{j_n} \in \Cal{A}$
defined as in \eqref{eqn:symm}.

As in \eqref{eqn:DN}, we understand 
$\D \in 
\Diff_G(I(\lambda_1, \lambda_2)^{(\eps_1, \eps_2)}, 
I(\nu_1,\nu_2)^{(\delta_1,\delta_2)})$
for $(\eps, \delta; \lambda, \nu) \in \Supp$ as a linear map
$\D\colon C^\infty(\R^3) \to C^\infty(\R^3)$
via the diffeomorphism
\begin{equation}\label{eqn:n-}
\R^3 \stackrel{\sim}{\To} N_-, \quad (x_1, x_2, x_3) 
\mapsto \exp(x_1 N_1^- + x_2 N_2^- +  x_3 N_3^-).
\end{equation}

We set
\begin{align*}
\D_1
:=dR(N_1^-)
=\frac{\partial}{\partial x_1} + \frac{x_2}{2}\frac{\partial}{\partial x_3},
\quad
\D_2
:=dR(N_2^-)
=\frac{\partial}{\partial x_2} - \frac{x_1}{2}\frac{\partial}{\partial x_3},
\quad
\D_3
:=dR(N_3^-)=\frac{\partial}{\partial x_3}.
\end{align*}
(We shall  discuss the formulas for $dR(N_j^-)$ in Lemma \ref{lem:dR} in
the next section.) Then, 
for $k, \ell \in \Z_{\geq 0}$ and $s \in \C$, we define three differential operators
$\D_+^{(k,\ell)}$, $\D_-^{(k,\ell)}$, and $\D_{\Cay}^{(s;k)}$ as follows.
\begin{equation*}
\D_+^{(k,\ell)}:=\eta_+^{(k,\ell)}(\D_1,\D_2,\D_3), \quad
\D_-^{(k,\ell)}:=\eta_-^{(k,\ell)}(\D_1,\D_2,\D_3), \quad
\D_{\Cay}^{(s;k)} := \eta_{\Cay}^{(s;k)}(\D_1,\D_2,\D_3).
\end{equation*}

\vskip 0.1in

\begin{thm}\label{thm:DIO2} 
The following holds.
\begin{equation*}
\Diff_G(I(\lambda_1, \lambda_2)^{(\eps_1, \eps_2)}, 
I(\nu_1,\nu_2)^{(\delta_1,\delta_2)})
=
\begin{cases}
\C \id &  \emph{if $(\eps, \delta; \lambda,\nu) \in \Supp_{\id}$}, \\[3pt]
\C \D_1^{1-\lambda_1} & \emph{if $(\eps, \delta; \lambda,\nu) \in \Supp_1$}, \\[3pt]
\C \D_2^{1-\lambda_2} & \emph{if $(\eps, \delta; \lambda,\nu) \in \Supp_2$}, \\[3pt]
\C \D_+^{(1-\lambda_1, 2-\lambda_1-\lambda_2)} & \emph{if $(\eps, \delta; \lambda,\nu) \in \Supp_+$}, \\[3pt]
\C \D_-^{(2-\lambda_1-\lambda_2, 1-\lambda_2)} & \emph{if $(\eps, \delta; \lambda,\nu) \in \Supp_-$}, \\[3pt]
\C \D_{\Cay}^{(\lambda_2-\lambda_1; 2-\lambda_1-\lambda_2)} & \emph{if $(\eps, \delta; \lambda,\nu) \in \Supp_{\Cay}$}, \\[3pt]
\{0\} & \emph{otherwise}.
\end{cases}
\end{equation*}
\end{thm}

The order of the DIO $\D$ for $(\eps, \delta; \lambda,\nu) \in \Supp$ is 
given by $(\nu_1+\nu_2)-(\lambda_1+\lambda_2)$.

\vskip 0.1in

Next, we give explicit formulas of $(\fg,B)$-homomorphisms $\varphi$.
For $k,\ell \in \Z_{\geq 0}$ and $s \in \C$, 
we define maps 
\begin{equation*}
\varphi_1^{(k)},
\varphi_2^{(k)},
\varphi_+^{(k,\ell)},
\varphi_-^{(k,\ell)},
\varphi_{\Cay}^{(s;k)} 
\in \Hom_{\C}(\C , \Cal{U}(\fn_-) \otimes \C)
\end{equation*}
as follows.
\begin{align*}
\varphi_1^{(k)} &\colon \mathbb{1} 
\longmapsto (N_1^-)^{k} \otimes \mathbb{1}, \\[3pt]
\varphi_2^{(k)} & \colon \mathbb{1}
\longmapsto (N_2^-)^{k} \otimes \mathbb{1}, \\[3pt]
\varphi_+^{(k,\ell)}& \colon \mathbb{1} 
\longmapsto \eta_+^{(k,\ell)}(N_1^-,N_2^-,N_3^-) \otimes \mathbb{1}, \\[3pt]
\varphi_-^{(k,\ell)}& \colon \mathbb{1}
\longmapsto \eta_-^{(k,\ell)}(N_1^-,N_2^-,N_3^-) \otimes \mathbb{1}, \\[3pt]
\varphi_{\Cay}^{(s;k)} & \colon \mathbb{1}
\longmapsto \eta_{\Cay}^{(s;k)}(N_1^-,N_2^-,N_3^-)\otimes \mathbb{1}.
\end{align*}

Via the identification
\begin{align*}
&\Hom_{\fg, B}(
M(-\nu_1,-\nu_2)^{(\delta_1, \delta_2)},
M(-\lambda_1, -\lambda_2)^{(\eps_1, \eps_2)})\\
&\simeq
\Hom_B(\C_{(\delta_1,\delta_2)} \boxtimes \C_{(-\nu_1,-\nu_2)},
M(-\lambda_1, -\lambda_2)^{(\eps_1, \eps_2)}),
\end{align*}
the $(\fg,B)$-homomorphisms $\varphi$ are given as follows.

\begin{thm}\label{thm:Hom2} 
The following holds.
\begin{equation*}
\Hom_{\fg, B}(
M(-\nu_1,-\nu_2)^{(\delta_1, \delta_2)},
M(-\lambda_1, -\lambda_2)^{(\eps_1, \eps_2)})
=
\begin{cases}
\C \id & \emph{if $(\eps, \delta; \lambda,\nu) \in \Supp_{\id}$}, \\[3pt]
\C \varphi_1^{(1-\lambda_1)} & 
\emph{if $(\eps, \delta; \lambda,\nu) \in \Supp_1$}, \\[3pt]
\C \varphi_2^{(1-\lambda_2)} & \emph{if $(\eps, \delta; \lambda,\nu) \in \Supp_2$}, \\[3pt]
\C \varphi_+^{(1-\lambda_1, 2-\lambda_1-\lambda_2)} & \emph{if $(\eps, \delta; \lambda,\nu) \in \Supp_+$}, \\[3pt]
\C \varphi_-^{(2-\lambda_1-\lambda_2,1-\lambda_2)} & \emph{if $(\eps, \delta; \lambda,\nu) \in \Supp_-$}, \\[3pt]
\C \varphi_{\Cay}^{(\lambda_2-\lambda_1; 2-\lambda_1-\lambda_2)} & \emph{if $(\eps, \delta; \lambda,\nu) \in \Supp_{\Cay}$}, \\[3pt]
\{0\} & \emph{otherwise}.
\end{cases}
\end{equation*}
\end{thm}

\subsection{Classification in terms of infinitesimal characters}
\label{sec:InfChar}

For later convenience we also give the classification of $(\fg, B)$-homomorphisms
$\varphi$ in terms of infinitesimal characters of Verma modules.

Let $\varpi_1$ and $\varpi_2$ be the first and second fundamental weights,
respectively, so that $\varpi_i(H_j) = \delta_{i,j}$ for $H_j$ in \eqref{eqn:H}.
Then the differential
$d\chi_A^{(\mu_1,\mu_2)}$ of the character 
$\chi_A^{(\mu_1,\mu_2)}$ defined in \eqref{eqn:chi} is given as
\begin{equation}\label{eqn:dchi}
d\chi_A^{(\mu_1,\mu_2)} = \mu_1\varpi_1+\mu_2\varpi_2.
\end{equation}
Therefore the Verma module $M(\mu_1,\mu_2)$ defined in 
\eqref{eqn:Verma}
can be understood as
\begin{equation*}
M(\mu_1,\mu_2)
=\Cal{U}(\fg)\otimes_{\Cal{U}(\fb)} \C_{\mu_1\varpi_1+\mu_2\varpi_2}.
\end{equation*}

Now, for $\mu \in \fa^*$, we denote by
\begin{equation*}
N(\mu):=\Cal{U}(\fg)\otimes_{\Cal{U}(\fb)} \C_{\mu-\rho},
\end{equation*}
the Verma module with infinitesimal character $\chi_\mu$.
As $\rho=\varpi_1 + \varpi_2$, 
for $\mu=(\mu_1,\mu_2)$, we have
\begin{equation*}
M(\mu_1,\mu_2) = N(\mu_1+1,\mu_2+1).
\end{equation*}

We wish to state Theorem \ref{thm:Hom2} in terms of $N(\mu)$.
To do so,
it is convenient to identify $\fa^*$ with
\begin{equation}\label{eqn:astar}
\fa^*=\{(v_1, v_2 , v_3) \in \C^3: v_1+v_2+v_3=0\}.
\end{equation} 
In this realization, we have
\begin{equation*}
\varpi_1=\tfrac{1}{3}(2,-1,-1)
\quad
\text{and}
\quad
\varpi_2=\tfrac{1}{3}(1,1,-2);
\end{equation*}
in particular, $\rho = \varpi_1+\varpi_2 = (1,0,-1)$.

In the notation, the weight $\mu_1\varpi_1+\mu_2\varpi_2+\rho$ is given as
\begin{equation*}
\mu_1\varpi_1+\mu_2\varpi_2+\rho
=
\tfrac{1}{3}(2\mu_1+\mu_2+3, -\mu_1+\mu_2, -\mu_1-2\mu_2-3).
\end{equation*}
Then, for $\lambda=(\lambda_1,\lambda_2)$, we put
\begin{align}
\mu_\lambda
&:=-(\lambda_1\varpi_1+\lambda_2\varpi_2)+\rho \nonumber\\[3pt]
&=\tfrac{1}{3}(-(2\lambda_1+\lambda_2-3), \lambda_1-\lambda_2, \lambda_1+2\lambda_2-3) \label{eqn:muL},
\end{align}
which gives
\begin{align*}
M(-\lambda_1, -\lambda_2)^{(\eps_1,\eps_2)}
=
N(-\lambda_1+1,-\lambda_2+1)^{(\eps_1,\eps_2)}
\equiv
N(\mu_\lambda)^{(\eps_1,\eps_2)},
\end{align*}
where the identification $\fa^* \subset \C^3$ is applied from the first identity 
to the second.

For Theorem \ref{thm:Hom3} below, we put
\begin{equation*}
\ga:=\gamma_1-\gamma_2=(1,-1,0),\quad
\beta:=\gamma_2-\gamma_3=(0,1,-1), \quad
\gamma:=\gamma_1-\gamma_3=(1,0,-1).
\end{equation*}
For $\eta \in \gD$, we denote by $s_\eta$ the root reflection with respect to $\eta$.

In the following,
we only consider the case $(\eps_1,\eps_2) = (+, +)$;
the other cases can be obtained easily by
the conditions for $(\delta_1,\delta_2)$ in \eqref{eqn:L0}--\eqref{eqn:L5}.

\begin{thm}\label{thm:Hom3}
For $(\eps_1, \eps_2) = (+, +)$,  the $(\fg, B)$-homomorphisms 
$\varphi$
are classified as follows.
\begin{alignat*}{4}
&\id
&&\colon 
N(\mu_\lambda)^{(+, +)}
&&\To 
N(\mu_\lambda)^{(+,+)}
\quad  
&&\emph{for $(\lambda_1,\lambda_2) \in \C^2$}, \\[3pt]
&\varphi_1^{(1-\lambda_1)} 
&&\colon 
N(s_{\ga}\mu_\lambda)^{(+,\, (-)^{1-\lambda_1})}
&&\To 
N(\mu_\lambda)^{(+,+)}
\quad  
&&\emph{for $(\lambda_1,\lambda_2) \in \Lambda_1$}, \\[3pt]
&\varphi_2^{(1-\lambda_2)} 
&&\colon 
N(s_{\gb}\mu_\lambda)^{((-)^{1-\lambda_2},\, +)}
&&\To
N(\mu_\lambda)^{(+,+)} 
\quad  
&&\emph{for $(\lambda_1,\lambda_2) \in \Lambda_2$}, \\[3pt]
&\varphi_+^{(1-\lambda_1,2-\lambda_1-\lambda_2)}
&& \colon 
N(s_{\gb}s_{\ga}\mu_\lambda)^{((-)^{2-\lambda_1-\lambda_2},\, (-1)^{1-\lambda_1})}
&& \To 
N(\mu_\lambda)^{(+,+)}
\quad  
&&\emph{for $(\lambda_1,\lambda_2) \in \Lambda_+$}, \\[3pt]
&\varphi_-^{(2-\lambda_1-\lambda_2,1-\lambda_2)}
&& \colon 
N(s_{\ga}s_{\gb}\mu_\lambda)^{((-1)^{1-\lambda_2},\, (-)^{2-\lambda_1-\lambda_2})} 
&&\To 
N(\mu_\lambda)^{(+,+)} 
\quad  
&&\emph{for $(\lambda_1,\lambda_2) \in \Lambda_-$}, \\[3pt]
&\varphi_{\Cay}^{(\lambda_2-\lambda_1;2-\lambda_1-\lambda_2)} 
&& \colon 
N(s_{\gamma}\mu_\lambda)^{((-1)^{2-\lambda_1-\lambda_2},\, (-)^{2-\lambda_1-\lambda_2})}
&&\To 
N(\mu_\lambda)^{(+,+)}
\quad  
&&\emph{for $(\lambda_1,\lambda_2) \in \Lambda_{\Cay}$}.
\end{alignat*}
\end{thm}

\begin{proof}
We show only $\varphi_1^{(1-\lambda_1)}$;
the other cases can be handled similarly.
It follows from Theorem \ref{thm:Hom2} that,
for $(\lambda_1, \lambda_2) \in \Lambda_1$, we have 
\begin{equation}\label{eqn:T421}
\varphi_1^{(1-\lambda_1)} \colon
M(-(2-\lambda_1),-(\lambda_1+\lambda_2-1))^{(+,(-)^{1-\lambda_1})}
\To 
M(-\lambda_1,-\lambda_2)^{(+,+)}.
\end{equation}
As in \eqref{eqn:muL}, for $(\nu_1,\nu_2)=(2-\lambda_1,\lambda_1+\lambda_2-1)$,
one computes that
\begin{equation*}
-(\nu_1\varpi_1+\nu_2\varpi_2)+\rho 
=\tfrac{1}{3}(\lambda_1-\lambda_2,
-(2\lambda_1+\lambda_2-3), 
\lambda_1+2\lambda_2-3)
=s_{\ga}\mu_\lambda,
\end{equation*}
which shows
\begin{equation*}
M(-(2-\lambda_1),-(\lambda_1+\lambda_2-1))^{(+,(-)^{1-\lambda_1})}
=N(s_\ga\mu_\lambda)^{(+,(-)^{1-\lambda_1})}.
\end{equation*}
Now the equality
$M(-\lambda_1,-\lambda_2)^{(+,+)}=N(\mu_\lambda)^{(+,+)}$ and
\eqref{eqn:T421} conclude the case for $\varphi_1^{(1-\lambda_1)}$.
\end{proof}

\begin{example}
The diagrams \eqref{eqn:homs1} and \eqref{eqn:homs2} below
illustrate all $\fg$-homomorphisms $\varphi$ in Theorem \ref{thm:Hom3} 
for the case $(\lambda_1, \lambda_2) = (0,0)$.
\begin{equation}\label{eqn:homs1}
\begin{tikzcd}[row sep=1cm, column sep=1cm]
&N(s_{\beta}s_{\ga}\rho) 
\arrow[r, "\varphi_2^{(2)}"]  
\arrow[ddr, pos=0.8,"\varphi_{\Cay}^{(3; 1)}"]
& N(s_\ga\rho)  
\arrow[dr, "\varphi_1^{(1)}"]&\\
N(s_\gamma \rho) 
\arrow[ur, "\varphi_1^{(1)}"]
\arrow[dr, "\varphi_2^{(1)}"']
\arrow[rrr, pos=0.2, "\varphi_{\Cay}^{(0;2)}"]
& & & N(\rho)\\
&N(s_\ga s_\gb \rho)  
\arrow[r, "\varphi_1^{(2)}"']  
\arrow[uur, pos=0.85, "\varphi_{\Cay}^{(-3;1)}"']
&N(s_\gb \rho) 
\arrow[ur, "\varphi_2^{(1)}"']&
\end{tikzcd}
\end{equation}
\begin{equation}\label{eqn:homs2}
\begin{tikzcd}[row sep=1cm, column sep=1cm]
&N(s_{\beta}s_{\ga}\rho) 
\arrow[rrd, "\varphi_+^{(1,2)}"']  
& N(s_\ga\rho)  &\\
N(s_\gamma \rho) 
\arrow[urr, "\varphi_-^{(1,2)}"']
\arrow[drr, "\varphi_+^{(2,1)}"]
& & & N(\rho)\\
&N(s_\ga s_\gb \rho)    
\arrow[urr, "\varphi_-^{(2,1)}"]
&N(s_\gb \rho) &
\end{tikzcd}
\end{equation}

Here we have
$s_\gamma \rho = s_{\ga}s_{\gb}s_{\ga}\rho = s_{\gb}s_{\ga}s_{\gb}\rho = -\rho$.
Moreover, these homomorphisms admit factorization identities such as
\begin{alignat}{3}
\varphi_{\Cay}^{(0;2)} 
&= \varphi_1^{(1)} \circ  && \,\, \varphi_2^{(2)} \circ  \varphi_1^{(1)}
&&=  \varphi_2^{(1)} \circ \varphi_1^{(2)} \circ  \varphi_2^{(1)},\label{eqn:ExF1}\\[3pt]
\varphi_{\Cay}^{(3;1)} \circ \varphi_1^{(1)}
&=&&\varphi_+^{(2,1)}
&&=\varphi_1^{(2)}\circ \varphi_2^{(1)}.\label{eqn:ExF2}
\end{alignat}
In terms of singular vectors,
the identities \eqref{eqn:ExF1} and \eqref{eqn:ExF2} read
\begin{equation*}
\left(
\sym_4((N_1^-)^2(N_2^-)^2)-\tfrac{1}{2}(N_3^-)^2
\right)
\otimes \mathbb{1}_0
=(N_1^-)(N_2^-)^2(N_1^-)\otimes \mathbb{1}_0
=(N_2^-)(N_1^-)^2(N_2^-)\otimes \mathbb{1}_0
\end{equation*}
and
\begin{equation}\label{eqn:ExF4}
\begin{aligned}
\left(
(N_1^-)\cdot \sym_2((N_1^-)(N_2^-))+\tfrac{3}{2}(N_3^-)
\right)\otimes \mathbb{1}_{s_\beta \rho-\rho}
&=
\left(
\sym_3((N_1^-)^2(N_2^-))+\sym_2((N_1^-)(N_3^-))
\right)
\otimes \mathbb{1}_{s_\beta \rho-\rho}\\[3pt]
&=
(N_2^-)(N_1^-)^2\otimes \mathbb{1}_{s_\beta \rho-\rho},
\end{aligned}
\end{equation}
respectively.
Remark that we have
$(\varphi_1^{(2)}\circ \varphi_2^{(1)})(1\otimes \mathbb{1}_{s_\gamma\rho-\rho})=(N_2^-)(N_1^-)^2\otimes \mathbb{1}_{s_\beta \rho-\rho}$, not 
$(N_1^-)^2(N_2^-)\otimes \mathbb{1}_{s_\beta \rho-\rho}$
(see Lemma \ref{lem:hom-comp}).
We shall discuss the factorization identities 
in detail in Section \ref{sec:factorization}.
\end{example}

\begin{rem}
Via the algebraic Fourier transform $F_c$,
the symmetrized form of a singular vector can be obtaind
from the non-symmetrized one.
Indeed, for instance,
it follows from a direct computation that
$\hdL(N_2^-)\hdL(N_1^-)^2\acts 1 = \zeta_1^2\zeta_2+\zeta_1\zeta_3$.
Then we have 
\begin{align*}
(N_2^-)(N_1^-)^2\otimes \mathbb{1}_{s_\beta\rho-\rho}
&=(F_c^{-1}\circ F_c)((N_2^-)(N_1^-)^2\otimes \mathbb{1}_{s_\beta\rho-\rho})\\[3pt]
&=F_c^{-1}\big((\hdL(N_2^-)\hdL(N_1^-)^2\acts 1)
\otimes \mathbb{1}_{s_\beta\rho-\rho}\big)\\[3pt]
&=F_c^{-1}\big((\zeta_1^2\zeta_2+\zeta_1\zeta_3)\otimes \mathbb{1}_{s_\beta\rho-\rho})\\[3pt]
&=\big(\sym_3((N_1^-)^2(N_2^-))+\sym_2((N_1^-)(N_3^-))\big)
\otimes \mathbb{1}_{s_\beta\rho-\rho}.
\end{align*}
\end{rem}

\begin{example}\label{example:symb3}
It follows from \eqref{eqn:ExF4} that the truncated symbol 
$\Symb_0(\D_2\D_1^2)$ of $\D_2\D_1^2$ is 
\begin{equation}\label{eqn:symbEx}
\Symb_0(\D_2\D_1^2) 
= \Symb_0(\sym_3(\D_1^2\D_2)+\sym_2(\D_1\D_3))
=\zeta_1^2\zeta_2+\zeta_1\zeta_3. 
\end{equation}
The equality \eqref{eqn:symbEx} can be obtained 
directly from the definition of $\Symb_0$.
Indeed, a direct computation shows that the differential operator
$\D_2\D_1^2$ is given in the local coordinates \eqref{eqn:n-} as follows:
\begin{align*}
\D_2\D_1^2
=\frac{\partial^3}{\partial x_1^2 \partial x_2}
+\frac{\partial^2}{\partial x_ 1\partial x_3}
&+\frac{1}{2}x_2\frac{\partial^3}{\partial x_3^3}
+x_2\frac{\partial^3}{\partial x_1\partial x_2 \partial x_3}\\[3pt]
&+\frac{1}{4}x_2^2\frac{\partial^3}{\partial x_2 \partial x_3^2}
-\frac{1}{2}x_1\frac{\partial^3}{\partial x_1^2 \partial x_3}
-\frac{1}{2}x_1x_2\frac{\partial^3}{\partial x_1 \partial x_3^2}
-\frac{1}{8}x_1x_2^2\frac{\partial^3}{\partial x_3^3}.
\end{align*}
Thus the truncated operator $(\D_2\D_1^2)_0:=\Trun_0(\D_2\D_1^2)$ is
$(\D_2\D_1^2)_0 =\frac{\partial^3}{\partial x_1^2 \partial x_2}
+\frac{\partial^2}{\partial x_ 1\partial x_3}$.
Therefore, the truncated symbol $\Symb_0(\D_2\D_1^2)$ is given as
\begin{equation*}
\Symb_0(\D_2\D_1^2)=\Symb((\D_2\D_1^2)_0)=\zeta_1^2\zeta_2+\zeta_1\zeta_3,
\end{equation*}
which agrees with \eqref{eqn:symbEx}.
\end{example}

\begin{example}\label{example:symb4}
A direct computation shows that the differential operator 
\begin{equation*}
\D_{\Cay}^{(0;1)}
=\sym_2(\D_1\D_2)
=\frac{1}{2}(\D_1\D_2 + \D_2\D_1)
\end{equation*}
is given in the local coordinates \eqref{eqn:n-} as 
\begin{equation*}
\D_{\Cay}^{(0;1)}
=
\frac{\partial^2}{\partial x_1 \partial x_2}
-\frac{1}{2}x_1\frac{\partial^2}{\partial x_1 \partial x_3}
+\frac{1}{2}x_2\frac{\partial^2}{\partial x_3 \partial x_3}
-\frac{1}{4}x_1x_2\frac{\partial^2}{\partial x_3^2}
\end{equation*}
(see Example \ref{example:xy}).
Therefore,
the truncated symbol $\Symb_0(\D_{\Cay}^{(0;1)})$ of $\D_{\Cay}^{(0;1)}$ is
$\Symb_0(\D_{\Cay}^{(0;1)})=\zeta_1\zeta_2$. Remark that the $K$-type decompositions of the 
representations realized in the kernels $\Ker(\D_{\Cay}^{(0;1)})$ 
and $\Ker(\D_{\Cay}^{(s;1)})$ are explicitly computed
in \cite{KuOr19} and \cite{KuOr25b}, respectively.
\end{example}

\section{Proof of the classification and construction of $\D$ and $\varphi$: Part I}
\label{sec:proof1}

The aim of the next two sections is to prove 
Theorems 
\ref{thm:DIO1},
\ref{thm:Hom1}, 
\ref{thm:DIO2},
and
\ref{thm:Hom2}.
Our primary goal is to classify
the space $\Sol(\eps,\delta;\lambda,\nu)$
of solutions  to the F-system:
\begin{align*}
&
\Sol(\eps,\delta;\lambda,\nu)\\
&=
\{ \psi \in 
\Hom_{MA}(\C_{\delta}\boxtimes \C_{-\nu}, 
\Pol(\fn_+)\otimes (\C_\eps\boxtimes \C_{-\lambda})): 
\text{
$\psi$ solves \eqref{eqn:Fsys3} below.}\}.
\end{align*}
\begin{equation}\label{eqn:Fsys3}
\widehat{\dpi_{(\eps, \lambda)^*}}(N_j^+)\psi =0
\quad
\text{for $j=1,2$}.
\end{equation}
Remark that since $N_3^+ = [N_1^+, N_2^+]$, 
it suffices to consider 
$\widehat{\dpi_{(\eps, \lambda)^*}}(N_j^+)\psi=0$ for $j=1,2$ in \eqref{eqn:Fsys3}.

We proceed with the following recipe to compute 
$\Sol(\eps,\delta;\lambda,\nu)$,
where 
the order of Steps 1 and 2 is reversed
compared with that in Section \ref{sec:recipe}.

\vskip 0.1in

\begin{enumerate}

\item[Step 1]
Classify and construct
$\psi \in
\Hom_{MA}(\C_{\delta}\boxtimes \C_{-\nu}, 
\Pol(\fn_+)\otimes (\C_\eps\boxtimes \C_{-\lambda}))$.

\vskip 0.1in

\item[Step 2]
Compute $d\pi_{(\eps,\lambda)^*}(N_j^+)$ and 
$\widehat{d\pi_{(\eps,\lambda)^*}}(N_j^+)$
for $j=1,2$.

\vskip 0.1in

\item[Step 3]
Solve the F-system \eqref{eqn:Fsys3}.
\vskip 0.1in

\item[Step 4]
Apply $\Symb_0^{-1}$ and $F_c^{-1}$ 
to the solutions $\psi \in \Sol(\eps,\delta;\lambda,\nu)$.

\end{enumerate}

\vskip 0.1in

In this section we focus on Steps 1 and 2. 
The remaining steps will be 
handled in Section \ref{sec:proof2} by solving the resulting system of 
differential equations.

\subsection{Step 1: 
Classify and construct
$\psi \in
\Hom_{MA}(\C_{\delta}\boxtimes \C_{-\nu}, 
\Pol(\fn_+)\otimes (\C_\eps\boxtimes \C_{-\lambda}))$}
For $X, Y \in \fg$, let $\Tr(X,Y):=\text{Trace}(XY)$ denote the trace form of $\fg$. 
Then $N_i^+$ and $N_j^-$ satisfy $\Tr(N_i^+,N_j^-)=\delta_{i,j}$.
We identify  the dual $\fn_+^\vee$
with $\fn_+^\vee \simeq \fn_-$ via the trace form $\Tr(\cdot, \cdot)$,
which gives rise to an $MA$-isomorphism
$\Pol(\fn_+)=S(\fn_+^\vee) \simeq S(\fn_-)$.
In particular, we have
$\zeta_j(\cdot)=\Tr(\cdot, N_j^-)$
for $\Pol(\fn_+)=\C[\zeta_1, \zeta_2, \zeta_3]$.

As $N_j^-$ are root vectors, the spaces 
$\C\zeta_j^n\simeq \C(N_j^-)^n$
are $MA$-representations for any $n \in \Z_{\geq 0}$.
Observe that, given a  character $(\eps_1,\eps_2)$ of $M$, we have
$(\eps_1, \eps_2)^n=(\eps_1^n, \eps_2^n)$.

\begin{lem}\label{lem:zetaChar}
Let $n \in \Z_{\geq 0}$.
Then, as $MA$-representations, we have
\begin{equation}\label{eqn:zetaChar1}
\C \zeta_1^n\simeq \C_{(+,(-)^n)} \boxtimes \C_{(-2n,n)},
\quad
\C \zeta_2^n\simeq \C_{((-)^n,+)} \boxtimes \C_{(n,-2n)},
\quad
\C \zeta_3^n\simeq \C_{((-)^n,(-)^n)} \boxtimes \C_{(-n,-n)}.
\end{equation}
In particular, for $k, \ell \in \Z_{\geq 0}$ and $j \in \{0, 1,2,\ldots, \min(k,\ell)\}$,
we have 
\begin{equation}\label{eqn:zetaChar2}
\C\zeta_1^{k-j}\zeta_2^{\ell-j}\zeta_3^j \simeq 
\C_{((-)^\ell,(-)^k)}\boxtimes \C_{(\ell-2k, k-2\ell)}.
\end{equation}
\end{lem}

\begin{proof}
The isomorphisms in \eqref{eqn:zetaChar1}
can be checked directly via the action $\Ad_{\#}$ in \eqref{eqn:sharp}.
Or,  as $\zeta_j(\cdot)=\Tr(\cdot, N_j^-)$, they can also be shown by computing
the characters on $\C (N_j^-)^n$. 
The second identity \eqref{eqn:zetaChar2} is an immediate consequence 
of \eqref{eqn:zetaChar1}. 
\end{proof}

For $k,\ell \in \Z_{\geq 0}$, we put
\begin{equation*}
\Pol(k,\ell):=\spn_{\C}\{\zeta_1^{k-j}\zeta_2^{\ell-j}\zeta_3^j:
j=0,1,\ldots, \min(k,\ell)\}.
\end{equation*}
By Lemma \ref{lem:zetaChar}, as $MA$-modules, we have
\begin{equation*}
\Pol(k,\ell) \simeq  
\bigoplus_{j=0}^{\min(k,\ell)}
\C_{((-)^\ell,(-)^k)}\boxtimes \C_{(\ell-2k, k-2\ell)}.
\end{equation*}

\begin{prop}\label{prop:MA}
The following conditions on 
$(\eps, \delta; \lambda,\nu) \in (\Z/2\Z)^4 \times \C^4$ are equivalent.
\begin{enumerate}
\item[\emph{(i)}] 
$\Hom_{MA}(\C_{\delta}\boxtimes \C_{-\nu}, 
\Pol(\fn_+)\otimes (\C_\eps\boxtimes \C_{-\lambda}))\neq \{0\}$.
\item[\emph{(ii)}] There exist $k, \ell \in \Z_{\geq 0}$ such that
\begin{equation}\label{eqn:MAcond}
(\eps_1\delta_1, \eps_2\delta_2)=((-)^\ell, (-)^k)
\quad
\text{and}
\quad
(\nu_1-\lambda_1,\nu_2-\lambda_2)
= (2k-\ell,2\ell-k).
\end{equation}
\end{enumerate}
\end{prop}

\begin{proof}
Observe that we have
\begin{equation*}
\Pol(\fn_+)=\bigoplus_{k,\ell \in \Z_{\geq 0}} \Pol(k,\ell).
\end{equation*}
Thus, the $MA$-representation $\Pol(\fn_+)$ can be decomposed as
\begin{equation}\label{eqn:PolDecomp}
\Pol(\fn_+)\simeq \bigoplus_{k,\ell \in \Z_{\geq 0}} 
\bigoplus_{j=0}^{\min(k,\ell)}
\C_{((-)^\ell,(-)^k)}\boxtimes \C_{(\ell-2k, k-2\ell)},
\end{equation}
which yields the proposition.
\end{proof}

\begin{cor}\label{cor:Sol}
If $\Sol(\eps, \delta; \lambda,\nu) \neq \{0\}$, then
$(\eps, \delta;\lambda, \nu)$ satisfies \eqref{eqn:MAcond} 
for some $k,\ell \in \Z_{\geq0}$.
\end{cor}

\begin{proof}
As $\Sol(\eps, \delta; \lambda,\nu) \subset 
\Hom_{MA}(\C_{\delta}\boxtimes \C_{-\nu}, 
\Pol(\fn_+)\otimes (\C_\eps\boxtimes \C_{-\lambda}))$,
this is an immediate consequence of Proposition \ref{prop:MA}.
\end{proof}

Given $k, \ell \in \Z_{\geq 0}$,
for $(\eps,\delta;\lambda,\nu)$ satisfying \eqref{eqn:MAcond}
with $\lambda = (\lambda_1, \lambda_2)$, 
we put
\begin{equation*}
\Sol_{(k,\ell)}(-\lambda_1,-\lambda_2):=
\{ \psi \in 
\Hom_{MA}(\C_{\delta}\boxtimes \C_{-\nu}, 
\Pol(k,\ell)\otimes (\C_\eps\boxtimes \C_{-\lambda})): 
\text{$\psi$ solves \eqref{eqn:Fsys3}.}\}.
\end{equation*}
By Proposition \ref{prop:MA}, to determine 
$\Sol(\eps,\delta;\lambda,\nu)$,
it suffices to compute 
$\Sol_{(k,\ell)}(-\lambda_1,-\lambda_2)$ for 
all $k, \ell \in \Z_{\geq 0}$. We shall carry it out in Section \ref{sec:proof2}.

\subsection{Step 2: Compute 
$d\pi_{(\eps,\lambda)^*}(N_j^+)$ and 
$\widehat{d\pi_{(\eps,\lambda)^*}}(N_j^+)$ for $j=1,2$}

For the rest of this section, we prepare for computing
$\Sol_{(k,\ell)}(-\lambda_1,-\lambda_2)$.
The first objects to look at are
$d\pi_{(\eps,\lambda)^*}(N_j^+)$ and 
$\widehat{d\pi_{(\eps,\lambda)^*}}(N_j^+)$ for $j=1,2$.
To the end, we first compute $dR(N_j^-)$ for $j=1,2,3$.

\begin{lem}\label{lem:dR}
The differential operators $dR(N_j^-)$ for $j=1,2,3$ are given as follows.
\begin{equation*}
dR(N_1^-) = \frac{\partial}{\partial x_1} + \frac{x_2}{2}\frac{\partial}{\partial x_3},
\quad
dR(N_2^-) = \frac{\partial}{\partial x_2} - \frac{x_1}{2}\frac{\partial}{\partial x_3},
\quad
dR(N_3^-) = \frac{\partial}{\partial x_3}
\end{equation*}
\end{lem}

\begin{proof}
We only demonstrate the formula for $dR(N_1^-)$; the other cases can be handled
similarly. Via the local coordinates \eqref{eqn:n-}, for 
$f(\bar{n})\in C^\infty(N_-)\simeq C^\infty(\R^3)$,
we have
\begin{align*}
dR(N_1^-)f(\bar{n})
&=\frac{d}{dt}\bigg\vert_{t=0} f\left(\bar{n}\exp(tN_1^-)\right)\\[3pt]
&=\frac{d}{dt}\bigg\vert_{t=0} 
f\left(\exp(x_1 N_1^- + x_2 N_2^- + x_3 N_3^-)
\exp(tN_1^-)\right)\\[3pt]
&=\frac{d}{dt}\bigg\vert_{t=0} 
f\left(\exp((x_1+t)N_1^- + x_2 N_2^- + (x_3 +\tfrac{x_2}{2}t)N_3^-)\right)\\[3pt]
&=\frac{d}{dt}\bigg\vert_{t=0} f(x_1+t, x_2, x_3 +\tfrac{x_2}{2}t).
\end{align*}
Evaluate the derivative at $t=0$ to obtain the desired identity.
\end{proof}

As the character $\eps=(\eps_1,\eps_2)$ 
of $M$ does not contribute to solving the F-system  \eqref{eqn:Fsys3},
in the following, we simply write $\dpi_{\lambda^*}=d\pi_{(\eps,\lambda)^*}$
with
\begin{equation*}
\lambda^* :=2\rho - d\chi_A^{(\lambda_1,\lambda_2)}
=(2-\lambda_1)\varpi_1 + (2-\lambda_2)\varpi_2,
\end{equation*}
where $d\chi_A^{(\lambda_1,\lambda_2)}$ is the differential of the character of
$\chi_A^{(\lambda_1,\lambda_2)}$ given in  \eqref{eqn:dchi}.

\begin{prop}\label{prop:dpi}
We have 
\begin{align*}
\dpil(N_1^+)&=
(2-\lambda_1)x_1+x_1^2\frac{\partial}{\partial x_1} 
-\frac{1}{2}(x_1x_2-2x_3)\frac{\partial}{\partial x_2} 
+\frac{1}{4}(x_1^2x_2+2x_1x_3)\frac{\partial}{\partial x_3},\\[3pt]
\dpil(N_2^+)&=
(2-\lambda_2)x_2-\frac{1}{2}(x_1x_2+2x_3)\frac{\partial}{\partial x_1}
+x_2^2\frac{\partial}{\partial x_2} 
-\frac{1}{4}(x_1x_2^2-2x_2x_3)\frac{\partial}{\partial x_3}.
\end{align*}
\end{prop}

\begin{proof}
We only show the formula for $\dpil(N_1^+)$; 
the other case can be shown analogously.
By \eqref{eqn:dpi3} with $\fl =\fa$, 
one needs to compute
\begin{equation*}
d\pi_{\lambda^*}(N_1^+)f(\bar{n})
=\lambda^*((\Ad(\bar{n}^{-1})N_1^+)_\fa)f(\bar{n})
-\left(dR((\Ad(\cdot^{-1})N_1^+)_{\fn_-})f\right)(\bar{n}).
\end{equation*}

Write $X: =x_1 N_1^- + x_2 N_2^- + x_3 N_3^-$ and set $\bar{n}:=\exp(X)$.
Then it follows from a direct computation that 
$\Ad(\bar{n}^{-1})N_1^-$ is evaluated as
\begin{align*}
\Ad(\bar{n}^{-1})N_1^+
&=\exp(\ad(-X))N_1^+\\[3pt]
&=N_1^+
+x_1H_1
-x_1^2N_1^-
+\tfrac{1}{2}(x_1x_2-2x_3)N_2^-
+\tfrac{1}{2}(x_1^2x_2-2x_1x_3)N_3^-.
\end{align*}
Thus, the terms
$(\Ad(\bar{n}^{-1})N_1^+)_{\fa}$ and $(\Ad(\bar{n}^{-1})N_1^+)_{\fn_-}$
are 
\begin{align*}
(\Ad(\bar{n}^{-1})N_1^+)_{\fa}&=x_1 H_1,\\
(\Ad(\bar{n}^{-1})N_1^+)_{\fn_-}&=-x_1^2N_1^-
+\tfrac{1}{2}(x_1x_2-2x_3)N_2^-
+\tfrac{1}{2}(x_1^2x_2-2x_1x_3)N_3^-.
\end{align*}
Therefore, we have
\begin{align*}
\dpil(N_1^+)
&=\lambda^*((\Ad(\bar{n}^{-1})N_1^+)_\fa)f(\bar{n})
-\left(dR((\Ad(\cdot^{-1})N_1^+)_{\fn_-})f\right)(\bar{n})\\[3pt]
&=x_1 \lambda^*(H_1)
+x_1^2dR(N_1^-)
+\tfrac{1}{2}(x_1x_2-2x_3)dR(N_2^-)
+\tfrac{1}{2}(x_1^2x_2-2x_1x_3)dR(N_3^-).
\end{align*}
As $\lambda^*(H_1)=2-\lambda_1$,
Lemma \ref{lem:dR} and some simple manipulations lead to the formula.
\end{proof}

The next step is to compute $\hdpil(N_j^+)$ for $j=1,2$.
For later convenience, we instead consider $-\zeta_j\hdpil(N_j^+)$.
Remark that the equations $\hdpil(N_j^+)\psi(\zeta)=0$ 
for $\psi(\zeta) \in \Pol(\fn_+)$ are 
equivalent to $-\zeta_j\hdpil(N_j^+)\psi(\zeta)=0$.
For $i=1,2,3$, we write $\vartheta_i := \zeta_i \tfrac{\partial}{\partial \zeta_i}$
for the Euler operator for $\zeta_i$.

\begin{prop}\label{prop:dpih}
We have 
\begin{align*}
-\zeta_1\hdpil(N_1^+)
&=\lambda_1\vartheta_1 + (\vartheta_1^2-\vartheta_1)-\frac{1}{2}\vartheta_1\vartheta_2+
\frac{1}{2}\vartheta_1\vartheta_3+\frac{\zeta_1\zeta_2}{\zeta_3}\vartheta_3
+\frac{1}{4}\frac{\zeta_3}{\zeta_1\zeta_2}(\vartheta_1^2-\vartheta_1)\vartheta_2,\\[3pt]
-\zeta_2\hdpil(N_2^+)
&=\lambda_2\vartheta_2+(\vartheta_2^2-\vartheta_2)-\frac{1}{2}\vartheta_1\vartheta_2
+\frac{1}{2}\vartheta_2\vartheta_3-\frac{\zeta_1\zeta_2}{\zeta_3}\vartheta_3
-\frac{1}{4}\frac{\zeta_3}{\zeta_1\zeta_2}\vartheta_1(\vartheta_2^2-\vartheta_2).
\end{align*}
\end{prop}

\begin{proof}
Direct applications of 
the algebraic Fourier transform
\eqref{eqn:Weyl2} to $\dpil(N_j^+)$ in Proposition \ref{prop:dpi}.
\end{proof}

\subsection{T-saturation}
\label{sec:T}

Now we reduce the system of PDEs  $-\zeta_j\hdpil(N_j^+)\psi(\zeta)=0$
to that of ODEs. To do so, identify $\Pol(k,\ell)$ with the space
$\Pol_{\min(k,\ell)}[t]$ of polynomials $p(t)$ of one variable $t$ with
$\deg p(t) \leq \min(k,\ell)$ via the linear map
\begin{equation*}
T_{k,\ell} \colon 
\Pol_{\min(k,\ell)}[t] \stackrel{\sim}{\To}\Pol(k,\ell),
\quad
p(t) \longmapsto 
\zeta_1^k\zeta_2^\ell p(\tfrac{\zeta_3}{\zeta_1\zeta_2}).
\end{equation*}
We call $T_{k,\ell}$ a \emph{saturation map}.
Given $D \colon \Pol(k,\ell) \to \Pol(k,\ell)$, 
we define
\begin{equation*}
T^\sharp_{k,\ell}(D) \colon \Pol_{\min(k,\ell)}[t] \to \Pol_{\min(k,\ell)}[t]
\end{equation*}
by
\begin{equation*}
T_{k, \ell} \circ T^\sharp_{k,\ell}(D) = D \circ T_{k,\ell},
\end{equation*}
so that the following diagram commutes.
\begin{equation*}
\begin{tikzcd}[row sep=1cm, column sep=1cm]
\Pol_{\min(k,\ell)}[t] 
\arrow[d, "T^\sharp_{k,\ell}(D)"']  
\arrow[r, "T_{k, \ell}", "\sim"']  
&\Pol(k,\ell)
\arrow[d, "D"]  \\
\Pol_{\min(k,\ell)}[t] 
\arrow[r, "T_{k, \ell}", "\sim"']  
\arrow[ur,  pos=0.5, phantom, "\circlearrowleft"]
&\Pol(k,\ell)
\end{tikzcd}
\end{equation*}

As in \cite[Sect.\ 3.2]{KP2},
we call $T^\sharp_{k,\ell}(D)$ the \emph{T-saturation} 
of $D$.
Via the linear isomorphism $T_{k,\ell}$, the equation 
$T^\sharp_{k,\ell}(D)p(t)=0$ is equivalent to 
$DT_{k,\ell}(p)(\zeta)=0$ for a differential operator $D$ on $\Pol(k,\ell)$.
Proposition \ref{prop:dpih} shows that $-\zeta_1\hdpil(N_j^+)$ for $j=1,2$
preserve $\Pol(k,\ell)$. We then aim to compute the T-saturation
$T^\sharp_{k,\ell}(-\zeta_1\hdpil(N_j^+))$. 

The next lemma collects several useful formulas 
for some basic operators on $\Pol(k,\ell)$.
We write $\vartheta_t:=t\frac{d}{dt}$, the Euler operator for $t$.

\begin{lem}\label{lem:T}
The following hold for any $k, \ell \in \Z_{\geq 0}$.

\begin{enumerate}

\item[\emph{(1)}]
$T^\sharp_{k,\ell}(\vartheta_1)=k-\vartheta_t$


\item[\emph{(2)}]
$T^\sharp_{k,\ell}(\vartheta_2)=\ell-\vartheta_t$


\item[\emph{(3)}]
$T^\sharp_{k,\ell}(\vartheta_3)=\vartheta_t$


\item[\emph{(4)}]
$T^\sharp_{k,\ell}(\frac{\zeta_3}{\zeta_1\zeta_2})=t$


\item[\emph{(5)}]
$T^\sharp_{k,\ell}(\frac{\zeta_1\zeta_2}{\zeta_3})=t^{-1}$


\end{enumerate}
\end{lem}

\begin{proof}
We only prove (1) and (4); other formulas can be shown in a similar way.

For (1), it suffices to show
$T_{k,\ell} \circ T^\sharp_{k,\ell}(\vartheta_1)p(t)=T_{k,\ell}(kp(t)-\vartheta_tp(t))$.
Indeed, we have
\begin{align*}
T_{k,\ell} \circ T^\sharp_{k,\ell}(\vartheta_1)p(t)
&=\vartheta_1 \circ T_{k,\ell}(p(t))\\[3pt]
&=\vartheta_1 \left(\zeta_1^k\zeta_2^\ell\,  
p\left(\frac{\zeta_3}{\zeta_1\zeta_2}\right)\right)\\[3pt]
&=\zeta_1^k\zeta_2^\ell\left(
kp\left(\frac{\zeta_3}{\zeta_1\zeta_2}\right) - \frac{\zeta_3}{\zeta_1\zeta_2}
\frac{dp}{dt}\left(\frac{\zeta_3}{\zeta_1\zeta_2}\right)\right)\\[3pt]
&=T_{k,\ell}(kp(t)-\vartheta_tp(t)).
\end{align*}

For (4), we show 
$T_{k,\ell}\circ T^\sharp_{k,\ell}(\frac{\zeta_3}{\zeta_1\zeta_2})p(t)
=T_{k,\ell}(tp(t))$. We have
\begin{align*}
T_{k,\ell}\circ T^\sharp_{k,\ell}\left(\frac{\zeta_3}{\zeta_1\zeta_2}\right)p(t)
&=\frac{\zeta_3}{\zeta_1\zeta_2}T_{k,\ell}(p(t))\\[3pt]
&=\frac{\zeta_3}{\zeta_1\zeta_2} \cdot \zeta_1^k\zeta_2^\ell\,
p\left(\frac{\zeta_3}{\zeta_1\zeta_2}\right)\\[3pt]
&=T_{k,\ell}(tp(t)).
\end{align*}
\end{proof}

\begin{prop}\label{prop:tdpih}
We have 
\begin{align}
T^\sharp_{k,\ell}(-\zeta_1\hdpil(N_1^+))
&=
\frac{d}{dt} + (\lambda_1+k-\tfrac{1}{2}\ell-1)(k-\vartheta_t)
+\frac{1}{4}t(k-1-\vartheta_t)(k-\vartheta_t)(\ell-\vartheta_t), \label{prop:thdpi1}\\[3pt]
T^\sharp_{k,\ell}(-\zeta_2\hdpil(N_2^+))
&=
-\frac{d}{dt}+(\lambda_2-\tfrac{1}{2}k+\ell-1)(\ell-\vartheta_t)
-\frac{1}{4}t(k-\vartheta_t)(\ell-1-\vartheta_t)(\ell-\vartheta_t).\label{prop:thdpi2}
\end{align}
\end{prop}

\begin{proof}
Observe that, for differential operators $D_1, D_2$ on $\Pol(k,\ell)$,
we have $T^\sharp_{k,\ell}(D_1 \cdot D_2) 
= T^\sharp_{k,\ell}(D_1)\cdot T^\sharp_{k,\ell}(D_2)$.
The proposed formulas are then obtained by 
straightforward computations with the aid of Lemma \ref{lem:T}.
\end{proof}

For $a, b, \mu \in \C$, we write
\begin{equation*}
D^{(a,b)}(\mu;t):=
\frac{d}{dt}+(\mu+a-\tfrac{b}{2}-1)(a-\vartheta_t)
+\frac{1}{4}t(a-1-\vartheta_t)(a-\vartheta_t)(b-\vartheta_t).
\end{equation*}
It follows from \eqref{prop:thdpi1} and \eqref{prop:thdpi2} that 
\begin{equation*}
D^{(k,\ell)}(\lambda_1;t)=
T^\sharp_{k,\ell}(-\zeta_1\hdpil(N_1^+))
\quad
\text{and}
\quad
D^{(\ell,k)}(\lambda_2;-t)=
T^\sharp_{k,\ell}(-\zeta_2\hdpil(N_2^+)).
\end{equation*}

For $\lambda_j \in \C$ for $j=1,2$ and $k,\ell \in \Z_{\geq 0}$, we put 
\begin{equation}\label{eqn:Solsec5}
\wSol_{(k,\ell)}(-\lambda_1,-\lambda_2)
:=\{p(t) \in \Pol_{\min(k,\ell)}[t] :
D^{(k,\ell)}(\lambda_1;t)p(t)=D^{(\ell,k)}(\lambda_2;-t)p(t)=0\}.
\end{equation}
Then the map $T_{k,\ell} \colon 
\Pol_{\min(k,\ell)}[t] \stackrel{\sim}{\To}\Pol(k,\ell)$
induces a linear isomorphism
\begin{equation}\label{eqn:SolSol}
\wSol_{(k,\ell)}(-\lambda_1,-\lambda_2) 
\stackrel{\sim}{\To}
\Sol_{(k,\ell)}(-\lambda_1,-\lambda_2).
\end{equation}
Thus, to solve the F-system \eqref{eqn:Fsys3},
it suffices to consider 
\begin{equation}\label{eqn:Deq}
D^{(k,\ell)}(\lambda_1;t)p(t)=0
\quad
\text{and}
\quad
D^{(\ell,k)}(\lambda_2;-t)p(t)=0
\end{equation}
for $p(t) \in \Pol_{\min(k,\ell)}[t]$.
The next proposition reduces the range of $k,\ell$ to consider.

\begin{prop}\label{prop:duality}
There exists a linear isomorphism
\begin{equation*}
\wSol_{(k,\ell)}(-\lambda_1,-\lambda_2)
\stackrel{\sim}{\To}
\wSol_{(\ell, k)}(-\lambda_2,-\lambda_1),
\quad
p(t) \longmapsto p(-t).
\end{equation*}
\end{prop}

\begin{proof}
By definition,  polynomials 
$p(t) \in \wSol_{(k,\ell)}(-\lambda_1,-\lambda_2)$ satisfy
\begin{equation*}
D^{(k,\ell)}(\lambda_1;t)p(t)=D^{(\ell,k)}(\lambda_2;-t)p(t)=0.
\end{equation*}
Thus, for $s :=-t$, we have
\begin{align*}
D^{(\ell,k)}(\lambda_2;t)p(-t)&=
D^{(\ell,k)}(\lambda_2;-s)p(s)=0,\\[3pt]
D^{(k,\ell)}(\lambda_1;-t)p(-t)&=
D^{(k,\ell)}(\lambda_1;s)p(s)=0,
\end{align*}
which show $p(-t) \in \wSol_{(\ell,k)}(-\lambda_2,-\lambda_1)$.
\end{proof}

In the next section, we explicitly compute 
$\wSol_{(k,\ell)}(-\lambda_1,-\lambda_2)$.

\section{Proof of the classification and construction of $\D$ and $\varphi$: Part II}
\label{sec:proof2}

In this section we perform Steps 3 and 4 of the recipe for the F-method 
given at the beginning of Section \ref{sec:proof1} to prove
Theorems 
\ref{thm:DIO1},
\ref{thm:Hom1}, 
\ref{thm:DIO2},
and
\ref{thm:Hom2}.
The proof is achieved in Section \ref{sec:Step4}.

\subsection{Step 3: Solve the F-system \eqref{eqn:Fsys3}}
\label{sec:Fsys1}

Via the isomorphism \eqref{eqn:SolSol}, 
we first aim to solve the system \eqref{eqn:Deq} of 
ODEs for each $k, \ell \in \Z_{\geq 0}$.
Figure \ref{graph:lattice} depicts the lattices of $(k,\ell)$ in consideration.

\begin{figure}[H]
\caption{Lattices of $(k,\ell)$ with $k, \ell \in \Z_{\geq 0}$}
\begin{center}
\begin{tikzpicture}[xscale =1, yscale=1]
\tikzset{axes/.style={}}
 
  \draw[-latex, name path = xline] 
 (-0.7, 0) -- (5, 0) 
 node[right, font = \normalsize]{$k$};
  
 \draw[-latex] 
 (0, -.5) -- (0, 5) 
 node[left, font = \normalsize] {$\ell$};

\draw (0, .05) -- (0, -.05) node[below, font = \normalsize] {0 \phantom{aa}};  
\draw (1, .05) -- (1, -.05) node[below, font = \normalsize] {1};
\draw (2, .05) -- (2, -.05) node[below, font = \normalsize] {2};
\draw (3, .05) -- (3, -.05) node[below, font = \normalsize] {3};
\draw (4, .05) -- (4, -.05) node[below, font = \normalsize] {4};

\foreach \y in {1, 2, 3, 4}{
\draw (.05, \y) -- (-.05, \y) node[left, font = \normalsize] {\y}; }

\draw[black, name path = para] plot[domain = 0:4.5, samples = 100]
({\x}, {(\x)});

\draw[blue!40, fill = black] (0, 0) circle[radius = .07cm];


\node[mark size=2.5pt] at (0,1) {\pgfuseplotmark{square*}};
\node[mark size=2.5pt] at (0,2) {\pgfuseplotmark{square*}};
\node[mark size=2.5pt] at (0,3) {\pgfuseplotmark{square*}};
\node[mark size=2.5pt] at (0,4) {\pgfuseplotmark{square*}};


\node[mark size=3.5pt] at (1,0) {\pgfuseplotmark{triangle*}};
\node[mark size=3.5pt] at (2,0) {\pgfuseplotmark{triangle*}};
\node[mark size=3.5pt] at (3,0) {\pgfuseplotmark{triangle*}};
\node[mark size=3.5pt] at (4,0) {\pgfuseplotmark{triangle*}};


\node[mark size=3.5pt] at (1,1) {\pgfuseplotmark{diamond*}};
\node[mark size=3.5pt] at (2,2) {\pgfuseplotmark{diamond*}};
\node[mark size=3.5pt] at (3,3) {\pgfuseplotmark{diamond*}};
\node[mark size=3.5pt] at (4,4) {\pgfuseplotmark{diamond*}};

\node[mark size=3pt] at (1,2) {\pgfuseplotmark{o}};
\node[mark size=3pt] at (1,3) {\pgfuseplotmark{o}};
\node[mark size=3pt] at (1,4) {\pgfuseplotmark{o}};


\node[mark size=3pt] at (2,1) {\pgfuseplotmark{o}};
\node[mark size=3pt] at (2,3) {\pgfuseplotmark{o}};
\node[mark size=3pt] at (2,4) {\pgfuseplotmark{o}};


\node[mark size=3pt] at (3,1) {\pgfuseplotmark{o}};
\node[mark size=3pt] at (3,2) {\pgfuseplotmark{o}};
\node[mark size=3pt] at (3,4) {\pgfuseplotmark{o}};


\node[mark size=3pt] at (4,1) {\pgfuseplotmark{o}};
\node[mark size=3pt] at (4,2) {\pgfuseplotmark{o}};
\node[mark size=3pt] at (4,3 ){\pgfuseplotmark{o}};


\end{tikzpicture}
\end{center}\label{graph:lattice}
\end{figure}
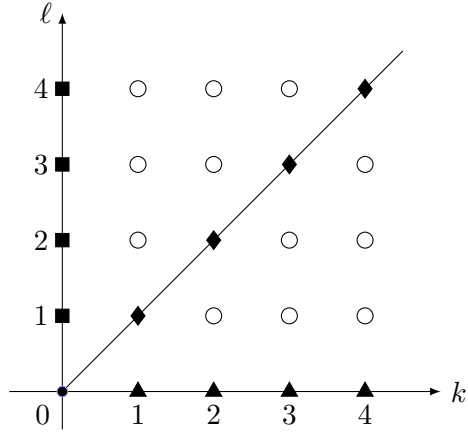

We consider the following five cases separately.

\begin{enumerate}

\item $(k,\ell) = (0,0)$: the origin (\textbullet)
\vskip 0.1in

\item $(k, \ell) \in (1+\Z_{\geq 0}) \times \{0\}$: the $k$-axis $\setminus \{(0,0)\}$
($\blacktriangle$)
\vskip 0.1in

\item $(k, \ell) \in \{0\}  \times  (1+\Z_{\geq 0})$: the $\ell$-axis $\setminus \{(0,0)\}$
($\blacksquare$)
\vskip 0.1in

\item $(k,\ell) \in (1+\Z_{\geq 0})^2$ with $k \neq \ell$:
any lattice neither on the $k$-axis, $\ell$-axis, nor diagonal line
($\circ$)
\vskip 0.1in

\item $(k,\ell) \in (1+\Z_{\geq 0})^2$ with $k=\ell$:
the diagonal line $\setminus \{(0,0)\}$
($\blacklozenge$)
\vskip 0.1in

\end{enumerate}

We write
\begin{align}
p_+^{(k,\ell)}(t) &:= \sum_{m=0}^{\min(k,\ell)} \frac{1}{2^m}m!\binom{k}{m}\binom{\ell}{m}
t^m, \label{eqn:p+0}\\[3pt]
p_-^{(k,\ell)}(t) &:= \sum_{m=0}^{\min(k,\ell)} \frac{(-1)^m}{2^m}m!\binom{k}{m}\binom{\ell}{m}
t^m, \label{eqn:p-0}\\[3pt]
p_{\Cay}^{(s;k)}(t)&:= \sum_{m=0}^{k} \frac{1}{2^m}\binom{k}{m}\Cay_m(s;k)t^m.\label{eqn:pc0}
\end{align}

\begin{rem}
The polynomials $p_{\pm}^{(k,\ell)}(t)$ can also be expressed as 
\begin{equation*}
p_{\pm}^{(k,\ell)}(t)={}_2F_0(-k,-\ell;\pm\tfrac{t}{2}),
\end{equation*}
where ${}_2F_0(a,b;z)$ is the generalized 
hypergeometric series 
\begin{equation*}
{}_2F_0(a,b;z) =\sum_{m=0}^\infty \frac{(a)_m (b)_m}{m!} z^k.
\end{equation*}
Here, $(r)_m$ is the Pochhammer symbol, i.e.,
$(r)_m (= r^{\overline{m}})= r(r+1)(r+2)\cdots (r+m-1)=\frac{\Gamma(r+m)}{\Gamma(r)}$.
Indeed, as $r^{\underline{m}} = (-1)^m(-r)_m$ and 
$\binom{r}{m} = \frac{r^{\underline{m}}}{m!}$, we have 
\begin{align*}
p_{\pm}^{(k,\ell)}(t) 
&=\sum_{m=0}^{\min(k,\ell)} \frac{(\pm1)^m}{2^m}m!\binom{k}{m}\binom{\ell}{m}t^m\\[3pt]
&=\sum_{m=0}^{\min(k,\ell)} \frac{(-k)_m(-\ell)_m}{m!} \left(\pm\frac{t}{2}\right)^m\\[3pt]
&={}_2F_0(-k,-\ell;\pm \tfrac{t}{2}).
\end{align*}
\end{rem}

In the following,
we deal with Case (1) in Proposition \ref{prop:1},
Cases (2) and (3) in Proposition \ref{prop:2},
Case (4) in Proposition \ref{prop:3},
and Case (5) in Proposition \ref{prop:4}.

\begin{prop}\label{prop:1}
Let $(k,\ell) = (0,0)$. Then we have
\begin{equation*}
\wSol_{(0, 0)}(-\lambda_1, -\lambda_2) = \C
\quad 
\text{for all $\lambda_1 ,\lambda_2 \in \C$}.
\end{equation*}
\end{prop}

\begin{proof}
As $\Pol_{\min(0,0)}[t] = \C$, the set  
$\wSol_{(0,0)}(-\lambda_1,-\lambda_2)$
is given by
\begin{equation*}
\wSol_{(0,0)}(-\lambda_1,-\lambda_2)
=\{u \in \C : 
D^{(0,0)}(\lambda_1;t)u=
D^{(0,0)}(\lambda_2;-t)u=0\}.
\end{equation*}
Since
\begin{align*}
D^{(0,0)}(\lambda_1;t)&=
\frac{d}{dt} - (\lambda_1-1)\vartheta_t - \frac{1}{4}t(\vartheta_t+1)\vartheta_t^2,\\[3pt]
D^{(0,0)}(\lambda_2;-t)&=
-\frac{d}{dt} - (\lambda_2-1)\vartheta_t + \frac{1}{4}t(\vartheta_t+1)\vartheta_t^2,
\end{align*}
it is clear that
\begin{equation*}
D^{(0,0)}(\lambda_1;t)u=
D^{(0,0)}(\lambda_2;-t)u=0
\quad
\text{for all $\lambda_1, \lambda_2 \in \C$},
\end{equation*}
which shows the proposition.
\end{proof}

\begin{prop}\label{prop:2}
The following hold.
\begin{enumerate}
\item[\emph{(a)}]
Let $(k,\ell) \in (1+\Z_{\geq 0}) \times \{0\}$. Then the following conditions on
$(\lambda_1,\lambda_2) \in \C^2$ are equivalent.
\begin{enumerate}
\item[\emph{(i)}]
$\wSol_{(k,0)}(-\lambda_1,-\lambda_2) \neq \{0\}$.
\item[\emph{(ii)}]
$(\lambda_1,\lambda_2) \in \{1-k\} \times \C $.
\end{enumerate}
\vskip 0.1in

\item[\emph{(b)}]
Let $(k,\ell) \in \{0\} \times (1+\Z_{\geq 0})$. Then the following conditions on
$(\lambda_1,\lambda_2) \in \C^2$ are equivalent.
\begin{enumerate}
\item[\emph{(i)}]
$\wSol_{(0,\ell)}(-\lambda_1,-\lambda_2) \neq \{0\}$.

\item[\emph{(ii)}]
$(\lambda_1,\lambda_2) \in \C \times \{1-\ell\}$.

\end{enumerate}
\end{enumerate}
Moreover, we have
\begin{align*}
\wSol_{(k,0)}( -(1-k), -\lambda_2) &= \C,\\[3pt]
\wSol_{(0,\ell)}(-\lambda_1, -(1-\ell)) &= \C.
\end{align*}
\end{prop}

\begin{proof}
Case (b) simply follows from Proposition \ref{prop:duality} and Case (a).
Thus, it suffices to consider $(k,\ell) \in (1+\Z_{\geq 0}) \times  \{0\}$.
As $\Pol_{\min(k,0)}[t] = \C$, 
the set 
$\wSol_{(k,0)}(-\lambda_1,-\lambda_2)$
is given as
\begin{equation*}
\wSol_{(k,0)}(-\lambda_1,-\lambda_2)
=\{u \in \C : 
D^{(k,0)}(\lambda_1;t)u=
D^{(0,k)}(\lambda_2;-t)u=0\},
\end{equation*}
where
\begin{align*}
D^{(k,0)}(\lambda_1;t)&=
\tfrac{d}{dt} + (\lambda_1+k-1)(k-\vartheta_t) 
- \tfrac{1}{4}t(k-1-\vartheta_t)(k-\vartheta_t)\vartheta_t,\\[3pt]
D^{(0,k)}(\lambda_2;-t)&=
-\tfrac{d}{dt} - (\lambda_2-\tfrac{1}{2}k-1)\vartheta_t 
- \tfrac{1}{4}t(1+\vartheta_t)(k-\vartheta_t)\vartheta_t.
\end{align*}
Therefore, 
$D^{(k,0)}(\lambda_1;t)u=D^{(0,k)}(\lambda_2;-t)u=0$
if and only if $k(\lambda_1+k-1)=0$, which is equivalent to 
$(\lambda_1, \lambda_2) \in \{1-k\} \times \C$ as $k \neq 0$.
\end{proof}

\begin{prop}\label{prop:3}
Let $k, \ell \in 1+\Z_{\geq 0}$ with $ k\neq \ell$. Then the following 
two conditions on $(\lambda_1,\lambda_2) \in \C^2$  are equivalent.
\begin{enumerate}
\item[\emph{(i)}] 
$\wSol_{(k,\ell)}(-\lambda_1,-\lambda_2) \neq \{0\}$.

\item[\emph{(ii)}] 
$(\lambda_1,\lambda_2) \in \{(1-k, 1-\ell+k), (1-k+\ell ,1-\ell )\}$.
\end{enumerate}
Moreover, we have 
\begin{align*}
\wSol_{(k,\ell)}(-(1-k), -(1-\ell+k)) &= \C  p_+^{(k,\ell)}(t),\\[3pt]
\wSol_{(k,\ell)}(-(1-k+\ell), -(1-\ell)) &= \C  p_-^{(k,\ell)}(t).
\end{align*}
\end{prop}

\begin{proof}
It follows from Proposition \ref{prop:duality} that it suffices to consider $k > \ell$.
As $\Pol_{\min(k,\ell)}[t] = \Pol_\ell[t]$, 
the set 
$\wSol_{(k,\ell)}(-\lambda_1,-\lambda_2)$
is given as
\begin{equation*}
\wSol_{(k,\ell)}(-\lambda_1,-\lambda_2)
=\{p(t) \in  \Pol_\ell[t] : 
D^{(k,\ell)}(\lambda_1;t)p(t)=
D^{(\ell,k)}(\lambda_2;-t)p(t)=0\}.
\end{equation*}

Take $p(t)=  \sum_{j=0}^\ell a_jt^j\in \Pol_\ell[t]$.
For simplicity, we put
\begin{align*}
\gamma_\lambda(k,\ell)&:=\lambda_1+k-\tfrac{\ell}{2}-1,\\[3pt]
\delta_\lambda(k,\ell)&:=\lambda_2-\tfrac{k}{2}+\ell-1,
\end{align*}
so that 
\begin{align*}
D^{(k,\ell)}(\lambda_1;t)
&=
\tfrac{d}{dt} + \gamma_\lambda(k,\ell)(k-\vartheta_t)
+\tfrac{1}{4}t(k-1-\vartheta_t)(k-\vartheta_t)(\ell-\vartheta_t), \\[3pt]
D^{(\ell,k)}(\lambda_2;-t)
&=
-\tfrac{d}{dt}+\delta_\lambda(k,\ell)(\ell-\vartheta_t)
-\tfrac{1}{4}t(k-\vartheta_t)(\ell-1-\vartheta_t)(\ell-\vartheta_t).
\end{align*}

Write $D^{(k,\ell)}(\lambda_1;t)p(t) = \sum_{m=0}^\ell A_m t^m$
and $D^{(\ell,k)}(\lambda_2;-t)p(t) = \sum_{m=0}^\ell B_m t^m$.
A direct computation shows that the coefficients $A_m$ and $B_m$ are given as
\begin{align*}
A_0&=\gamma_\lambda(k,\ell)ka_0+a_1,\\[3pt]
A_1&=\tfrac{1}{4}(k-1)k\ell a_0 + \gamma_\lambda(k,\ell)(k-1)a_1+2a_2,\\[3pt]
A_m&=\tfrac{1}{4}(k-m)(k-m+1)(\ell-m+1)a_{m-1}+\gamma_\lambda(k,\ell)(k-m)a_m
+(m+1)a_{m+1} \quad
\text{for $m=2,\ldots, \ell-1$},\\[3pt]
A_\ell&=\tfrac{1}{4}(k-\ell)(k-\ell+1) a_{\ell-1} +\gamma_\lambda(k,\ell)(k-\ell)a_\ell,
\end{align*}
and
\begin{align*}
B_0&=\delta_\lambda(k,\ell)\ell a_0-a_1,\\[3pt]
B_1&=-\tfrac{1}{4}(\ell-1)k\ell a_0 + \delta_\lambda(k,\ell)(\ell-1)a_1-2a_2,\\[3pt]
B_m&=-\tfrac{1}{4}(\ell-m)(k-m+1)(\ell-m+1)a_{m-1}+\delta_\lambda(k,\ell)(\ell-m)a_m
-(m+1)a_{m+1} \quad
\text{for $m=2,\ldots, \ell-1$}, \\[3pt]
B_\ell&=0.
\end{align*}
Then, to solve the system of differential equations 
$D^{(k,\ell)}(\lambda_1;t)p(t)=D^{(\ell,k)}(\lambda_2;-t)p(t)=0$,
it suffices to consider that of linear equations
$A_m=0$ and $B_m=0$ for $m=0,1,\ldots, \ell$.

If $a_0=0$, then the recurrence relations force $p(t)\equiv 0$. 
Thus, without loss of generality, we assume that $a_0=1$.
Since $k,\ell, k-\ell \neq 0$, by the system of linear equations
\begin{align*}
A_0 + B_0 &= 0,\\
A_1 + B_1 &= 0,
\end{align*}
one obtains
\begin{equation*}
\lambda_1 \in \{1-k,1-k+\ell\}
\quad
\text{and}
\quad 
-\ell \delta_\lambda(k,\ell) = k\gamma_\lambda(k,\ell).
\end{equation*}

First, suppose that $\lambda_1=1-k$. In this case, we have
\begin{equation*}
\lambda_2=1-\ell+k,
\quad
\gamma_\lambda(k,\ell)=-\frac{\ell}{2},
\quad
\delta_\lambda(k,\ell)=\frac{k}{2}.
\end{equation*}
Then, by solving the system of equations $A_m=0$, one obtains
\begin{equation}\label{eqn:am}
a_m=\frac{1}{2^m}m! \binom{k}{m}\binom{\ell}{m}
\quad
\text{for $1\leq m \leq \ell$},
\end{equation}
which amounts to $p(t)= p_+^{(k,\ell)}(t)$.
It is easy to check that the constants $a_m$ in \eqref{eqn:am} 
also satisfy the other system of equations $B_m =0$.

Next, observe that if $\lambda_1 = 1-k+\ell$, 
then $\lambda_2=1-\ell$. It thus follows from 
Proposition \ref{prop:duality} and the former case that, in this case, 
we have $p(t)= p_-^{(k,\ell)}(t)$.
This completes the proof.
\end{proof}

\begin{prop}\label{prop:4}
Let $k, \ell \in 1+\Z_{\geq 0}$ with $ k=\ell$. Then the following 
two conditions on $(\lambda_1,\lambda_2) \in \C^2$  are equivalent.
\begin{enumerate}
\item[\emph{(i)}] 
$\wSol_{(k,k)}(-\lambda_1,-\lambda_2) \neq \{0\}$.

\item[\emph{(ii)}] 
$(\lambda_1,\lambda_2) \in 
\{(\tfrac{1}{2}(2-k-s), \tfrac{1}{2}(2-k+s)) : s\in \C\}$.
\end{enumerate}
Moreover, we have 
\begin{equation*}
\wSol_{(k,k)}(-\tfrac{1}{2}(2-k-s), -\tfrac{1}{2}(2-k+s))= \C p_{\Cay}^{(s;k)}(t).
\end{equation*}
\end{prop}

\begin{proof}
In this case we have
\begin{equation*}
\wSol_{(k,k)}(-\lambda_1,-\lambda_2)
=\{p(t) \in  \Pol_k[t] : 
D^{(k,k)}(\lambda_1;t)p(t)=
D^{(k,k)}(\lambda_2;-t)p(t)=0\},
\end{equation*}
where
\begin{align*}
D^{(k,k)}(\lambda_1;t)
&=
\tfrac{d}{dt} + (\lambda_1+\tfrac{1}{2}k-1)(k-\vartheta_t)
+\tfrac{1}{4}t(k-1-\vartheta_t)(k-\vartheta_t)^2, \\[3pt]
D^{(k,k)}(\lambda_2;-t)
&=
-\tfrac{d}{dt}+(\lambda_2+\tfrac{1}{2}k-1)(k-\vartheta_t)
-\tfrac{1}{4}t(k-1-\vartheta_t)(k-\vartheta_t)^2.
\end{align*}

Now we put
\begin{align*}
D^{(k)}_+&:=D^{(k,k)}(\lambda_1;t)+D^{(k,k)}(\lambda_2;-t)
=(\lambda_1+\lambda_2+k-2)(k-\vartheta_t),\\[3pt]
D^{(k)}_-&:=D^{(k,k)}(\lambda_1;t)-D^{(k,k)}(\lambda_2;-t)
=2\tfrac{d}{dt}+(\lambda_1-\lambda_2)(k-\vartheta_t)
+\tfrac{1}{2}t(k-1-\vartheta_t)(k-\vartheta_t)^2.
\end{align*}
It suffices to consider $D^{(k)}_+p(t)=D^{(k)}_-p(t)=0$.

For $p(t)=\sum_{m=0}^k a_m t^m \in \Pol_k[t]$,
the system $D^{(k)}_+p(t)=D^{(k)}_-p(t)=0$ of equations forces
\begin{equation}\label{eqn:cond1}
\lambda_1+\lambda_2+k-2=0
\end{equation}
and
\begin{equation}\label{eqn:cond2}
a_m=\frac{-1}{4m}(k-m+1)
\{(k-m+2)^2a_{m-2} + 2(\lambda_1-\lambda_2)a_{m-1}\}
\end{equation}
for $m=1,\ldots, k$ with $a_{-1}=0$. It follows from \eqref{eqn:cond2} that
if $a_0=0$, then $p(t)\equiv 0$. Then, without loss of generality, we assume that 
$a_0=1$.

Now, we put
\begin{equation*}
\Lambda_{\Cay}(k):=
\{(\tfrac{1}{2}(2-k-s), \tfrac{1}{2}(2-k+s)): s\in \C\}.
\end{equation*}
Then \eqref{eqn:cond1} is equivalent to 
$(\lambda_1,\lambda_2) \in \Lambda_{\Cay}(k)$.
For $(\lambda_1,\lambda_2) \in \Lambda_{\Cay}(k)$, the coefficients
$a_m$ are given as
\begin{equation*}
a_m=\frac{-1}{4m}(k-m+1)
\{(k-m+2)^2a_{m-2} - 2sa_{m-1}\}.
\end{equation*}

\begin{claim}\label{claim:am}
We have 
\begin{equation*}
a_m=\frac{1}{2^m}\binom{k}{m}\Cay_m(s;k),
\end{equation*}
where $\Cay_m(s;k)$ is a Cayley continuant defined 
at the beginning of Section \ref{sec:construction}.
\end{claim}

If Claim \ref{claim:am} holds, then $p(t)=p_{\Cay}^{(s;k)}(t)$.
It thus suffices to show the claim. Put 
\begin{equation*}
b_m:=\frac{1}{2^m}\binom{k}{m}\Cay_m(s;k).
\end{equation*}
We wish to show $a_m = b_m$ for $m=0,1,2\ldots, k$.
We prove it by induction.

As $\Cay_0(s;k) =1$ and $\Cay_1(s;k)=s$, we have 
$b_0 = 1$ and 
\begin{equation*}
b_1 = \frac{1}{2}\binom{k}{1}\Cay_1(s;k)=\frac{1}{2}ks,
\end{equation*}
while $a_0=1$ by assumption and
\begin{equation*}
a_1 = \frac{-1}{4}(k-1+1)\{(k-1+2)^2a_{-1} - 2sa_0\}=\frac{1}{2}ks.
\end{equation*}
Now assume that $a_m=b_m$ for $m=0, 1, \ldots, j-1$. Then, for $m=j$, we have
\begin{align}
&(k-j+2)^2 a_{j-2}-2sa_{j-1} \nonumber \\[3pt]
&=(k-j+2)^2 \frac{1}{2^{j-2}} \binom{k}{j-2} \Cay_{j-2}(s;k)
-2s\frac{1}{2^{j-1}}\binom{k}{j-1}\Cay_{j-1}(s;k) \nonumber \\[3pt]
&=\frac{-1}{2^{j-2}}\binom{k}{j-1}
\{s\Cay_{j-1}(s;k) - (j-1)(k-j+2)\Cay_{j-2}(s;k)\}.\label{eqn:acay}
\end{align}
It is known that Cayley continuants $\Cay_{j}(s;k)$ satisfy
the following recurrence relation (cf.\ \cite[(3)]{MT05}):
\begin{equation*}
\Cay_j(s;k)=s\Cay_{j-1}(s;k) - (j-1)(k-j+2)\Cay_{j-2}(s;k).
\end{equation*}
Therefore it follows from \eqref{eqn:acay} that
\begin{equation*}
(k-j+2)^2 a_{j-2}-2sa_{j-1}
=\frac{-1}{2^{j-2}}\binom{k}{j-1}\Cay_j(s;k).
\end{equation*}
Hence,
\begin{align*}
a_j 
&= \frac{-1}{4j}(k-j+1)\{(k-j+2)^2 a_{j-2}-2sa_{j-1}\}\\[3pt]
&= \frac{-1}{4j}(k-j+1)\cdot \frac{-1}{2^{j-2}}\binom{k}{j-1}\Cay_j(s;k)\\[3pt]
&=\frac{1}{2^{j}}\binom{k}{j}\Cay_j(s;k).
\end{align*}
By induction, the claim holds.
\end{proof}

Recall from \eqref{eqn:SolSol} that 
the saturation map $T_{k,\ell}$ gives rise to the linear isomorphism
\begin{align*}
\wSol_{(k,\ell)}(-\lambda_1,-\lambda_2) 
&\stackrel{\sim}{\To}
\Sol_{(k,\ell)}(-\lambda_1,-\lambda_2),\\
p(t) &\longmapsto T_{k,\ell}p(t):=
\zeta_1^k\zeta_2^\ell  p(\tfrac{\zeta_3}{\zeta_1\zeta_2}).
\end{align*}
For simplicity, we write
\begin{equation*}
T_{k,\ell}( p_\pm) := T_{k,\ell}\, p_\pm^{(k,\ell)}(t)
\quad
\text{and}
\quad
T_{k}(p_{\Cay}^{(s)}) := T_{k,k}\, p_{\Cay}^{(s;k)}(t).
\end{equation*}
Now we classify 
$\Sol(\eps,\delta;\lambda,\nu)$.

\begin{thm}\label{thm:Sol}
The space 
$\Sol(\eps,\delta;\lambda,\nu)$
of polynomial solutions to the F-system \eqref{eqn:Fsys3}
is classified as follows.
\begin{equation*}
\Sol(\eps,\delta;\lambda,\nu)
=
\begin{cases}
\C \id & \emph{if $(\delta;\nu) = (\eps; \lambda)$},\\[3pt]
\C \zeta_1^{1-\lambda_1} & \emph{if $(\eps, \delta; \lambda,\nu) \in \Supp_1$}, \\[3pt]
\C \zeta_
2^{1-\lambda_2} & \emph{if $(\eps, \delta; \lambda,\nu) \in \Supp_2$}, \\[3pt]
\C T_{1-\lambda_1, 2-\lambda_1-\lambda_2}(p_+) & \emph{if $(\eps, \delta; \lambda,\nu) \in \Supp_+$}, \\[3pt]
\C T_{2-\lambda_1-\lambda_2, 1-\lambda_2}(p_-) & \emph{if $(\eps, \delta; \lambda,\nu) \in \Supp_-$}, \\[3pt]
\C T_{2-\lambda_1-\lambda_2}(p_{\Cay}^{(\lambda_2-\lambda_1)}) & \emph{if $(\eps, \delta; \lambda,\nu) \in \Supp_{\Cay}$}, \\[3pt]
\{0\} & \emph{otherwise},
\end{cases}
\end{equation*}
where $\Supp_1,\Supp_2, \ldots, \Supp_{\Cay}$ are the sets defined in Section \ref{sec:DIO}.
\end{thm}

\begin{proof}
It follows from Corollary \ref{cor:Sol} that 
$\Sol(\eps,\delta;\lambda,\nu)\neq \{0\}$ if and only 
if $\Sol_{(k,\ell)}(-\lambda_1,-\lambda_2)\neq \{0\}$ for some
$k,\ell \in \Z_{\geq 0}$, which is further equivalent to
$\wSol_{(k,\ell)}(-\lambda_1,-\lambda_2)\neq \{0\}$.
One then rewrites $k,\ell \in \Z_{\geq 0}$ and $s \in \C$ 
in terms of $\lambda_1, \lambda_2$ 
given in Propositions \ref{prop:1}, \ref{prop:2}, \ref{prop:3}, and \ref{prop:4}
to obtain the proposed parameters in $\Supp_1,\Supp_2, \ldots, \Supp_{\Cay}$.
The generators simply follow from the application of 
the saturation map $T_{k,\ell}$ to $1$, $p_{\pm}^{(k,\ell)}$, and $p_{\Cay}^{(s;k)}$.
Now the theorem holds.
\end{proof}

\subsection{Step 4:
Apply $\Symb_0^{-1}$ and $F_c^{-1}$ 
to the solution $\psi \in \Sol(\eps,\delta;\lambda,\nu)$}
\label{sec:Step4}

Recall from \eqref{eqn:Finv2} and \eqref{eqn:symb} that, 
given $\psi(\zeta) \in \Sol(\eps,\delta;\lambda,\nu)$ with
\begin{equation*}
\psi(\zeta) 
= \sum_{\mathbf{r} \in \Z_{\geq 0}^3}a_\mathbf{r}  \zeta_1^{r_1}\zeta_2^{r_2}\zeta_3^{r_3},
\end{equation*}
we have 
\begin{align*}
\Symb_0^{-1}(\psi(\zeta))
&=
\sum_{\mathbf{r} \in \Z_{\geq 0}^3}a_\mathbf{r}  
dR(\sym_{|\mathbf{r}|}((N_1^-)^{r_1}(N_3^-)^{r_2}(N_3^-)^{r_3}))
=\sum_{\mathbf{r} \in \Z_{\geq 0}^3}a_\mathbf{r}  
\sym_{|\mathbf{r}|}(\D_1^{r_1}\D_2^{r_2}\D_3^{r_3}),\\[3pt]
F_c^{-1}(\psi(\zeta))
&=\sum_{\mathbf{r} \in \Z_{\geq 0}^3}a_\mathbf{r}  
\sym_{|\mathbf{r}|}((N_1^-)^{r_1}(N_2^-)^{r_2}(N_3^-)^{r_3}),
\end{align*}
where $\D_j = dR(N_j^-)$ for $j=1,2,3$ and
$\sym_{|\mathrm{r}|}$ denotes the symmetrization.

Now we are ready to prove 
Theorems 
\ref{thm:DIO1},
\ref{thm:Hom1}, 
\ref{thm:DIO2},
and
\ref{thm:Hom2}.

\begin{proof}[Proof of Theorems 
\ref{thm:DIO1},
\ref{thm:Hom1}, 
\ref{thm:DIO2},
and
\ref{thm:Hom2}]

It is clear that $\D_j^k=\Symb_0^{-1}(\zeta_j^k)$ for $j=1,2$,
$\D_\pm^{(k,\ell)} = \Symb_0^{-1}(T_{k,\ell}(p_{\pm}))$, and 
$\D_{\Cay}^{(s;k)}=\Symb_0^{-1}(T_{k}(p_{\Cay}^{(s)}))$. Therefore,
Theorems \ref{thm:DIO1} and \ref{thm:DIO2} follow 
from Theorems \ref{thm:symb} and \ref{thm:Sol}.
Likewise, let 
$\eta_{\pm}^{(k,\ell)}(N_1^-,N_2^-,N_3^-)$ and 
$\eta_{\Cay}^{(s;k)}(N_1^-,N_2^-,N_3^-)$ be the polynomials
defined in \eqref{eqn:eta1}, \eqref{eqn:eta2}, and \eqref{eqn:eta3}.
By Corollary \ref{cor:Fc}, we have
$(N_j^-)^k=F_c^{-1}(\zeta_j^k)$,
$\eta_{\pm}^{(k,\ell)}(N_1^-,N_2^-,N_3^-) =F_c^{-1}(T_{k,\ell}(p_{\pm}))$,
and
$\eta_{\Cay}^{(s;k)}(N_1^-,N_2^-,N_3^-) = F_c^{-1}(T_{k}(p_{\Cay}^{(s)}))$.
Hence, Theorems \ref{thm:Fmethod} and \ref{thm:Sol} yield
Theorems \ref{thm:Hom1} and \ref{thm:Hom2}. This ends the proof.
\end{proof}

\section{Uniform expressions of $p_{\pm}^{(k,\ell)}(t)$ and $p_{\Cay}^{(s;k)}(t)$
by Jacobi polynomials $P_m^{(\ga,\beta)}(z)$}
\label{sec:Uniform}

In this short section we aim to show that
the key polynomials $p_{\pm}^{(k,\ell)}(t)$ and $p_{\Cay}^{(s;k)}(t)$ obtained in 
Section \ref{sec:proof2} can be uniformly by Jacobi polynomials
$P_m^{(\ga,\beta)}(z)$ for certain parameters $\ga, \beta$ at $z=0$.
This is accomplished in Theorem \ref{thm:unif}.
We also discuss that $p_{\Cay}^{(s;k)}(t)$ can be 
given also by binary Krawtchouk polynomials $K_m(x;y)$
for particular values of $s$,
which play a role in the factorization identities in Section \ref{sec:factorization}.

\subsection{Jacobi polynomials $P^{(\ga,\gb)}_m(z)$} 
\label{sec:Jacobi}

First, we briefly recall some basic properties of 
Jacobi polynomials $P^{(\ga,\gb)}_m(z)$
for $\ga, \gb \in \C$ and $m \in \Z_{\geq 0}$,
which may be given as follows (cf.\ \cite[p.\ 68]{Szego75}).
\begin{equation*}
P^{(\ga,\gb)}_m(z) = \sum_{j=0}^m\binom{m+\ga}{m-j}\binom{m+\beta}{j}
\left(\frac{z-1}{2}\right)^j\left(\frac{z+1}{2}\right)^{m-j}.
\end{equation*}
The Jacobi polynomials $P_m^{(\ga,\beta)}(z)$ have the following 
generating function: 
\begin{equation}\label{eqn:Jacobi1}
(1+\frac{z+1}{2}w)^\ga(1+\frac{z-1}{2}w)^\beta
=\sum_{m=0}^\infty P_m^{(\ga-m,\beta-m)}(z) w^m.
\end{equation}
In particular, if $z=0$ and $w=2t$, then
\begin{equation}\label{eqn:Jacobi2}
(1+t)^\ga(1-t)^\beta = \sum_{m=0}^\infty 2^m P_m^{(\ga-m,\beta-m)}(0) t^m.
\end{equation}

It follows from the generating function \eqref{eqn:Jacobi1} that
\begin{equation}\label{eqn:Jacobi3}
P^{(\ga,\beta)}_m(-z) = (-1)^mP_m^{(\beta,\ga)}(z).
\end{equation}
In particular,
\begin{equation}\label{eqn:Jacobi4}
P^{(\ga,\beta)}_m(0) = (-1)^mP_m^{(\beta,\ga)}(0).
\end{equation}

\subsection{Uniform expressions by $P_m^{(\ga,\beta)}(z)$} 

Now we show that the polynomials 
$p_\pm^{(k,\ell)}(t)$ and $p_{\Cay}^{(s;k)}(t)$ are given in terms of 
Jacobi polynomials $P_m^{(\ga,\beta)}(z)$.

For $k, \ell \in \Z_{\geq 0}$ and $s \in \C$, we define
\begin{equation*}
q^{(s;k,\ell)}(t) := 
\sum_{m=0}^{m(k,\ell)}M(k,\ell)^{\underline{m}}
P^{(\frac{m(k,\ell)+s-2m}{2},\frac{m(k,\ell)-s-2m}{2})}(0)t^m,
\end{equation*}
where $m(k,\ell):=\min(k,\ell)$, $M(k,\ell):=\max(k,\ell)$, and 
$r^{\underline{m}}:=r(r-1)(r-2)\cdots (r-m+1) = \frac{\Gamma(r+1)}{\Gamma(r-m+1)}$. 

\begin{thm}\label{thm:unif}
Let $k, \ell \in \Z_{\geq 0}$ and $s\in \C$.
Then $p_{\pm}^{(k,\ell)}(t)$ and $p_{\Cay}^{(s;k)}(t)$ can be expressed as follows:
\begin{alignat}{2}
p_+^{(k,\ell)}(t)
&=q^{(m(k,\ell);k,\ell)}(t)
&&=
\begin{cases}
\sum_{m=0}^\ell k^{\underline{m}} \, 
P_m^{(\ell-m,\, -m)}(0)t^m & \text{if $k \geq \ell$},\\[5pt]
\sum_{m=0}^k \ell^{\underline{m}} \, 
P_m^{(k-m,\, -m)}(0)t^m & \text{if $k < \ell$},
\end{cases}\label{eqn:p+J}\\[5pt]
p_-^{(k,\ell)}(t)
&=q^{(-m(k,\ell);k,\ell)}(t)
&&=
\begin{cases}
\sum_{m=0}^\ell k^{\underline{m}}\, P_m^{(-m,\, \ell-m)}(0)t^m 
& \text{if $k \geq \ell$},\\[5pt]
\sum_{m=0}^k \ell^{\underline{m}} \, P_m^{(-m,\, k-m)}(0)t^m 
& \text{if $k < \ell$},
\end{cases}\label{eqn:p-J}\\[3pt]
p_{\Cay}^{(s;k)}(t)
&=q^{(s;k,k)}(t)
&&=\sum_{m=0}^k k^{\underline{m}}\, 
P_m^{(\frac{k+s-2m}{2}, \, \frac{k-s-2m}{2})}(0)t^m. \label{eqn:pcJ}
\end{alignat}
In particular, we have 
\begin{equation}\label{eqn:pcJpm}
p_{\Cay}^{(\pm k;k)}(t) = p_{\pm}^{(k,k)}(t).
\end{equation}
\end{thm}

\begin{proof}
The identity \eqref{eqn:pcJpm} readily follows from
\eqref{eqn:p+J}, \eqref{eqn:p-J}, and \eqref{eqn:pcJ}.
Thus, we only focus on showing the first three identities.

We start with the proof for \eqref{eqn:p+J}. 
Without loss of generality, we assume $k \geq \ell$. Thus,
\begin{equation}\label{eqn:p+}
p_+^{(k,\ell)}(t) 
= \sum_{m=0}^{\ell} \frac{1}{2^m}m!\binom{k}{m}\binom{\ell}{m}t^m
= \sum_{m=0}^{\ell} \frac{k^{\underline{m}}}{2^m}\binom{\ell}{m}t^m.
\end{equation}
Observe that the polynomial $(1+t)^\ell$ is expanded as
\begin{equation}\label{eqn:BinomGen}
(1+t)^\ell = \sum_{m=0}^\ell \binom{\ell}{m}t^m,
\end{equation}
which is the case for $(\ga, \beta) = (\ell, 0)$ in \eqref{eqn:Jacobi2}. Therefore,
\begin{equation}\label{eqn:binom1}
\binom{\ell}{m} = 2^m \cdot P_m^{(\ell-m, -m)}(0).
\end{equation}
Now \eqref{eqn:binom1} and \eqref{eqn:p+} conclude the proposed identity \eqref{eqn:p+J}.
The identity \eqref{eqn:p-J} for $p_-^{(k,\ell)}(t)$ is obtained by applying 
\eqref{eqn:Jacobi4} to \eqref{eqn:p+J}.

To show \eqref{eqn:pcJ}, recall from \eqref{eqn:pc0} that
\begin{equation}\label{eqn:pc}
p_{\Cay}^{(s;k)}(t)
=\sum_{m=0}^k \frac{1}{2^m}\binom{k}{m}\Cay_m(s;k) t^m
=\sum_{m=0}^k \frac{1}{2^m} \frac{k^{\underline{m}} }{m!}\Cay_m(s;k) t^m.
\end{equation}
It is known that 
the generating function of Cayley continuants $\{\Cay_m(x;y)\}$
is given as
\begin{equation*}
(1+t)^{\frac{y+x}{2}}(1-t)^{\frac{y-x}{2}} = 
\sum_{m=0}^\infty\Cay_m(x;y) \frac{t^m}{m!},
\end{equation*}
which is the case for $(\ga, \beta) = (\frac{y+x}{2}, \frac{y-x}{2})$ in \eqref{eqn:Jacobi2}
(cf.\ \cite[(4)]{MT05}). Thus,
\begin{align}\label{eqn:CayP}
\Cay_m(s;k) 
&= m! \cdot 2^m \cdot P_m^{(\frac{k+s-2m}{2}, \frac{k-s-2m}{2})}(0).
\end{align}
The formula \eqref{eqn:pcJ} is then given by substituting \eqref{eqn:CayP} into \eqref{eqn:pc}. 
\end{proof}

\begin{rem}\label{rem:pJ}
The identity \eqref{eqn:pcJpm} 
plays a role in the degenerate cases of
the factorization identities of DIO $\D$ 
and $(\fg, B)$-homomorphisms $\varphi$.
(See Cases (5) and (6) of 
Theorems \ref{thm:factor1}  and \ref{thm:factor2}.)
\end{rem}

\begin{rem}
As $P_m^{(0,0)}(t) = 1$, 
it is clear that $q^{(0;k,0)}(t)=q^{(0;0,\ell)}(t) = q^{(0;0,0)}(t)=1$.
It then follows from Propositions \ref{prop:1} and \ref{prop:2}
and Theorem \ref{thm:unif} that the solution space
$\wSol_{(k,\ell)}(-\lambda_1,-\lambda_2)$ in \eqref{eqn:Solsec5} 
is of the form 
\begin{equation*}
\wSol_{(k,\ell)}(-\lambda_1,-\lambda_2) = \C q^{(s;k,\ell)}(t)
\end{equation*}
for some appropriate parameters $\lambda_1, \lambda_2, s, k, \ell$.
\end{rem}

For $x, y \in \C$ and $m\in \Z_{\geq 0}$, we write
\begin{equation}\label{eqn:BKraw}
K_m(x;y):=\sum_{j=0}^m(-1)^j\binom{x}{j}\binom{y-x}{m-j},
\end{equation}
where the binomial coefficient $\binom{a}{m}$ is defined as follows.
\begin{equation*}
\binom{a}{m}=
\begin{cases}
\frac{a(a-1) \cdots (a-m+1)}{m!} & \text{if $m \in 1+\Z_{\geq 0}$},\\
1 & \text{if $m=0$},\\
0 & \text{otherwise}.
\end{cases}
\end{equation*}
The polynomials $K_m(x;y)$ are called (binary) Krawtchouk polynomials 
(cf.\ \cite[p.\ 151]{MS77} and \cite{Podesta16+}).

\begin{cor}\label{cor:pcK}
For $k, \ell \in \Z_{\geq 0}$, the following hold.
\begin{align}
p_{\Cay}^{(k-2\ell;k)}(t) &= \sum_{m=0}^k \frac{1}{2^m} 
 k^{\underline{m}} K_m(\ell;k) t^m, \label{eqn:Kraw1}\\[3pt]
p_{\Cay}^{(2k-\ell;\ell)}(t) &= \sum_{m=0}^\ell \frac{(-1)^m}{2^m}  \ell^{\underline{m}} 
K_m(k;\ell) t^m. \label{eqn:Kraw2}
\end{align}
\end{cor}

\begin{proof}
To show \eqref{eqn:Kraw1},
observe that if $s=k-2\ell$, then
\begin{equation*}
\tfrac{1}{2}(k+s-2m)=k-\ell-m
\quad \text{and} \quad
\tfrac{1}{2}(k-s-2m) = \ell-m.
\end{equation*}
Then, by Theorem \ref{thm:unif}, we have 
\begin{equation}\label{eqn:pcP}
p_{\Cay}^{(k-2\ell;k)}(t) = \sum_{m=0}^k k^{\underline{m}}
P_m^{(k-\ell-m, \ell-m)}(0)t^m.
\end{equation}
On the other hand, it follows from \eqref{eqn:BKraw} that
$K_m(x;y)$ satisfies the following generating function:
\begin{equation}\label{eqn:KroGen}
(1+t)^{y-x}(1-t)^x = \sum_{m=0}^\infty K_m(x;y) t^m,
\end{equation}
which is the case for $(\ga, \beta)= (y-x, x)$ in \eqref{eqn:Jacobi2}. Thus,
\begin{equation}\label{eqn:KJ}
K_m(\ell;k) = 
2^m \cdot P_m^{(k-\ell-m, \ell-m)}(0).
\end{equation}
Now \eqref{eqn:pcP} and \eqref{eqn:KJ} yield the desired formula.

Likewise, if $s=2k-\ell$, then the same argument shows that
\begin{align*}
p_{\Cay}^{(2k-\ell;\ell)}(t) 
&= \sum_{m=0}^\ell \ell^{\underline{m}}
P_m^{(k-m, \ell-k-m)}(0)t^m \\[3pt]
&= \sum_{m=0}^\ell \ell^{\underline{m}}(-1)^m
P_m^{(\ell-k-m,k-m)}(0)t^m,
\end{align*}
where \eqref{eqn:Jacobi4} is applied from line one to line two.
The identity $K_m(k;\ell)=2^m \cdot P_m^{(\ell-k-m,k-m)}(0)$ 
then concludes the other formula \eqref{eqn:Kraw2}.
\end{proof}

We end this section by showing an identity involving
binary Krawtchouk polynomials $K_m(x;y)$, which will 
be used in the proof of  the factorization identities in the next section.

\begin{lem}
For $x,y \in \C$ and $m \in \Z_{\geq 0}$, we have 
\begin{equation}\label{eqn:BinKro}
\binom{x}{m} = (-1)^m\sum^m_{r=0}\binom{x-y}{m-r}K_r(x;y).
\end{equation}

\end{lem}

\begin{proof}
Observe that 
\begin{equation*}
(1-t)^{x} = (1+t)^{x-y}(1+t)^{y-x}(1-t)^x,
\end{equation*}
which implies that 
\begin{equation*}
\sum_{m=0}^\infty(-1)^m\binom{x}{m}t^m = 
\big(\sum_{u=0}^\infty \binom{x-y}{u}t^u \big) \big(\sum_{r=0}^\infty K_r(x;y)t^r\big)
\end{equation*}
by \eqref{eqn:BinomGen} and \eqref{eqn:KroGen}.
Now take the coefficients of $t^m$ on the both sides to conclude \eqref{eqn:BinKro}.
\end{proof}

\section{Factorization identities}
\label{sec:factorization}


The aim of this section is to show the factorization identities 
of the differential intertwining operators $\D$ and 
$(\fg, B)$-homomorphisms $\varphi$.
Such identities are accomplished  
in Theorems \ref{thm:factor1} and \ref{thm:factor2}
for $\varphi$ and $\D$, respectively. 

\subsection{General theory of homomorphisms between Verma modules}
\label{subsec:Verma1}
We begin with a quick overview of a general theory of homomorphisms 
between Verma modules.
Let $\fg$ be a complex simple Lie algebra and fix a Cartan subalgebra $\fh$ of $\fg$.
We also fix an inner product $\IP{\cdot}{\cdot}$ on $\fh^*$.
Let $\Delta$ denote the set of roots of $\fg$ with respect to $\fh$.
Choose a positive system $\gD^+$ for $\gD$ and write $\Pi$ for the 
set of simple roots for $\gD^+$. We write $\fb = \fh \oplus \fn_+$ for
the Borel subalgebra corresponding to $\gD^+$.
For $\ga \in \gD$ and $\mu \in \fh^*$, 
we write $s_\ga\mu = \mu - \IP{\mu}{\ga^\ssv}\ga$
with $\ga^\ssv =\frac{2}{\IP{\ga}{\ga}}\ga$.

As in Section \ref{sec:InfChar}, for $\mu \in \fh^*$, we write
\begin{equation*}
N(\mu):=\Cal{U}(\fg)\otimes_{\Cal{U}(\fb)} \C_{\mu-\rho},
\end{equation*}
where $\rho$ is half the sum of the positive roots.

It is well known that we have
$\dim_\C\Hom_{\fg}(N(\nu),N(\lambda)) \leq 1$
for any $\nu, \lambda \in \fh^*$.
To describe the condition when $\dim_\C\Hom_{\fg}(N(\nu),N(\lambda)) = 1$,
we next recall the notion of a \emph{link} between weights.

\begin{defn}
[Bernstein--Gelfand--Gelfand]
\label{def:Link}
Let $\gl, \lambda \in \fh^*$ and $\gb_1, \ldots, \gb_t \in \gD^+$. Set $\lambda_0 = \lambda$
and $\lambda_i = s_{\gb_i} \cdots s_{\gb_1}(\lambda)$ for $1 \leq i \leq t$.
We say that the sequence $(\gb_1, \ldots, \gb_t)$ \emph{links} $\lambda$ to $\gl$ if 
the following two conditions are satisfied:
\begin{enumerate}
\item[(1)] $\lambda_t = \gl$;
\item[(2)] $\IP{\lambda_{i-1}}{\gb_i^{\ssv}} \in \Z_{\geq 0}$ for $1\leq i \leq t$.
\end{enumerate}
\end{defn}

Let $L(\gl)$ denote the unique irreducible quotient of 
the Verma module $N(\gl)$.
The following celebrated result of BGG--Verma shows when 
$\dim_\C\Hom_{\fg}(N(\nu),N(\lambda)) = 1$.

\begin{thm}
[BGG--Verma]
\label{thm:BGGV}
The following three conditions on $\gl, \lambda \in \fh^*$ are equivalent.
\begin{enumerate}
\item[\emph{(i)}] $\dim_\C\Hom_{\fg}(N(\nu),N(\lambda)) =1$.
\item[\emph{(ii)}] $L(\gl)$ is a composition factor of $N(\lambda)$.
\item[\emph{(iii)}] There exists a sequence $(\gb_1, \ldots, \gb_t)$ with $\gb_i \in \gD^+$
that links $\lambda$ to $\gl$.
\end{enumerate}
\end{thm}

To show factorization identities,
singular vectors of the composition of 
two homomorphisms between Verma modules have to be computed.
The next two lemmas will be useful for this matter.

\begin{lem}\label{lem:hom-comp}
For $\lambda, \nu, \mu \in \fh^*$,
let 
$\varphi_1\colon N(\nu)\to N(\mu)$ and $\varphi_2\colon N(\mu)\to N(\lambda)$
be
homomorphisms 
between Verma modules such that 
$\varphi_1(1\otimes \mathbb{1}_{\nu-\rho})=u_1\otimes \mathbb{1}_{\mu-\rho}$
and $\varphi_2(1\otimes \mathbb{1}_{\mu-\rho})=u_2\otimes \mathbb{1}_{\lambda-\rho}$.
Then $(\varphi_2 \circ \varphi_1)(1\otimes \mathbb{1}_{\nu-\rho})$ 
is given as
\begin{equation*}
(\varphi_2 \circ \varphi_1)(1\otimes \mathbb{1}_{\nu-\rho}) 
= u_1u_2\otimes \mathbb{1}_{\lambda-\rho}.
\end{equation*}
\end{lem}

One ensures the order of $u_1$ and $u_2$ in Lemma \ref{lem:hom-comp},
namely, $(\varphi_2 \circ \varphi_1)(1\otimes \mathbb{1}_{\nu-\rho}) 
\neq u_2u_1\otimes \mathbb{1}_{\lambda-\rho}$
(cf.\  \cite[Lem.\ 3.7]{KuOr19}).

In the following lemma, we denote by $\acts$ the action of 
the elements of the Weyl algebra of $\fn$ (see Section \ref{sec:Weyl2}).

\begin{lem}
Retain the hypothesis of Lemma \ref{lem:hom-comp}. 
Let $\widehat{\dpi}\equiv\widehat{d\pi_{(\xi,\lambda)^*}}$ denote the 
Fourier transformed representation of $\fg$ defined as in \eqref{eqn:hdpi}.
Then the algebraic Fourier transform
$F_c(\varphi_2\circ \varphi_1)$ of $\varphi_2\circ \varphi_1  \in \Hom_{\fg}(N(\nu),N(\lambda))$ 
is given as
\begin{equation}\label{eqn:Fc-comp0}
F_c(\varphi_2\circ \varphi_1) 
= (\widehat{\dpi}(u_1)\widehat{\dpi}(u_2))\acts 1
= \widehat{\dpi}(u_1) \acts (\widehat{\dpi}(u_2)\acts 1).
\end{equation}
In particular, if $u_j \in \Cal{U}(\fn_-)$ for $j=1,2$, then
\begin{equation}\label{eqn:Fc-comp}
F_c(\varphi_2\circ \varphi_1) 
= (\widehat{dL}(u_1)\widehat{dL}(u_2))\acts 1
= \widehat{dL}(u_1) \acts (\widehat{dL}(u_2)\acts 1).
\end{equation} 
\end{lem}

\begin{proof}
These
are immediate consequences of Lemma \ref{lem:hom-comp} and \eqref{eqn:dpiL}.
\end{proof}

\subsection{Factorization identities}
\label{subsec:Verma2}
Now we consider the factorization identities of 
$(\fg, B)$-homomorphisms $\varphi$ and DIOs $\D$ in order.
Recall from Section \ref{sec:DIO} that
\begin{align*}
\Lambda_1&=(-\Z_{\geq 0})\times \C, \\[3pt]
 \Lambda_2&=\C \times (-\Z_{\geq 0}), \\[3pt]
 \Lambda_+&=\{(1-k, 1-\ell+k) : k, \ell \in 1+\Z_{\geq 0}\}, \\[3pt]
 \Lambda_-&=\{(1-k+\ell, 1-\ell) : k, \ell \in 1+\Z_{\geq 0}\}, \\[3pt]
\Lambda_{\Cay}&=\{(\tfrac{1}{2}(2-k-s), \tfrac{1}{2}(2-k+s)) : k \in 1+\Z_{\geq 0} \; 
\text{and} \;  s \in \C\}.
\end{align*}
Moreover, for $\square \in \{> , <, =\}$,
we put
\begin{align*}
 \Lambda_+^{(k\,\square\, \ell)}
&:=\{(1-k, 1-\ell+k) \in  \Lambda_+:  k\, \square\, \ell \}, \\[3pt]
\Lambda_-^{(k\, \square\, \ell)}
&:=\{(1-k+\ell, 1-\ell)  \in  \Lambda_-:
k\, \square\, \ell\}.
\end{align*}
One can easily check that
\begin{equation*}
 \Lambda_+ \subset \Lambda_1 \cap \Lambda_{\Cay},
\quad
 \Lambda_- \subset  \Lambda_2 \cap \Lambda_{\Cay},
\quad
\text{and}
\quad
 \Lambda_+^{(k < \ell)} = \Lambda_-^{(k>\ell)}.
\end{equation*}

Also, under the identification \eqref{eqn:astar} of $\fa^*$,
we have 
\begin{align*}
\mu_\lambda
&:=-(\lambda_1\varpi_1+\lambda_2\varpi_2)+\rho\\[3pt]
&=\tfrac{1}{3}(-(2\lambda_1+\lambda_2-3), \lambda_1-\lambda_2, \lambda_1+2\lambda_2-3).
\end{align*}
We put 
\begin{equation*}
\theta^+_{(k,\ell)}:=\frac{1}{3}(k+\ell, \ell-2k, k-2\ell)
\quad
\text{and}
\quad
\theta^-_{(k,\ell)}:=\frac{1}{3}(2k-\ell, 2\ell-k, -(k+\ell)).
\end{equation*}
Then,
\begin{equation*}
\mu_\lambda 
=
\begin{cases}
\theta^+_{(k,\ell)} &
\text{if $(\lambda_1,\lambda_2)=(1-k,1-\ell+k) \in  \Lambda_+$},\\[3pt]
\theta^-_{(k,\ell)} &
\text{if $(\lambda_1,\lambda_2)=(1-k+\ell,1-\ell) \in  \Lambda_-$}
\end{cases}
\end{equation*}
and
\begin{equation*}
\theta^+_{(k,\ell)} = \theta^-_{(\ell, \ell-k)}
\quad
\text{for $k \leq \ell$}
\quad
\text{and}
\quad
\theta^-_{(k, \ell)}= \theta^+_{(k-\ell,k)}
\quad
\text{for $k\geq \ell$}.
\end{equation*}
Moreover, we have
\begin{equation*}
s_\beta s_\ga\theta^+_{(k,k)} = s_\gamma \theta^+_{(k,k)}
\quad
\text{and}
\quad
s_\ga s_\beta\theta^-_{(k,k)} = s_\gamma \theta^-_{(k,k)}.
\end{equation*}

\medskip

First we show that, to study factorization identities, 
it suffices to consider $(\lambda_1,\lambda_2) \in \Lambda_+\cup \Lambda_-$.
As in Theorem \ref{thm:Hom3},
we write
\begin{equation*}
\ga=(1,-1,0),\quad
\beta=(0,1,-1), \quad
\gamma=(1,0,-1).
\end{equation*}

\begin{prop}
Suppose that 
$\varphi\colon N(\mu'_\lambda)^{(\delta_1,\delta_2)} \to N(\mu_\lambda)^{(+,+)}$
is a non-zero $(\fg, B)$-homomorphism in Theorem \ref{thm:Hom3}.
If $\varphi$ admits a factorization identity, then 
$(\lambda_1,\lambda_2)\in \Lambda_+\cup \Lambda_-$.
\end{prop}

\begin{proof}
Since $\varphi$ admits a factorization identity, 
it follows from Theorem \ref{thm:BGGV} that
we have 
$\mu'_\lambda \in 
\{s_\ga s_\beta \mu_\lambda, \, 
s_\beta s_\ga \mu_\lambda,\, s_\gamma \mu_\lambda\}$.
By Theorem \ref{thm:Hom3},
if $\mu_\lambda' \in \{ s_\ga s_\beta \mu_\lambda,\, s_\beta s_\ga \mu_\lambda\}$, then
$\varphi = \varphi^{(k,\ell)}_{\pm}$ for some $k,\ell \in 1+\Z_{\geq 0}$.
Thus, in this case, $(\lambda_1,\lambda_2) \in \Lambda_+\cup \Lambda_-$.

We now show that even if $\mu_\lambda' = s_\gamma \mu_\lambda$,
the parameters $(\lambda_1,\lambda_2)$ satisfy the proposed condition.
Indeed, in this case, we have $\varphi = \varphi_{\Cay}^{(s; k)}$ and 
$(\lambda_1,\lambda_2) = (\tfrac{1}{2}(2-k-s), \tfrac{1}{2}(2-k+s))$
for some $k \in 1+\Z_{\geq 0}$ and $s \in \C$. 
Let $\varphi_{\Cay}^{(s; k)} = \varphi'' \circ \varphi'$ denote a factorization
of $\varphi_{\Cay}^{(s; k)}$. Then 
$\varphi'' \in \{\varphi^{(\ell)}_1,  \varphi^{(\ell)}_2, 
\varphi^{(\ell_1, \ell_2)}_{\pm}\}$ for some $\ell, \ell_1, \ell_2 \in 1+\Z_{\geq 0}$.
If $\varphi'' = \varphi^{(\ell_1, \ell_2)}_{\pm}$, then 
$(\lambda_1,\lambda_2) \in \Lambda_+\cup \Lambda_-$.
Next, suppose that $\varphi'' \in \{\varphi^{(\ell)}_1,  \varphi^{(\ell)}_2\}$.
If $\varphi'' = \varphi^{(\ell)}_1$, then
Theorem \ref{thm:Hom3} shows that
$\varphi''$ is a map
$\varphi'' \colon N(s_\ga\mu_\lambda)^{(+,(-)^\ell)} \to N(\mu_\lambda)^{(+,+)}$
with
$\IP{\mu_\lambda}{\ga} =\ell$. 
On the other hand,
a direct computation shows that 
$\IP{\mu_\lambda}{\ga}=\tfrac{1}{2}(k+s)$.
The equality $\tfrac{1}{2}(k+s)=\ell$
leads to the condition $(\lambda_1,\lambda_2) \in \Lambda_+$.
Likewise, if $\varphi'' = \varphi^{(\ell)}_2$, then
$(\lambda_1,\lambda_2) \in \Lambda_-$.
Now the proposition follows.
\end{proof}

We next exhibit the factorization identities for 
$(\lambda_1,\lambda_2) \in \Lambda_+\cup \Lambda_-$.

\begin{thm}\label{thm:factor1}
Let $(\lambda_1,\lambda_2) \in  \Lambda_+ \cup  \Lambda_-$.
Then the following factorization identities hold.

\begin{enumerate}

\item[\emph{(1)}] 
For $(\lambda_1,\lambda_2) = (1-k, 1-\ell+k) 
\in  \Lambda_+^{(k < \ell)}(=\Lambda_-^{(k>\ell)})$,
the homomorphisms $\varphi_+^{(k,\ell)}$, $\varphi_-^{(\ell-k,\ell)}$, 
and $\varphi_{\Cay}^{(2k-\ell;\ell)}$
can be factored as follows.
\begin{equation*}
\varphi_+^{(k,\ell)}=\varphi_1^{(k)}\circ \varphi_2^{(\ell)},
\quad
\varphi_-^{(\ell-k,\ell)}=\varphi_2^{(\ell)}\circ \varphi_1^{(\ell-k)},
\end{equation*}
and
\begin{equation*}
\varphi_{\Cay}^{(2k-\ell;\ell)}
=\varphi_+^{(k,\ell)}\circ \varphi_1^{(\ell-k)}
=\varphi_1^{(k)}\circ \varphi_-^{(\ell-k,\ell)}
=\varphi_1^{(k)}\circ \varphi_2^{(\ell)} \circ \varphi_1^{(\ell-k)}.
\end{equation*}
Equivalently, the following diagram commutes.
\begin{equation*}
\begin{tikzcd}[row sep=1cm, column sep=1cm]
&
N(s_\beta s_\alpha\theta^+_{(k,\ell)})^{((-)^\ell,(-)^k)} 
\arrow[rrd, "\varphi_+^{(k,\ell)}"']
\arrow[r,"\varphi_2^{(\ell)}"]  
& N(s_\alpha\theta^+_{(k,\ell)})^{(+,(-)^k)}
 \arrow[dr, "\varphi_1^{(k)}"]&\\
 N(s_\gamma\theta^+_{(k,\ell)})^{((-)^\ell,(-)^\ell)}
\arrow[urr, "\varphi_-^{(\ell-k,\ell)}"']
\arrow[ur, "\varphi_1^{(\ell-k)}"]
\arrow[rrr, "\varphi_{\Cay}^{(2k-\ell;\ell)}"']
& 
\arrow[u,  pos=0.89, phantom, "\circlearrowleft"]
\arrow[ru,  pos=0.55, phantom, "\circlearrowleft"]
& \arrow[u,  pos=0.89, phantom, "\circlearrowleft"]
& N(\theta^+_{(k,\ell)})^{(+,+)}\\
\end{tikzcd}
\end{equation*}

\item[\emph{(2)}]
For $(\lambda_1,\lambda_2) =(1-k, 1-\ell+k) \in  \Lambda_+^{(k>\ell)}$,
the homomorphism $\varphi_+^{(k,\ell)}$ can be factored as 
\begin{equation*}
\varphi_{\Cay}^{(2k-\ell;\ell)} \circ \varphi^{(k-\ell)}_1
= \varphi_+^{(k,\ell)}
=\varphi^{(k)}_1 \circ \varphi^{(\ell)}_2.
\end{equation*}
Equivalently, 
the following diagram commutes.
\begin{equation*}
\begin{tikzcd}[row sep=1cm, column sep=1cm]
& N(s_\gamma \theta^+_{(k,\ell)})^{((-)^\ell,(-)^\ell)} 
\arrow[ddr, "\varphi_{\Cay}^{(2k-\ell;\ell)}"] 
&\\
N(s_\beta s_\ga\theta^+_{(k,\ell)})^{((-)^\ell,(-)^k)} 
\arrow[drr,"\varphi_+^{(k,\ell)}"]
\arrow[ur, "\varphi_1^{(k-\ell)}"]
\arrow[dr, "\varphi_2^{(\ell)}"']
 & & \\
&N(s_\ga\theta^+_{(k,\ell)})^{(+,(-)^k)}
\arrow[uu,  pos=0.5, phantom, "\circlearrowleft"] 
\arrow[uu,  pos=0.03, phantom, "\circlearrowleft"] 
\arrow[r, "\varphi_1^{(k)}"']  
&N(\theta^+_{(k,\ell)})^{(+,+)}  
\\
\end{tikzcd}
\end{equation*}

\item[\emph{(3)}] 
For $(\lambda_1,\lambda_2) =(1-k+\ell, 1-\ell) \in
\Lambda_-^{(k>\ell)}(= \Lambda_+^{(k < \ell)})$,
the homomorphisms 
$\varphi_+^{(k,k-\ell)}$,
$\varphi_-^{(k,\ell)}$, and $\varphi_{\Cay}^{(k-2\ell;k)}$ can be 
factored as follows
\begin{equation*}
\varphi_+^{(k,k-\ell)} = \varphi^{(k)}_1 \circ \varphi^{(k-\ell)}_2,\quad
\varphi_-^{(k,\ell)} = \varphi^{(\ell)}_2 \circ \varphi^{(k)}_1,
\end{equation*}
and
\begin{equation*}
\varphi_{\Cay}^{(k-2\ell;k)} 
=\varphi_2^{(\ell)}\circ \varphi_+^{(k,k-\ell)}
=\varphi_-^{(k,\ell)} \circ \varphi^{(k-\ell)}_2
=\varphi_2^{(\ell)}\circ \varphi_1^{(k)}\circ \varphi_2^{(k-\ell)}.
\end{equation*}
Equivalently, 
the following diagram commutes.
\begin{equation*}
\begin{tikzcd}[row sep=1cm, column sep=1cm]
N(s_\gamma \theta^-_{(k,\ell)})^{((-)^k,(-)^k)}
\arrow[dr, "\varphi_2^{(k-\ell)}"']
\arrow[rrr, "\varphi_{\Cay}^{(k-2\ell;k)}"]
\arrow[drr, "\varphi_+^{(k,k-\ell)}"]
& 
\arrow[d,  pos=0.89, phantom, "\circlearrowleft"]
\arrow[rd,  pos=0.55, phantom, "\circlearrowleft"]
& 
\arrow[d,  pos=0.89, phantom, "\circlearrowleft"]
&  N(\theta^-_{(k,\ell)})^{(+,+)}\\
&N(s_\ga s_\beta\theta^-_{(k,\ell)})^{((-)^\ell,(-)^k)}   
\arrow[r, "\varphi_1^{(k)}"']  
\arrow[urr, "\varphi_-^{(k,\ell)}"]
&N(s_\beta\theta^-_{(k,\ell)})^{((-)^\ell,+)}  
\arrow[ur, "\varphi_2^{(\ell)}"']&\\
\end{tikzcd}
\end{equation*}

\item[\emph{(4)}]
For $(\lambda_1,\lambda_2)=(1-k+\ell, 1-\ell)\in \Lambda_-^{(k<\ell)}$, 
the homomorphism $\varphi_-^{(k,\ell)}$ 
can be factored as 
\begin{equation*}
\varphi_{\Cay}^{(k-2\ell;k)} \circ \varphi_2^{(\ell-k)}
= \varphi_-^{(k,\ell)}=\varphi_2^{(\ell)}\circ \varphi_1^{(k)}.
\end{equation*}
Equivalently, the following diagram commutes.
\begin{equation*}
\begin{tikzcd}[row sep=1cm, column sep=1cm]
&
N(s_\beta\theta^-_{(k,\ell)})^{((-)^\ell,+)}
\arrow[r, "\varphi_2^{(\ell)}"]  
\arrow[dd,  pos=0.05, phantom, "\circlearrowleft"] 
\arrow[dd,  pos=0.5, phantom, "\circlearrowleft"] 
&N(\theta^-_{(k,\ell)})^{(+,+)} &\\
N(s_\alpha  s_\beta\theta^-_{(k,\ell)})^{((-)^\ell,(-)^k)} 
\arrow[rru, "\varphi_-^{(k,\ell)}"']
\arrow[ur, "\varphi_1^{(k)}"]
\arrow[dr, "\varphi_2^{(\ell-k)}"']
& & & \\
& N(s_\gamma\theta^-_{(k,\ell)})^{((-)^k,(-)^k)} 
\arrow[uur, "\varphi_{\Cay}^{(k-2\ell;k)}"']
& &\\
\end{tikzcd}
\end{equation*}

\item[\emph{(5)}]  
For $(\lambda_1,\lambda_2) =(1-k, 1) \in  \Lambda_+^{(k=\ell)}$,
we have
\begin{equation*}
\varphi_{\Cay}^{(k;k)} 
= \varphi_+^{(k,k)}
=\varphi^{(k)}_1 \circ \varphi^{(k)}_2.
\end{equation*}
Equivalently, 
the following diagram commutes.
\begin{equation*}
\begin{tikzcd}[row sep=1cm, column sep=1cm]
&
N(s_\ga\theta^+_{(k,k)})^{(+,(-)^k)}
\arrow[rd, "\varphi_1^{(k)}"]
 & \\
N(s_\gamma \theta^+_{(k,k)})^{((-)^k,(-)^k)} 
\arrow[rr, "\varphi_{\Cay}^{(k;k)}=\varphi_+^{(k,k)}"']
\arrow[ur, "\varphi_2^{(k)}"]
&
\arrow[u,  pos=0.5, phantom, "\circlearrowleft"]
 &  
 N(\theta^+_{(k,k)})^{(+,+)}\\
\end{tikzcd}
\end{equation*}

\item[\emph{(6)}]
For $(\lambda_1,\lambda_2)=(1, 1-k)\in \Lambda_-^{(k=\ell)}$, 
we have
\begin{equation*}
\varphi_{\Cay}^{(-k;k)} 
= \varphi_-^{(k,k)}
=\varphi_2^{(k)}\circ \varphi_1^{(k)}.
\end{equation*}
Equivalently, the following diagram commutes.
\begin{equation*}
\begin{tikzcd}[row sep=1cm, column sep=1cm]
N(s_\gamma\theta^-_{(k,k)})^{((-)^k,(-)^k)} 
\arrow[rr, "\varphi_{\Cay}^{(-k;k)}=\varphi_-^{(k,k)}"]
\arrow[dr,"\varphi_1^{(k)}"']
&
\arrow[d,  pos=0.5, phantom, "\circlearrowleft"]
 & N(\theta^-_{(k,k)})^{(+,+)} \\
&
N(s_\beta\theta^-_{(k,k)})^{((-)^k,+)}
\arrow[ur, "\varphi_2^{(k)}"']
&
\end{tikzcd}
\end{equation*}

\end{enumerate}

\end{thm}

\begin{proof}
We demonstrate the factorization identities 
$\varphi_+^{(k,\ell)}=\varphi_1^{(k)}\circ \varphi_2^{(\ell)}$ and
$\varphi_{\Cay}^{(2k-\ell;\ell)}=\varphi_+^{(k,\ell)}\circ \varphi_1^{(\ell-k)}$
in (1), and $\varphi_-^{(k,\ell)}= 
\varphi_{\Cay}^{(k-2\ell;k)} \circ \varphi_2^{(\ell-k)}$ in (4);
the other factorizations can be shown similarly.
As the characters $(\pm,\pm)$ of $M$ 
in the diagrams directly follow from Theorem \ref{thm:Hom3},
we only focus on the factorizations of the homomorphisms $\varphi$.
Remark that the identities $\varphi_{\Cay}^{(\pm k;k)} = \varphi_\pm^{(k,k)}$
in Cases (5) and (6) follow from \eqref{eqn:pcJpm}.

To get ready for the proof, observe that
similar arguments for Lemma \ref{lem:dR}
and the algebraic Fourier transform $\widehat{\;\cdot\;}$ \eqref{eqn:Weyl} 
show that $\widehat{dL}(N_j^-)$ for $j=1,2,3$
are given as
\begin{equation}\label{eqn:dL8}
\widehat{dL}(N_1^-)=
\zeta_1 -\frac{\zeta_3}{2}\frac{\partial}{\partial \zeta_2},
\quad
\widehat{dL}(N_2^-)=\zeta_2 +\frac{\zeta_3}{2}\frac{\partial}{\partial \zeta_1},
\quad
\widehat{dL}(N_3^-)=\zeta_3.
\end{equation}
Moreover, recall from \eqref{eqn:eta3} that 
$\eta_{\Cay}^{(s;k)}(a_1,a_2,a_3)$ is a polynomial on an algebra $\Cal{A}$ defined as
\begin{equation*}
\eta_{\Cay}^{(s;k)}(a_1,a_2,a_3)
= \sum_{m=0}^{k} \frac{1}{2^m}\binom{k}{m}\Cay_m(s;k)
\sym_{2k-m}(a_1^{k-m} a_2^{k-m} a_3^m).
\end{equation*}
By \eqref{eqn:Kraw1} and \eqref{eqn:Kraw2}, we have 
\begin{align}
\eta_{\Cay}^{(k-2\ell;k)}(a_1,a_2,a_3)&=\sum_{m=0}^k \frac{1}{2^m}k^{\underline{m}}K_m(\ell;k)
\sym_{2k-m}(a_1^{k-m} a_2^{k-m} a_3^m), \label{eqn:nKraw1}\\[3pt]
\eta_{\Cay}^{(2k-\ell;\ell)}(a_1,a_2,a_3)&=\sum_{m=0}^\ell \frac{(-1)^m}{2^m}\ell^{\underline{m}}K_m(k;\ell)
\sym_{2\ell-m}(a_1^{\ell-m} a_2^{\ell-m} a_3^m).\label{eqn:nKraw2}
\end{align}

In the following, we omit writing the highest weights of Verma modules.
For instance, we simply write $1\otimes \mathbb{1}$
for a highest weight vector $1\otimes \mathbb{1}_{\theta^+_{(k,\ell)}-\rho}$
of $N(\theta^+_{(k,\ell)})$.

We start with $\varphi_+^{(k,\ell)} = \varphi^{(k)}_1 \circ \varphi^{(\ell)}_2$.
To prove this, it suffices to show 
$F_c(\varphi_+^{(k,\ell)}) = F_c(\varphi^{(k)}_1 \circ \varphi^{(\ell)}_2)$.
Observe that the homomorphisms
$\varphi^{(k)}_1 \circ \varphi^{(\ell)}_2$
and 
$\varphi_+^{(k,\ell)}$
map
\begin{align*}
\varphi^{(k)}_1 \circ \varphi^{(\ell)}_2
&\colon
1\otimes \mathbb{1}
\longmapsto
(N_2^-)^\ell(N_1^-)^k \otimes \mathbb{1},\\[3pt]
\varphi_+^{(k,\ell)}
&\colon
1\otimes \mathbb{1}
\longmapsto
\eta_+^{(k,\ell)}(N_1^-,N_2^-,N_3^-)\otimes \mathbb{1}.
\end{align*}
By \eqref{eqn:dL8}, we have 
$\widehat{dL}(N_1^-)^k\acts 1 = \zeta_1^k$.
Also, by \eqref{eqn:eta1},
the polynomial $F_c(\varphi_+^{(k,\ell)})$ is 
\begin{equation*}
F_c(\varphi_+^{(k,\ell)})
=\eta_+^{(k,\ell)}(\zeta_1,\zeta_2,\zeta_3)
=\sum_{m=0}^k\frac{1}{2^m}m!\binom{k}{m}\binom{\ell}{m} 
\zeta_1^{k-m}\zeta_2^{\ell-m}\zeta_3^m.
\end{equation*}
It then follows from \eqref{eqn:Fc-comp} that the Fourier transform 
$F_c(\varphi^{(k)}_1 \circ \varphi^{(\ell)}_2)$ is evaluated as

\begin{align*}
F_c(\varphi^{(k)}_1 \circ \varphi^{(\ell)}_2)
&=\widehat{dL}(N_2^-)^\ell\acts \left(\widehat{dL}(N_1^-)^k\acts 1 \right) \\[3pt]
&=\widehat{dL}(N_2^-)^\ell\acts \zeta_1^k \\[3pt]
&=\left(\zeta_2+\frac{\zeta_3}{2}\frac{\partial}{\partial \zeta_1} \right)^\ell 
\acts \zeta_1^k\\[3pt]
&=\left(
\sum_{m=0}^\ell\frac{1}{2^m}\binom{\ell}{m} \zeta_2^{\ell-m}\zeta_3^m
\frac{\partial^m}{\partial \zeta_1^m} \right) \acts \zeta_1^k\\[3pt]
&=
\sum_{m=0}^k\frac{1}{2^m}m!\binom{k}{m}\binom{\ell}{m} 
\zeta_1^{k-m}\zeta_2^{\ell-m}\zeta_3^m\\[3pt]
&=F_c(\varphi_+^{(k,\ell)}).
\end{align*}

Next, we show 
$\varphi_{\Cay}^{(2k-\ell;\ell)}
=\varphi_+^{(k,\ell)}\circ \varphi_1^{(\ell-k)}$.
The homomorphisms $\varphi_+^{(k,\ell)}\circ \varphi_1^{(\ell-k)}$ 
and $\varphi_{\Cay}^{(2k-\ell;\ell)}$ map
\begin{align*}
\varphi_+^{(k,\ell)}\circ \varphi_1^{(\ell-k)}
&\colon
1\otimes \mathbb{1} \longmapsto
(N_1^-)^{\ell-k}\eta_+^{(k,\ell)}(N_1^-,N_2^-,N_3^-) \otimes \mathbb{1}, \\[3pt]
\varphi_{\Cay}^{(2k-\ell;\ell)} &\colon 1\otimes \mathbb{1} \longmapsto
\eta_{\Cay}^{(2k-\ell;\ell)}(N_1^-,N_2^-,N_3^-) \otimes \mathbb{1}.
\end{align*}
It follows from \eqref{eqn:eta1} and \eqref{eqn:nKraw2} that
\begin{align*}
\widehat{dL}(\eta_+^{(k,\ell)}(N_1^-,N_2^-,N_3^-)) \acts 1
&=\eta_+^{(k,\ell)}(\zeta_1,\zeta_2,\zeta_3)\\
&=\sum_{j=0}^k\frac{1}{2^j}j!\binom{k}{j}\binom{\ell}{j} \zeta_1^{k-j}\zeta_2^{\ell-j}\zeta_3^j,\\[3pt]
F_c(\varphi_{\Cay}^{(2k-\ell;\ell)})
&=\eta_{\Cay}^{(2k-\ell;\ell)}(\zeta_1,\zeta_2,\zeta_3)\\
&=\sum_{m=0}^\ell \frac{(-1)^m}{2^m}\ell^{\underline{m}}K_m(k;\ell)
\zeta_1^{\ell-m} \zeta_2^{\ell-m} \zeta_3^m.
\end{align*} 
Thus we have 
\begin{align*}
F_c(\varphi_+^{(k,\ell)}\circ \varphi_1^{(\ell-k)})
&=
\widehat{dL}(N_1^-)^{\ell-k}\acts
(\widehat{dL}(\eta_+^{(k,\ell)}(N_1^-,N_2^-,N_3^-) \acts 1)\\[3pt]
&=
\left(\zeta_1-\frac{\zeta_3}{2}\frac{\partial}{\partial \zeta_2} \right)^{\ell-k}\acts
\left(
\sum_{j=0}^k\frac{1}{2^j}j!\binom{k}{j}\binom{\ell}{j} \zeta_1^{k-j}\zeta_2^{\ell-j}\zeta_3^j
\right)\\
&=\sum_{m=0}^\ell\frac{(-1)^m}{2^m}\ell^{\underline{m}}
\sum_{j=0}^m(-1)^j\binom{k}{j}\binom{\ell-k}{m-j}\zeta_1^{\ell-m}
\zeta_2^{\ell-m}\zeta_3^m\\[3pt]
&=\sum_{m=0}^\ell \frac{(-1)^m}{2^m}\ell^{\underline{m}}K_m(k;\ell)
\zeta_1^{\ell-m} \zeta_2^{\ell-m} \zeta_3^m\\[3pt]
&=F_c(\varphi_{\Cay}^{(2k-\ell;\ell)}).
\end{align*}

As the last identity, we prove 
$\varphi_-^{(k,\ell)}= 
\varphi_{\Cay}^{(k-2\ell;k)} \circ \varphi_2^{(\ell-k)}$.
Observe that we have
\begin{align*}
\varphi_{\Cay}^{(k-2\ell;k)} \circ \varphi_2^{(\ell-k)}
&\colon 1\otimes \mathbb{1} \longmapsto
(N_2^-)^{\ell-k}\eta_{\Cay}^{(k-2\ell;k)}(N_1^-,N_2^-,N_3^-)\otimes \mathbb{1},\\
\varphi_-^{(k,\ell)}
&\colon
1\otimes \mathbb{1} \longmapsto \eta_-^{(k,\ell)}(N_1^-,N_2^-,N_3^-)\otimes \mathbb{1}.
\end{align*}
It follows from \eqref{eqn:eta2} and \eqref{eqn:nKraw1} that
\begin{align*}
\widehat{dL}(\eta_{\Cay}^{(k-2\ell;k)}(N_1^-,N_2^-,N_3^-))\acts 1
&=\eta_{\Cay}^{(k-2\ell;k)}(\zeta_1,\zeta_2,\zeta_3)\\[3pt]
&= \sum_{r=0}^k \frac{1}{2^r}k^{\underline{r}}K_r(\ell;k)
\zeta_1^{k-r} \zeta_2^{k-r} \zeta_3^r,\\
F_c(\varphi_-^{(k,\ell)})
&=
\sum_{m=0}^k \frac{(-1)^m}{2^m}m!
\binom{k}{m}\binom{\ell}{m}\zeta_1^{k-m}\zeta_2^{\ell-m}\zeta_3^m.
\end{align*}
Hence,
\begin{align*}
F_c(\varphi_{\Cay}^{(k-2\ell;k)} \circ \varphi_2^{(\ell-k)})
&=
\widehat{dL}(N_2^-)^{\ell-k}\acts
(\widehat{dL}(\eta_{\Cay}^{(k-2\ell;k)}(N_1^-,N_2^-,N_3^-))\acts 1)\\[3pt]
&=\left(\zeta_2+\frac{\zeta_3}{2}\frac{\partial}{\partial \zeta_1} \right)^{\ell-k} \acts
\left(\sum_{r=0}^k \frac{1}{2^r}k^{\underline{r}}K_r(\ell;k)
\zeta_1^{k-r} \zeta_2^{k-r} \zeta_3^r\right)\\[3pt]
&=\sum_{m=0}^k \frac{m!}{2^m}\binom{k}{m}
\sum_{r=0}^m\binom{\ell-k}{m-r}K_r(\ell;k)\zeta_1^{k-m}\zeta_2^{\ell-m}\zeta_3^m\\[3pt]
&=\sum_{m=0}^k \frac{(-1)^m}{2^m}m!
\binom{k}{m}\binom{\ell}{m}\zeta_1^{k-m}\zeta_2^{\ell-m}\zeta_3^m\\[3pt]
&=F_c(\varphi_-^{(k,\ell)}).
\end{align*}
Here, \eqref{eqn:BinKro} is applied from line three to line four.
\end{proof}

\begin{rem}
As $ \Lambda_+^{(k < \ell)}=\Lambda_-^{(k>\ell)}$,
one may merge the diagrams in Cases (1) and (3) in one.
For instance, 
for $(\lambda_1,\lambda_2) = (1-k, 1-\ell+k) \in  \Lambda_+^{(k < \ell)}$,
we have
\begin{equation}\label{eqn:merge}
\begin{tikzcd}[row sep=1cm, column sep=1cm]
&
N(s_\beta s_\alpha\theta^+_{(k,\ell)})^{((-)^\ell,(-)^k)} 
\arrow[rrd, "\varphi_+^{(k,\ell)}"']
\arrow[r, "\varphi_2^{(\ell)}"]  
& N(s_\alpha\theta^+_{(k,\ell)})^{(+,(-)^k)}
\arrow[dr, "\varphi_1^{(k)}"]&\\
N(s_\gamma\theta^+_{(k,\ell)})^{((-)^\ell,(-)^\ell)}
\arrow[ur, "\varphi_1^{(\ell-k)}"]
\arrow[urr, "\varphi_-^{(\ell-k,\ell)}"']
\arrow[rrr, "\varphi_{\Cay}^{(2k-\ell;\ell)}"]
\arrow[dr, "\varphi_2^{(k)}"']
\arrow[drr, "\varphi_+^{(\ell,k)}"]
& & & N(\theta^+_{(k,\ell)})^{(+,+)}\\
&N(s_\ga s_\beta\theta^+_{(k,\ell)})^{((-)^{\ell-k},(-)^\ell)}   
\arrow[r, "\varphi_1^{(\ell)}"']  
\arrow[urr, "\varphi_-^{(\ell,\ell-k)}"]
&N(s_\beta \theta^+_{(k,\ell)})^{((-)^{\ell-k},+)}  
\arrow[ur, "\varphi_2^{(\ell-k)}"']&
\end{tikzcd}
\end{equation}

\noindent
Moreover, since $s_\ga\theta^+_{(k,\ell)} = \theta^-_{(\ell-k,\ell)}$ and 
$s_\beta\theta^+_{(k,\ell)} = \theta^+_{(\ell,k)}$,
the diagrams in Cases (2) and (4) can be combined into a single diagram as follows.
\begin{equation*}
\begin{tikzcd}[row sep=1cm, column sep=1cm]
&
N(s_\beta s_\alpha\theta^+_{(k,\ell)})^{((-)^\ell,(-)^k)} 
\arrow[ddr, pos=0.8, "\varphi_{\Cay}^{(2\ell-k;k)}"] 
\arrow[r, "\varphi_2^{(\ell)}"]  
& N(s_\alpha\theta^+_{(k,\ell)})^{(+,(-)^k)}
 \arrow[dr, "\varphi_1^{(k)}"]
&\\
 N(s_\gamma\theta^+_{(k,\ell)})^{((-)^\ell,(-)^\ell)}
\arrow[ur, "\varphi_1^{(\ell-k)}"]
\arrow[urr, "\varphi_-^{(\ell-k,\ell)}"']
\arrow[dr, "\varphi_2^{(k)}"']
\arrow[drr, "\varphi_+^{(\ell,k)}"]
& & & N(\theta^+_{(k,\ell)})^{(+,+)}
\\
&N(s_\ga s_\beta\theta^+_{(k,\ell)})^{((-)^{\ell-k},(-)^\ell)}  
\arrow[uur, pos=0.8,"\varphi_{\Cay}^{(-(k+\ell);\ell-k)}"'] 
\arrow[r, "\varphi_1^{(\ell)}"']  
&N(s_\beta \theta^+_{(k,\ell)})^{((-)^{\ell-k},+)}  
\arrow[ur, "\varphi_2^{(\ell-k)}"']&
\end{tikzcd}
\end{equation*}

\end{rem}

We end this section by the differential-operator counterpart of 
Theorem \ref{thm:factor1}. 
We write
\begin{equation*}
\w{I}(\lambda_1,\lambda_2)^{(\eps_1,\eps_2)}
:=\Ind_B^G(\C_{(\eps_1,\eps_2)}\boxtimes 
\C_{(\lambda_1+1,\lambda_2+1)})
\end{equation*}
for the normalized parabolically induced representation.
Under the identification \eqref{eqn:astar} of $\fa^*$,
the unnormalized induced representation 
$I(\lambda_1, \lambda_2)^{(\eps_1,\eps_2)}$ can be 
given as
\begin{equation*}
I(\lambda_1, \lambda_2)^{(\eps_1,\eps_2)}
=\w{I}(\lambda_1-1,\lambda_2-1)^{(\eps_1,\eps_2)}
\equiv\w{I}(-\mu_\lambda)^{(\eps_1,\eps_2)},
\end{equation*}
where $\mu_\lambda$ is the weight defined in \eqref{eqn:muL}.

\begin{thm}\label{thm:factor2}
Let $(\lambda_1,\lambda_2) \in  \Lambda_+ \cup  \Lambda_-$.
Then the following factorization identities hold.

\begin{enumerate}

\item[\emph{(1)}] 
For $(\lambda_1,\lambda_2) = (1-k, 1-\ell+k) 
\in  \Lambda_+^{(k < \ell)}(=\Lambda_-^{(k>\ell)})$,
the differential operators $\D_+^{(k,\ell)}$, $\D_-^{(\ell-k,\ell)}$,
and $\D_{\Cay}^{(2k-\ell;\ell)}$
can be factored as follows. 
\begin{equation*}
\D_+^{(k,\ell)}=
\D_2^{\ell}\circ \D_1^{k},
\quad
\D_-^{(\ell-k,\ell)}=\D_1^{\ell-k} \circ \D_2^\ell
\end{equation*}
and
\begin{equation*}
\D_{\Cay}^{(2k-\ell;\ell)}
=\D_1^{\ell-k}\circ \D_+^{(k,\ell)}
=\D_-^{(\ell-k,\ell)}\circ \D_1^k
=\D_1^{\ell-k}\circ \D_2^\ell \circ \D_1^k.
\end{equation*}
Equivalently, the following diagram commutes.
\begin{equation*}
\begin{tikzcd}[row sep=1cm, column sep=1cm]
&
\w{I}(-s_\alpha\theta^+_{(k,\ell)})^{(+,(-)^k)}
\arrow[r, "\D_2^{\ell}"]  
\arrow[rrd, "\D_-^{(\ell-k,\ell)}"']  
& 
\w{I}(-s_\beta s_\alpha\theta^+_{(k,\ell)})^{((-)^\ell,(-)^k)} 
\arrow[dr, "\D_1^{\ell-k}"]&\\
\w{I}(-\theta^+_{(k,\ell)})^{(+,+)} 
\arrow[rru, "\D_+^{(k,\ell)}"']
\arrow[ur, "\D_1^{k}"]
\arrow[rrr, "\D_{\Cay}^{(2k-\ell;\ell)}"']
& 
\arrow[u,  pos=0.89, phantom, "\circlearrowleft"]
\arrow[ru,  pos=0.65, phantom, "\circlearrowleft"]
& 
\arrow[u,  pos=0.89, phantom, "\circlearrowleft"]
& 
\w{I}(-s_\gamma\theta^+_{(k,\ell)})^{((-)^\ell,(-)^\ell)}\\
\end{tikzcd}
\end{equation*}

\item[\emph{(2)}]  
For $(\lambda_1,\lambda_2) =(1-k, 1-\ell+k) \in  \Lambda_+^{(k>\ell)}$,
the differential operator $\D_+^{(k,\ell)}$ can be factored as 
\begin{equation*}
\D^{k-\ell}_1
\circ \D_{\Cay}^{(2k-\ell;\ell)} 
= \D_+^{(k,\ell)}
=
\D^{\ell}_2
\circ \D^{k}_1.
\end{equation*}
Equivalently, 
the following diagram commutes.
\begin{equation*}
\begin{tikzcd}[row sep=1cm, column sep=1cm]
&
\w{I}(-s_\gamma \theta^+_{(k,\ell)})^{((-)^\ell,(-)^\ell)} 
\arrow[dr, "\D_1^{k-\ell}"]
& &\\
& &
\w{I}(-s_\beta s_\ga\theta^+_{(k,\ell)})^{((-)^\ell,(-)^k)} 
&
 \\
\w{I}(-\theta^+_{(k,\ell)})^{(+,+)}  
\arrow[urr, "\D_+^{(k,\ell)}"]
\arrow[uur, "\D_{\Cay}^{(2k-\ell;\ell)}"] 
\arrow[r, "\D_1^{k}"']  
&
\w{I}(-s_\ga\theta^+_{(k,\ell)})^{(+,(-)^k)}
\ar[ur, "\D_2^{\ell}"']
\arrow[uu,  pos=0.5, phantom, "\circlearrowleft"] 
\arrow[uu,  pos=0.01, phantom, "\circlearrowleft"] 
&
&\\
\end{tikzcd}
\end{equation*}

\item[\emph{(3)}] 
For $(\lambda_1,\lambda_2) =(1-k+\ell, 1-\ell) \in
\Lambda_-^{(k>\ell)}(= \Lambda_+^{(k < \ell)})$,
the differential operators
$\D_+^{(k,k-\ell)}$,
$\D_-^{(k,\ell)}$, and $\D_{\Cay}^{(k-2\ell;k)}$ can be 
factored as follows.
\begin{equation*}
\D_+^{(k,k-\ell)}=\D_2^{k-\ell}\circ \D_1^k,
\quad
\D_-^{(k,\ell)} = \D^{k}_1 \circ \D^{\ell}_2
\end{equation*}
and
\begin{equation*}
\D_{\Cay}^{(k-2\ell;k)} 
=\D_+^{(k,k-\ell)}\circ \D_2^\ell
=\D^{k-\ell}_2 \circ \D_-^{(k,\ell)}
=\D_2^{k-\ell}\circ \D_1^k \circ \D_2^\ell.
\end{equation*}
Equivalently, 
the following diagram commutes.
\begin{equation*}
\begin{tikzcd}[row sep=1cm, column sep=1cm]
\w{I}(-\theta^-_{(k,\ell)})^{(+,+)}
\arrow[drr, "\D_-^{(k,\ell)}"]
\arrow[dr, "\D_2^{\ell}"']
\arrow[rrr, "\D_{\Cay}^{(k-2\ell;k)}"]
& 
\arrow[d,  pos=0.89, phantom, "\circlearrowleft"]
\arrow[rd,  pos=0.65, phantom, "\circlearrowleft"]
& 
\arrow[d,  pos=0.89, phantom, "\circlearrowleft"]
&  
\w{I}(-s_\gamma \theta^-_{(k,\ell)})^{((-)^k,(-)^k)}\\
&
\w{I}(-s_\beta\theta^-_{(k,\ell)})^{((-)^\ell,+)} 
\arrow[r, "\D_1^{k}"'] 
\arrow[rru,"\D_+^{(k,k-\ell)}"]  
&
\w{I}(-s_\ga s_\beta\theta^-_{(k,\ell)})^{((-)^\ell,(-)^k)}   
\arrow[ur, "\D_2^{k-\ell}"']
&\\
\end{tikzcd}
\end{equation*}

\item[\emph{(4)}]
For $(\lambda_1,\lambda_2)=(1-k+\ell, 1-\ell)\in \Lambda_-^{(k<\ell)}$, 
the differential operator $\D_-^{(k,\ell)}$ 
can be factored as 
\begin{equation*}
\D_2^{\ell-k}\circ
\D_{\Cay}^{(k-2\ell;k)}
= \D_-^{(k,\ell)}
=\D_1^{k} \circ \D_2^{\ell}.
\end{equation*}
Equivalently, the following diagram commutes.
\begin{equation*}
\begin{tikzcd}[row sep=1cm, column sep=1cm]
\w{I}(-\theta^-_{(k,\ell)})^{(+,+)} 
\arrow[r, "\D_2^{\ell}"]  
\arrow[ddr,"\D_{\Cay}^{(k-2\ell;k)}"']
\arrow[rrd,"\D_-^{(k,\ell)}"']
&
\w{I}(-s_\beta\theta^-_{(k,\ell)})^{((-)^\ell,+)}
\arrow[dr,"\D_1^{k}"]
\arrow[dd,  pos=0.5, phantom, "\circlearrowleft"] 
\arrow[dd,  pos=0.05, phantom, "\circlearrowleft"] 
&
&\\
& & 
\w{I}(-s_\alpha  s_\beta\theta^-_{(k,\ell)})^{((-)^\ell,(-)^k)} 
& 
\\
& 
\w{I}(-s_\gamma\theta^-_{(k,\ell)})^{((-)^k,(-)^k)} 
\arrow[ur, "\D_2^{\ell-k}"']
& &\\
\end{tikzcd}
\end{equation*}

\item[\emph{(5)}]  
For $(\lambda_1,\lambda_2) =(1-k, 1) \in  \Lambda_+^{(k=\ell)}$,
we have
\begin{equation*}
\D_{\Cay}^{(k;k)} 
= \D_+^{(k,k)}
=\D^{(k)}_2 \circ \D^{(k)}_1.
\end{equation*}
Equivalently, 
the following diagram commutes.
\begin{equation*}
\begin{tikzcd}[row sep=1cm, column sep=1cm]
&
\w{I}(-s_\ga\theta^+_{(k,k)})^{(+,(-)^k)}
 \arrow[dr,"\D_2^{(k)}"]  
 & \\
\w{I}(-\theta^+_{(k,k)})^{(+,+)}
\arrow[rr,"\D_{\Cay}^{(k;k)}=\D_+^{(k,k)}"']
\arrow[ur,"\D_1^{(k)}"]
&
 \arrow[u,  pos=0.5, phantom, "\circlearrowleft"]
 &  
\w{I}(-s_\gamma \theta^+_{(k,k)})^{((-)^k,(-)^k)} \\
\end{tikzcd}
\end{equation*}

\item[\emph{(6)}]
For $(\lambda_1,\lambda_2)=(1, 1-k)\in \Lambda_-^{(k=\ell)}$, 
we have
\begin{equation*}
\D_{\Cay}^{(-k;k)} 
= \D_-^{(k,k)}
=\D_1^{(k)}\circ \D_2^{(k)}.
\end{equation*}
Equivalently, the following diagram commutes.
\begin{equation*}
\begin{tikzcd}[row sep=1cm, column sep=1cm]
\w{I}(-\theta^-_{(k,k)})^{(+,+)} 
\arrow[rr, "\D_{\Cay}^{(-k;k)}=\D_-^{(k,k)}"]
\arrow[dr,"\D_2^{(k)}"']
&
\arrow[d,  pos=0.5, phantom, "\circlearrowleft"]
& 
\w{I}(-s_\gamma\theta^-_{(k,k)})^{((-)^k,(-)^k)}\\
&
\w{I}(-s_\beta\theta^-_{(k,k)})^{((-)^k,+)}
\arrow[ur,"D_1^{(k)}"']  
&\\
\end{tikzcd}
\end{equation*}

\end{enumerate}

\end{thm}

\begin{proof}
Apply the duality theorem (Theorem \ref{thm:duality}) to Theorem \ref{thm:factor1}.
\end{proof}

\noindent
\textbf{Acknowledgements.}
The authors are indebted to Ryosuke Nakahama, Kazuki Kannaka, and
Tatsuro Hikawa for pointing out some gaps in the arguments in an early version 
of the manuscript. They are also  grateful to Taito Tauchi, Hiroyoshi Tamori, 
and Toshiyuki Kobayashi for fruitful discussions on this paper.
The first author was partially supported by 
Grant-in-Aid for Scientific Research(C) (JP22K03362).
The second author was supported by Grant-in-Aid for 
JSPS International Research Fellows (JP24KF0075).

\vskip 0.1in




\end{document}